\documentclass[11pt,reqno]{amsart}
\usepackage{amssymb}
\usepackage{amsmath, amssymb, amsthm}
\usepackage{mathrsfs}
\usepackage{soul,color}

\newtheorem{theorem}{Theorem}[section]
\newtheorem{definition}{Definition}[section]
\newtheorem{lemma}{Lemma}[section]
\newtheorem{remark}{Remark}[section]
\newtheorem{proposition}{Proposition}[section]
\newtheorem{corollary}{Corollary}[section]
\numberwithin{equation}{section}

\newcommand{\A}{{\Phi}}
\newcommand{\B}{{\Psi}}

\newcommand{\R}{{\mathbb R}}

\newcommand{\supp}{{\mathrm {supp}}}

\begin{document}
\title[Vanishing Viscosity Limit]{On the global-in-time  inviscid  limit  of the 3D isentropic  compressible Navier-Stokes equations with degenerate viscosities and  vacuum}

\author{Yongcai Geng}
\address[Y. C. Geng]{School of Science,
         Shanghai Institute of Technology,
         Shanghai 200235, China}
\email{\tt yongcaigeng@alumni.sjtu.edu.cn}

\author{Yachun Li}
\address[Y. C. Li]{School of Mathematical Sciences,  MOE-LSC, and SHL-MAC, Shanghai Jiao Tong University, Shanghai 200240, P.R.China.} \email{\tt ycli@sjtu.edu.cn}

\author{Shengguo Zhu}
\address[S. G.  Zhu]{Mathematical Institute,
University of Oxford, Oxford,  OX2 6GG, UK.}
\email{\tt zhus@maths.ox.ac.uk}

\begin{abstract}

In the recent  paper, the global-in-time inviscid limit   of the  three-dimensional (3D)  isentropic compressible
Navier-Stokes equations  is considered.   First, when viscosity coefficients are given as a constant multiple of density's power ($(\rho^\epsilon)^\delta$ with $\delta>1$),  for regular  solutions to the corresponding Cauchy problem,  via introducing one  `` quasi-symmetric hyperbolic"--``degenerate elliptic" coupled structure to control the behavior of the velocity near the vacuum,    we establish the uniform energy estimates for  the local sound speed  in $H^3$  and $(\rho^\epsilon)^{\frac{\delta-1}{2}}$ in $H^2$  with respect to the viscosity coefficients for arbitrarily large time under some  smallness assumption on the initial  density.  Second, by  making full use of this structure's  quasi-symmetric property and the weak smooth effect on solutions,   we prove the strong convergence  of the regular solutions of the  degenerate  viscous flow  to that of  the  inviscid flow  with vacuum in $H^2$  for arbitrarily large time.   The result here applies to a class of degenerate density-dependent viscosity coefficients, is independent of the B-D relation for viscosities, and seems to be the first  on the global-in-time  inviscid limit  of smooth solutions which have large velocities and contain vacuum state for compressible flow in three space dimensions without any symmetric assumption.

\end{abstract}

\date{Nov. 05, 2019}

\maketitle

\tableofcontents

\section{Introduction}

The time evolution of the  density $\rho^\epsilon \geq 0$ and the velocity $u^\epsilon=\left((u^\epsilon)^{(1)},(u^\epsilon)^{(2)},(u^\epsilon)^{(3)}\right)^\top$ $\in \mathbb{R}^3$ of a general viscous isentropic
compressible  fluid occupying a spatial domain $\Omega\subset \mathbb{R}^3$ is governed by the following isentropic  compressible  Navier-Stokes equations (\textbf{ICNS}):
\begin{equation}
\label{eq:1.1}
\begin{cases}
\rho^\epsilon_t+\text{div}(\rho^\epsilon u^\epsilon)=0,\\[8pt]
(\rho^\epsilon u^\epsilon)_t+\text{div}(\rho^\epsilon u^\epsilon\otimes u^\epsilon)
  +\nabla
   P^\epsilon =\text{div} \mathbb{T}.
\end{cases}
\end{equation}
Here, $x=(x_1,x_2,x_3)^\top\in \Omega$, $t\geq 0$ are the space and time variables, respectively. In considering the polytropic gases, the constitutive relation is given by
\begin{equation}
\label{eq:1.2}
P^\epsilon=A(\rho^\epsilon)^{\gamma}, \quad A>0,\quad  \gamma> 1,
\end{equation}
where $A$ is  an entropy  constant and  $\gamma$ is the adiabatic exponent. $\mathbb{T}$ denotes the viscous stress tensor with the  form
\begin{equation}
\label{eq:1.3}
\begin{split}
&\mathbb{T}=2\mu(\rho^\epsilon)D(u^\epsilon)+\lambda(\rho^\epsilon)\text{div}u^\epsilon\,\mathbb{I}_3,
\end{split}
\end{equation}
 where $D(u^\epsilon)=\frac{1}{2}\big(\nabla u^\epsilon+(\nabla u^\epsilon)^\top\big)$ is the deformation tensor, $\mathbb{I}_3$ is the $3\times 3$ identity matrix,
\begin{equation}
\label{fandan}
\mu(\rho^\epsilon)=\epsilon \alpha  (\rho^\epsilon)^\delta,\quad \lambda(\rho^\epsilon)=\epsilon \beta  (\rho^\epsilon)^\delta,
\end{equation}
for some  constants $\epsilon \in (0,1]$ and $\delta\geq 0$,
 $\mu(\rho^\epsilon)$ is the shear viscosity coefficient, $\lambda(\rho^\epsilon)+\frac{2}{3}\mu(\rho^\epsilon)$ is the bulk viscosity coefficient,  $\alpha$ and $\beta$ are both constants satisfying
 \begin{equation}\label{10000}\alpha>0 \quad \text{and} \quad   2\alpha+3\beta\geq 0.
 \end{equation}


  In addition,  when $\epsilon=0$, from
\eqref{eq:1.1}, one naturally has the compressible isentropic Euler
equations for  the  inviscid flow:
\begin{equation}
\label{eq:1.1E}
\begin{cases}
\displaystyle
\rho_t+\text{div}(\rho u)=0,\\[8pt]
\displaystyle
(\rho u)_t+\text{div}(\rho u\otimes u)
  +\nabla
   P =0,
\end{cases}
\end{equation}
where $\rho$, $u$ and $P=A\rho^\gamma$ are the mass density, velocity and pressure of the inviscid fluid.

Let $\Omega=\mathbb{R}^3$ and  $\delta>1$. In this paper,  we will consider the global-in-time  asymptotic behavior as $\epsilon \rightarrow 0$  of  the smooth solution
$(\rho^\epsilon,u^\epsilon)$   to  the Cauchy problem  (\ref{eq:1.1})-(\ref{10000}) with the following  initial data and far field behavior:
\begin{align}
&(\rho^\epsilon,u^\epsilon)|_{t=0}=(\rho^\epsilon_0(x)\geq 0,  u^\epsilon_0(x)) \ \ \ \ \ \ \ \  \  \ \  \  \text{for} \ \ \   x\in \mathbb{R}^3,\label{initial}\\[4pt]
&(\rho^\epsilon,u^\epsilon)(t,x)\rightarrow  (0,0) \ \ \   \text{as}\ \ \ |x|\rightarrow \infty \ \ \ \    \ \  \text{for} \ \ \  \ t\geq 0.  \label{far}
\end{align}
Our conclusion obtained in this paper will show that,   for arbitrarily large time,  the inviscid flow (\ref{eq:1.1E}) can be regarded as viscous flow (\ref{eq:1.1}) with vanishing real physical viscosities in the sense of the regular solution (see Definition \ref{d1}) with  vacuum.
Generally,  the above far field behavior cannot be avoided when some physical requirements are imposed, such as finite total mass
 in  $\mathbb{R}^3$, because at least we need $\rho^\epsilon\rightarrow 0$ as $ |x|\rightarrow \infty$.

In the theory of gas dynamics, the \textbf{CNS} can be derived  from the Boltzmann equations through the Chapman-Enskog expansion, cf. Chapman-Cowling \cite{chap} and Li-Qin  \cite{tlt}. Under some proper physical assumptions,  the viscosity coefficients and heat conductivity coefficient $\kappa$ are not constants but functions of absolute temperature $\theta^\epsilon$ (see \cite{chap}) such as:
\begin{equation}
\label{eq:1.5g}
\begin{split}
\mu(\theta^\epsilon)=&a_1 (\theta^\epsilon)^{\frac{1}{2}}F(\theta^\epsilon),\quad  \lambda(\theta^\epsilon)=a_2 (\theta^\epsilon)^{\frac{1}{2}}F(\theta^\epsilon), \quad \kappa(\theta^\epsilon)=a_3 (\theta^\epsilon)^{\frac{1}{2}}F(\theta^\epsilon)
\end{split}
\end{equation}
for some   constants $a_i$ $(i=1,2,3)$.  Actually for the cut-off inverse power force model, if the intermolecular potential varies as $r^{-a}$ ($ r$: intermolecular distance),  then in (\ref{eq:1.5g}):
 $$F(\theta^\epsilon)=(\theta^\epsilon)^{b}\quad  \text{with}\quad   b=\frac{2}{a} \in [0,+\infty).$$
  In particular  (see \S 10 of \cite{chap}), for  ionized gas,
 $$a=1\quad \text{and} \quad b=2;$$
  for Maxwellian molecules,
  $$a=4\quad \text{and} \quad b=\frac{1}{2};$$
while for rigid elastic spherical molecules,
$$a=\infty\quad \text{and} \quad b=0.$$
According to Liu-Xin-Yang \cite{taiping}, for  isentropic and polytropic fluids, such dependence is inherited through the laws of Boyle and Gay-Lussac:
$$
P=R\rho^\epsilon \theta^\epsilon=A(\rho^\epsilon)^\gamma,  \quad \text{for \ \ constant} \quad R>0,
$$
i.e.,
$\theta^\epsilon=AR^{-1}(\rho^\epsilon)^{\gamma-1}$,  and one finds that the viscosity coefficients are functions of the density.
Actually, the  similar  assumption that  viscosity coefficients depend on the density  can be seen in a lot of fluid models, such as  Korteweg system,  shallow water equations, lake  equations  and quantum Navier-Stokes system (see \cite{bd2,  bd, BN2,   gp,  ansgar, lions,   Mar,  vy}).

In the three-dimensional space, due to the complicated mathematical structures of the system \eqref{eq:1.1} and   \eqref{eq:1.1E},   there are few  literatures on the global  existence  of solutions with vacuum and their behaviors to the corresponding Cauchy  problems.  For the inviscid flow,   by extracting a dispersive effect after some invariant transformation,  the global well-posedness of the smooth solution with small data   in some homogeneous Sobolev spaces  has been proven  by Grassin and Serre  \cite{grassin, Serre}.
For the constant viscous flow ($\delta=0$ in \eqref{fandan}),
the first  breakthrough for the well-posedness of  solutions
 with generic initial data including vacuum  is due to Lions \cite{lions},  where he established the global existence of weak solutions with finite energy  provided that $\gamma >\frac{9}{5}$ (see also Feireisl-Novotn\'{y}-Petzeltov\'{a}  \cite{ fu2, fu3,fu1} for $\gamma>\frac{3}{2}$ and the corresponding existence of the variational solutions  for the non-isentropic flow).  However, the
uniqueness problem of these weak solutions  is widely open due to their fairly low regularities.   Later, based on the uniform estimate on the upper bound of the density,  the global well-posedenss of the smooth solution with small energy  has been established by  Huang-Li-Xin \cite{HX1}. For the degenerate viscous flow ($\delta>0$ in \eqref{fandan}), under the well-known B-D relation for viscosities:
\begin{equation}\label{bd}
\lambda(\rho^\epsilon)=2(\mu'(\rho^\epsilon)\rho^\epsilon-\mu(\rho^\epsilon)),
\end{equation}
which is introduced by Bresch-Desjardins  and their collaborators in \cite{bd6,bd7, bd2,bd},
the global existences of the multi-dimensional weak solutions with finite energy for some kind of $\mu(\rho^\epsilon)$ have been given  by Li-Xin \cite{lz} and  Vasseur-Yu \cite{vayu}.
Recently, for the case  $\delta>1$ in \eqref{fandan},  the  global well-posedness of smooth solutions with small density but possibly large velocity   in some homogeneous Sobolev spaces  has been established  by Xin-Zhu \cite{zz}. Readers are referred to \cite{cons, guo, Has,  HX2, JWX,  hailiang, sz3,sz333, lions,   vassu, zz2, ty1, zyj, tyc2} for some other related  interesting results.

Based on the well-posedness theory mentioned above,  naturally  there is a very important and interesting   question that:    for arbitrarily large time,   can we regard the regular solution  of inviscid flow in  \cite{grassin, Serre}  as  those of viscous flow for $\delta\geq 0$ in \cite{HX1, zz} with vanishing real physical viscosities?
Actually, such kind of idea  can date back to     \cite{guy, H1, R1, R2, stoke}. However, due to the high degeneracy of  hydrodynamics equations near the vacuum,   existences of strong solutions  to the viscous flow  and the inviscid flow are  established in totally different frameworks (even for local-in-time or 1-D space). The   arguments   used  in  the current well-posedness theory with vacuum  of viscous flow  usually rely on the  uniform (resp. weak) ellipticity of the Lame operator (resp.  some related elliptic operators),
and  both the   a priori estimates on the solutions
 and their life spans  obtained in the related references   strictly depend on the coefficients of   real physical viscosities. For example,
  when $\delta=0$   in \cite{HX1},  one finds $T^{\epsilon}=O(\epsilon)$ and
\begin{equation*}
\begin{split}
\|\nabla^{k+2}u^\epsilon\|_{L^2(\mathbb{R}^3)}\leq C\epsilon^{-1}\big(\|\nabla^k(\rho^\epsilon u^\epsilon_t+\rho^\epsilon u^\epsilon\cdot \nabla u^\epsilon+\nabla P^\epsilon)\|_{L^2(\mathbb{R}^3)}\big)
\end{split}
\end{equation*}
for  $k=0$, $1$, $2$  and  some constant $C>0$    independent of $\epsilon$,
which tells us  that the current frameworks do not seem to work for verifying the expected  limit relation  in multi-dimensional space.
Some new observations  need to be introduced for dealing with the difficulties caused by the vacuum.

 Actually, for the double degenerate system    ($\delta>0$ in \eqref{fandan}):
   $$
\displaystyle
 \underbrace{\rho^\epsilon(u^\epsilon_t+u^\epsilon\cdot \nabla u^\epsilon)}_{Degeneracy \ of  \ time \ evolution}+\nabla P^\epsilon= \underbrace{\text{div}(2\mu(\rho^\epsilon)D(u^\epsilon)+\lambda(\rho^\epsilon)\text{div}u^\epsilon \mathbb{I}_3)}_{Degeneracy\ of \ viscosities},
$$
it is very hard to control the behavior of the velocity near the vacuum for arbitrarily large time.
Recently, via introducing the following structures:
\begin{itemize}
 \item  a  ``quasi-symmetric hyperbolic''--``singular elliptic" coupled structure with strongly singular source term  for $0<\delta<1$ in \cite{Geng2};\\
\item   a ``quasi-symmetric hyperbolic''--``elliptic"  coupled structure with weakly singular source term  for $\delta=1$ in \cite{Ding};\\
\item  a  ``quasi-symmetric hyperbolic''--``degenerate elliptic" coupled structure with \\ strongly nonlinear source term  for $\delta>1$ in \cite{Geng};
\end{itemize}
when  the initial density vanishes in some open domain or  decays to zero in the far field,
the uniform local-in-time   estimates (with respect to the viscosity coefficients) on
\begin{itemize}
\item the local sound speed $\sqrt{(P^\epsilon)'(\rho)}$ in $H^3(\mathbb{R}^3)$ for $\delta>0$ and $\gamma>1$;\\
\item $\nabla^2 \log \rho^\epsilon $ in $H^1(\mathbb{R}^2)$  for $\delta=1$ and $\gamma>1$;\\
\item   $(\rho^\epsilon)^{\frac{\delta-1}{2}}$ in $H^2(\mathbb{R}^3)$  for   $\delta>1$ and $\gamma>1$
\end{itemize}
  to the corresponding Cauchy problem of  the system (\ref{eq:1.1})   have been given,
which
lead to the strong convergences  of the   local-in-time regular solution of  the viscous flow to
that of the  inviscid flow  in  some Sobolev space. We also  refer readers  to \cite{bianbu, chen3, chen4,  Germain,  gilbarg, OG, HL,  huang, HPWWZ,KM,  Masmoudi} for some important progress on the  local/global-in-time  inviscid limit  results for 1-D, spherical symmetric case  or the multi-dimensional case away from the vacuum.
However, as far as we know,  for the multi-dimensional problem without any symmetric assumption, there are indeed few results on the global-in-time  inviscid limit with non-negative initial density,   no matter in the bounded domain or the whole space.  The result obtained in this paper is independent of the well-known B-D relations for viscosities,    and seems to be the first one on such kind of study. Moreover, our arguments   can  also apply to a class of systems with degenerate  viscosity coefficients.  We hope that  the methodology developed in this work could give us a good understanding for  some other related vacuum problems  of  the  degenerate viscous flow  in a more general framework, such as the inviscid limit problem for multi-dimensional entropy weak solutions in the whole space.

\section{Main results}
In this section, we will  state our main results and give  the outline  of  this paper.
Throughout this paper,  we adopt the following simplified notations, most of them are for the standard homogeneous or non-homogeneous  Sobolev spaces:
\begin{equation*}\begin{split}
 & \|f\|_s=\|f\|_{H^s(\mathbb{R}^3)},\quad |f|_p=\|f\|_{L^p(\mathbb{R}^3)},\quad \|f\|_{m,p}=\|f\|_{W^{m,p}(\mathbb{R}^3)},\\[10pt]
& |f|_{C^k}=\|f\|_{C^k(\mathbb{R}^3)},\quad \|f\|_{L^pL^q_t}=\| f\|_{L^p([0,t]; L^q(\mathbb{R}^3))},\\[10pt]
 & D^{k,r}=\{f\in L^1_{loc}(\mathbb{R}^3): |f|_{D^{k,r}}=|\nabla^kf|_{r}<+\infty\},\quad D^k=D^{k,2},  \\[10pt] 
& \Xi=\{f\in L^1_{loc}(\mathbb{R}^3): |\nabla f|_\infty+\|\nabla^2 f\|_2<+\infty\},\quad \\[10pt]
& \Xi^*=\{f\in L^1_{loc}(\mathbb{R}^3): |\nabla f|_\infty+\|\nabla^2 f\|_3<+\infty\},\\[10pt]
& \|f\|_{\Xi}=|\nabla f|_\infty+\|\nabla^2 f\|_2,\quad \|f\|_{\Xi^*}=|\nabla f|_\infty+\|\nabla^2 f\|_3,\\[10pt]
&  \|f(t,x)\|_{T,\Xi}=\|\nabla f(t,x)\|_{L^\infty([0,T]\times\mathbb{R}^3)}+\|\nabla^2 f(t,x)\|_{L^\infty([0,T]; H^2(\mathbb{R}^3))},\\[10pt]
&  \|f\|_{X_1 \cap X_2}=\|f\|_{X_1}+\|f\|_{X_2},\quad 
  \int_{\mathbb{R}^3}  f \text{d}x  =\int f, \\[10pt]
 & \|(f,g)\|_X=\|f\|_X+\|g\|_X, \quad   X([0,T]; Y)=X([0,T]; Y(\mathbb{R}^3)).
\end{split}
\end{equation*}
 A detailed study of homogeneous Sobolev spaces  can be found in \cite{gandi}. We also denote  $C(a)$ as the constant  $C$ only depending  on the parameter $a$.

\subsection{Global-in-time uniform energy estimates}
Let $ \widehat{u}^\epsilon=((\widehat{u}^\epsilon)^{(1)},(\widehat{u}^\epsilon)^{(2)}, (\widehat{u}^\epsilon)^{(3)})^\top$ be the unique smooth  solution in $ [0,T]\times \mathbb{R}^3$  of the following  Cauchy problem:
\begin{equation} \label{eq:approximation}
\begin{cases}
\displaystyle
\widehat{u}^\epsilon_t+ \widehat{u}^\epsilon\cdot \nabla  \widehat{u}^\epsilon=0 \quad \text{in} \quad [0,T]\times \mathbb{R}^3,  \\[10pt]
 \displaystyle
\widehat{u}^\epsilon(t=0,x)=  \widehat{u}^\epsilon_0(x) \in \Xi \quad \text{for} \quad x\in \mathbb{R}^3.
 \end{cases}
\end{equation}

We first  introduce the following definition of regular solutions with vacuum  to the  Cauchy problem  (\ref{eq:1.1})-(\ref{10000}) with (\ref{initial})-(\ref{far}).

\begin{definition}\label{d1}{\bf(Regular solution of viscous  flow)}  Let $T> 0$ be a finite constant.
A solution $(\rho^\epsilon,u^\epsilon)$ to the  Cauchy problem  (\ref{eq:1.1})-(\ref{10000}) with (\ref{initial})-(\ref{far}) is called a regular solution in $ [0,T]\times \mathbb{R}^3$ if $(\rho^\epsilon,u^\epsilon)$ satisfies this problem in the sense of distributions and:
\begin{equation*}\begin{split}
&(\textrm{D})\quad  \rho^\epsilon\geq 0, \  \Big((\rho^\epsilon)^{\frac{\delta-1}{2}},(\rho^\epsilon)^{\frac{\gamma-1}{2}}\Big)\in C([0,T]; H^{s'}_{loc})\cap L^\infty([0,T]; H^3);\\[4pt]
& (\textrm{E})\quad u^\epsilon-  \widehat{u}^\epsilon\in C([0,T]; H^{s'}_{loc})\cap L^\infty([0,T]; H^{3}),\\[4pt]
&\quad \quad \  (\rho^\epsilon)^{\frac{\delta-1}{2}}\nabla^4 u^\epsilon \in L^2([0,T]; L^2),\quad u^\epsilon(0,x)=  \widehat{u}^\epsilon_0(x)=u^\epsilon_0(x);\\[4pt]
&(\textrm{F})\quad u^\epsilon_t+u^\epsilon\cdot\nabla u^\epsilon =0\quad  \text{as } \quad  \rho^\epsilon(t,x)=0,
\end{split}
\end{equation*}
where $s'\in[2,3)$ is an arbitrary constant.
\end{definition}
\begin{remark}
The regular solution defined above  has finite energy if the initial density is compactly supported,  and also  is locally momentum conserving {\rm (see page 2 of  Hoff \cite{HP})}.
\end{remark}

For simplicity, we also denote 
\begin{equation*}\begin{split}
M_1=&\frac{2\alpha+3\beta}{2\alpha+\beta},\quad
M_2=-3\delta+1+\frac{1}{2}M_3,\\
M_3=&\frac{(\delta-1)^2}{4(2\alpha+\beta)}+\frac{4\delta^2(2\alpha+\beta)}{(\delta-1)^2}M^2_1+2M_1\delta,\\
M_4=&\frac{1}{2}\min\Big\{ \frac{3\gamma-3}{2},\frac{-M_2-1}{2},1\Big\}+M_2,\\
\end{split}
\end{equation*}
Based on the  above definition and notations, now we can   give the global-in-time uniform energy estimates with respect to $\epsilon$.

\begin{theorem}\label{tha}{\bf(Uniform energy estimates)}  Let parameters  $(\gamma,\delta, \alpha,\beta)$ satisfy
\begin{equation}\label{canshu}
\gamma>1,\quad \delta>1, \quad \alpha>0, \quad 2\alpha+3\beta\geq 0,
\end{equation}
and any one of the following  conditions $(P_1)$-$(P_4)$:
\begin{itemize}

\item[$(\rm P_1)$] $0<M_1<\frac{3}{2}-\frac{1}{\delta}$ and 
$
M_2<-1
$;\\
\item[$(\rm P_2)$] $2\alpha+3\beta=0$;\quad $(\rm P_3)$ $\delta\geq 2\gamma-1$; \quad $(\rm P_4)$ $\delta=\gamma$.
\end{itemize}
If the   initial data $( \rho^\epsilon_0, u^\epsilon_0)$ satisfy
\begin{itemize}

\item[$(\rm A_1)$] $u^\epsilon_0\in \Xi$, $\|u^\epsilon_0\|_{\Xi}\leq C$ for some constant $C>0$ that is independent of $\epsilon$,
and  there exists a constant  $\kappa>0$ that is also independent of $\epsilon$ such that,
$$
\text{Dist}\big(\text{Sp}( \nabla u^\epsilon_0(x)), \mathbb{R}_{-} \big)\geq \kappa \quad \text{for\  all}\quad  x\in \mathbb{R}^3,
$$
where $Sp(\nabla u^\epsilon_0(x))$ denotes the spectrum of the matrix $\nabla u^\epsilon_0(x)$;

\item[$(\rm A_2)$] $\rho^\epsilon_0\geq 0$,  $(\rho^\epsilon_0)^{\frac{\gamma-1}{2}}\in H^3$,    $(\rho^\epsilon_0)^{\frac{\delta-1}{2}}\in H^3$, and
$$\big\|(\rho^\epsilon_0)^{\frac{\gamma-1}{2}}\big\|_3+ \big\|(\rho^\epsilon_0)^{\frac{\delta-1}{2}}\big\|_2+\epsilon^{\frac{1}{2}}|\nabla^3 (\rho^\epsilon_0)^{\frac{\delta-1}{2}}|_2 \leq D_0$$
 for some constant  $D_0=D_0(\gamma, \delta, \alpha,\beta,  A, \kappa, \|u^\epsilon_0\|_{\Xi})>0$ that   is independent of $\epsilon$,
\end{itemize}
then for any $T>0$,  there exists a   unique  regular solution $(\rho^\epsilon, u^\epsilon)$ in $[0,T]\times \mathbb{R}^3$  to the  Cauchy problem  (\ref{eq:1.1})-(\ref{10000}) with (\ref{initial})-(\ref{far}) satisfying the following uniform estimates:
\begin{equation}\label{uniformtime}\begin{split}
\sum_{k=0}^3 (1+t)^{2(k-2.5)}\Big(\big|\nabla^k(\rho^\epsilon)^{\frac{\gamma-1}{2}}(t,\cdot)\big|^2_2+\big|\nabla^k(u^\epsilon- \widehat{u}^\epsilon)(t,\cdot)\big|^2_2\Big)&\\
+\sum_{k=0}^2(1+t)^{2(k-3)}\big|\nabla^k(\rho^\epsilon)^{\frac{\delta-1}{2}}(t,\cdot)\big|^2_2+\epsilon\big|(\rho^\epsilon)^{\frac{\delta-1}{2}}(t,\cdot)\big|^2_{D^3}\leq & C_{01}(1+t)^{-\iota},\\
\epsilon \sum_{k=0}^3\int_0^t (1+s)^{2(k-2.5)}\big|(\rho^\epsilon)^{\frac{\delta-1}{2}}\nabla^{k+1} (u^\epsilon- \widehat{u}^\epsilon)(s,\cdot)\big|_2^2
\text{d}s\leq &C_{01}
\end{split}
\end{equation}
for any $t\in [0,T]$ and some  constant $C_{01}=C_{01}(\gamma, \delta, \alpha,\beta,  A, \kappa, \rho^\epsilon_0, u^\epsilon_0)>0$ that is  independent of both  $\epsilon$ and the time $T$, where  $\iota=\iota(\gamma,\delta,\alpha,\beta)=(1-\eta^*)b_*\in (1,2)$, and 
 \begin{equation*}
\begin{split}
 \eta^*=&\min\Big\{\frac{3\gamma-3}{4(3\gamma-1)}, \frac{-M_4-1}{6\delta-M_3}, \frac{1}{10}\Big\}>0,\\
b_*=&
 \begin{cases}
\min\Big\{2, \frac{3}{2}\delta-\frac{1}{4}M_3\Big\}>1 \;\qquad\quad \ \ \      \text{if} \ \ \gamma\geq \frac{5}{3},\\[6pt]
\min\Big\{\frac{3\gamma}{2}-0.5, \frac{3}{2}\delta-\frac{1}{4}M_3\Big\}>1  \ \  \  \    \text{if} \ \  1< \gamma<\frac{5}{3}.
 \end{cases}
\end{split}
 \end{equation*}
Particularly,  when the condition $(P_3)$ holds, the smallness assumption on $(\rho^\epsilon_0)^{\frac{\delta-1}{2}}$ could be removed.

\end{theorem}

\begin{remark}\label{r1} Now we give some comments on the above theorem.  First, under  $(A_1)$, the global well-posedess  of the unique smooth solution $\widehat{u}$ to the  Cauchy problem \eqref{eq:approximation} can be found  in \cite{grassin, Serre} {\rm (see also Proposition \ref{p1} in our appendix)}.

Second,  the conditions $(A_1)$-$(A_2)$ in Theorem \ref{tha} identify a class of admissible initial data that
provides the unique solvability to the  Cauchy problem  (\ref{eq:1.1})-(\ref{10000}) with (\ref{initial})-(\ref{far}).
Such initial data contain  the following examples:
\begin{equation*}\begin{split}
\rho^\epsilon_0(x)=&\frac{\nu_1}{(1+|x|)^{2\sigma_1}}, \ \  \nu_1 g^{2\sigma_2}(x), \ \ \nu_1\exp\{-x^2\}, \ \ \frac{\nu_1 |x|}{(1+|x|)^{2\sigma_3}}; \\
u^\epsilon_0(x)=&\mathcal{A}x+\textbf{b}+\nu_2 f(x),
\end{split}
\end{equation*}
where $\nu_i$ {\rm ($i=1,2$)} are both sufficiently small constants,   $\sigma_i$ {\rm ($i=1,2,3$)} are all constants satisfying
\begin{equation*}\begin{split}
&\sigma_1>\frac{3}{2}\max\Big\{\frac{1}{\delta-1}, \frac{1}{\gamma-1}\Big\}, \quad   \sigma_2>3\max\Big\{\frac{1}{\delta-1}, \frac{1}{\gamma-1}\Big\},\\
&\sigma_3>\frac{3}{2}\max\Big\{\frac{1}{\delta-1}, \frac{1}{\gamma-1}\Big\}+\frac{1}{2};\\
\end{split}
\end{equation*}
$0\leq g(x)\in C^3_c(\mathbb{R}^3)$,  $\mathcal{A}$ is a $3\times 3$ constant matrix whose  eigenvalues are all greater than $2\kappa$,     $f=(f^1,f^2,f^3)^\top $ with $f^i \in \Xi \ (i=1,2,3)$ and $\textbf{b}\in \mathbb{R}^3$ is a constant vector.

At last, the contribution of the  above  theorem is to provide the uniform estimates \eqref{uniformtime} with respect to $\epsilon$ for arbitrarily large time, which also implies that the $L^2$ norm of the third order derivatives of the local sound speed will decay to zero as time goes to infinity.
\end{remark}

Next,  for the multi-dimensional   compressible isentropic Euler equations,  we  introduce the following definition of regular solutions to the  Cauchy problem  \eqref{eq:1.1E} with
\begin{align}
&(\rho,u)|_{t=0}=(\rho_0(x)\geq 0,  u_0(x)) \ \ \ \ \ \ \ \  \  \ \  \  \text{for} \ \ \   x\in \mathbb{R}^3,\label{winitial}\\[4pt]
&(\rho,u)(t,x)\rightarrow  (0,0) \ \ \   \text{as}\ \ \ |x|\rightarrow \infty \ \ \ \    \ \  \text{for} \ \ \  \ t\geq 0.  \label{wfar}
\end{align}

\begin{definition}\label{Eulerd1} {\bf(Regular solution of inviscid flow)}  Let $T> 0$ be a finite constant.  A solution $(\rho,u)$ to the  Cauchy problem  \eqref{eq:1.1E} with \eqref{winitial}-\eqref{wfar} is called a regular solution in $ [0,T]\times \mathbb{R}^3$ if $(\rho,u)$ satisfies this problem in the sense of distributions and:
\begin{equation*}\begin{split}
&(\textrm{A})\quad  \rho\geq 0, \quad   \rho^{\frac{\gamma-1}{2}}\in C([0,T]; H^3);\\
& (\textrm{B})\quad u- \widehat{u}\in C([0,T]; H^{3});\\
&(\textrm{C})\quad u_t+u\cdot\nabla u =0\quad  \text{as } \quad  \rho(t,x)=0,
\end{split}
\end{equation*}
where $\widehat{u}$  satisfying  $\widehat{u}(0,x)=u_0(x)$  is the unique  solution of the Cauchy problem \eqref{eq:approximation}.
\end{definition}

According to the uniform estimates \eqref{uniformtime} for the degenerate viscous flow and the Aubin-Lions Lemma, one can obtain that
\begin{corollary} \label{serre} Let $\gamma>1$.
If the following initial hypothesis hold:
\begin{itemize}
\item[$(\rm H_1)$]  $u_0\in \Xi$, and  there exists a constant  $\kappa>0$ such that ,
$$
\text{Dist}\big(\text{Sp}( \nabla u_0(x)), \mathbb{R}_{-} \big)\geq \kappa \quad \text{for\  all}\quad  x\in \mathbb{R}^3;\\
$$
\item[$(\rm H_2)$] $\rho_0\geq 0$, \ $\rho^{\frac{\gamma-1}{2}}_0\in H^3$, \   and  \ $\big\|\rho^{\frac{\gamma-1}{2}}_0\big\|_3 \leq D^*_0(\gamma,  A, \kappa, \|u_0\|_{\Xi})$
\end{itemize}
for some constant $D^*_0=D^*_0( \gamma, A,  \kappa, \|u_0\|_{\Xi})>0$,
then for any $T>0$,  there exists a   unique  regular solution $(\rho, u)$ in $[0,T]\times \mathbb{R}^3$  to the  Cauchy problem  \eqref{eq:1.1E} with (\ref{winitial})-(\ref{wfar}) satisfying
\begin{equation}\label{Euniformtime}\begin{split}
\sum_{k=0}^3 (1+t)^{2(k-2.5)}\Big(\big|\nabla^k \rho^{\frac{\gamma-1}{2}}(t,\cdot)\big|^2_2+\big|\nabla^k(u-\widehat{u})(t,\cdot)\big|^2_2\Big)\leq C_{02}(1+t)^{-\iota},
\end{split}
\end{equation}
for any $t\in [0,T]$ and some   constant $C_{02}=C_{02}(\gamma, A,  \kappa, \rho_0, u_0)>0$ that is   independent of  the time $T$.

\end{corollary}

\subsection{Global-in-time inviscid limit}

Based on the above global-in-time uniform estimates for both degenerate viscous flow \eqref{eq:1.1} and the inviscid flow \eqref{eq:1.1E}, i.e., \eqref{uniformtime} and \eqref{Euniformtime},  now we can  show that the inviscid flow (\ref{eq:1.1E}) can be regarded as viscous flow (\ref{eq:1.1}) with vanishing real physical viscosities in the sense of the regular solution  with  vacuum   for arbitrarily large time. For simplicity, we denote
  \begin{equation*}
\overline{W}^\epsilon=\left(\sqrt{\frac{4A\gamma}{(\gamma-1)^2}}\big((\rho^\epsilon)^{\frac{\gamma-1}{2}}-\rho^{\frac{\gamma-1}{2}}\big), u^\epsilon- u\right),\quad \overline{W}^\epsilon_0=\overline{W}^\epsilon(0,x).
\end{equation*}

\begin{theorem}\label{th3A} {\bf (Inviscid limit)} Let $T>0$ be an arbitrarily large time. Under the assumptions of Theorem {\rm\ref{tha}}, if we assume in addition that  $u^\epsilon_0=u_0$ is independent of $\epsilon$, and there exists one  function  $\rho_0(x)\in L^2(\mathbb{R}^3)$   such that
\begin{equation}\label{initialrelation1}
\lim_{\epsilon\rightarrow 0}\Big|(\rho^\epsilon_0)^{\frac{\gamma-1}{2}}-\rho^{\frac{\gamma-1}{2}}_0\Big|_{2}=0,
\end{equation}
then there exist functions $(\rho,u)$ defined in $[0,T]\times \mathbb{R}^3$  satisfying:
\begin{equation}\label{jkkab1}
\sum_{k=0}^3 (1+t)^{2(k-2.5)}\Big(\big|\nabla^k\rho^{\frac{\gamma-1}{2}}(t,\cdot)\big|^2_2+\big|\nabla^k(u-\widehat{u})(t,\cdot)\big|^2_2\Big)  \leq C_{03}(1+t)^{-\iota},
\end{equation}
 for any  $t\in [0,T]$ and some   constant $C_{03}=C_{03}(\gamma, \alpha,\beta,  A, \kappa, \rho_0,u_0)>0$,  and  $(\rho, u)$ is  the unique regular
solution       to the Cauchy problem  \eqref{eq:1.1E} with (\ref{winitial})-(\ref{wfar}) satisfying
$$(\rho,u)|_{t=0}=(\rho_0(x)\geq 0,  u_0(x)) \quad \text{for} \quad   x\in \mathbb{R}^3.$$
Furthermore, we still have the following convergence estimates:
\begin{equation}\label{shou1A1}
\begin{split}
|\overline{W}^\epsilon(t)|^2_2
\leq& (1+t)^{C_{04}}\Big(|\overline{W}^\epsilon_0|^2_2+C_{04}\epsilon^2\ln(1+t)\Big) \quad \text{when}\quad \iota=\frac{7}{4};\\
|\overline{W}^\epsilon(t)|^2_2
\leq& (1+t)^{C_{04}}\Big(|\overline{W}^\epsilon_0|^2_2+C_{04}\epsilon^2\big|(1+t)^{7-4\iota}-1\big|\Big)\quad \text{when} \quad \iota\neq \frac{7}{4};\\
|\overline{W}^\epsilon(t)|^2_{D^1}\leq& (1+t)^{C_{04}}\Big(|\overline{W}^\epsilon_0|^2_{D^1}+C_{04}\epsilon^2\ln(1+t)\Big) \quad \text{when}\quad \iota=\frac{5}{4};\\
|\overline{W}^\epsilon(t)|^2_{D^1}\leq& (1+t)^{C_{04}}\Big(|\overline{W}^\epsilon_0|^2_{D^1}+C_{04}\epsilon^2\big|(1+t)^{5-4\iota}-1\big|\Big)\quad \text{when} \quad \iota\neq \frac{5}{4};\\
|\overline{W}^\epsilon(t)|^2_{D^2}\leq&  \exp\big( C_{04}((1+t)^{3-2\iota}-1) \big)\Big(|\overline{W}^\epsilon_0|^2_{D^2}
+C_{04}\epsilon\Big(1+\int_0^t f(s) \ \text{\rm d}s\Big)\\
&+C_{04}|\overline{W}^\epsilon_0|^2_{D^1}\int_0^t (1+s)^{1-4\iota+C_{04}}\ \text{\rm d}s\Big) \quad \text{when}\quad \iota \in \big(1,\frac{3}{2} \big);\\
|\overline{W}^\epsilon(t)|^2_{D^2}\leq& (1+t)^{C_{04}}\Big(|\overline{W}^\epsilon_0|^2_{D^2}+C_{04}|\overline{W}^\epsilon_0|^2_{D^1}\int_0^t (1+s)^{1-4\iota+C_{04}}\ \text{\rm d}s+C_{04}\epsilon\\
&+C_{04}\epsilon\int_0^t (1+s)^{1-4\iota+C_{04}}\big|(1+s)^{5-4\iota}-1\big| \ \text{\rm d}s\Big)\quad  \text{when} \quad \iota\in \big[\frac{3}{2}, 2\big),
\end{split}
\end{equation}
for any $0\leq t \leq T$ and  some   constant $C_{04}=C_{04}(\gamma, \delta, \alpha,\beta,  A, \kappa, D_0, \|u_0\|_{\Xi})>0$, where the function  $f$ of the time  is  defined by
 \begin{equation*}
\begin{split}
f(t)=&
 \begin{cases}
(1+t)^{-4+C_{04}}\ln (1+t)  \;\qquad\quad \ \ \    \  \  \text{when} \ \ \iota=\frac{5}{4},\\[10pt]
(1+t)^{1-4\iota+C_{04}}\big|(1+t)^{5-4\iota}-1\big|  \ \  \  \    \text{when} \ \  \iota \in \big(1,\frac{5}{4}\big)\cup \big(\frac{5}{4},\frac{3}{2} \big).
 \end{cases}
\end{split}
 \end{equation*}
\end{theorem}

The above theorem implies that 

\begin{corollary}\label{th2} Let \eqref{canshu} hold, and functions $(\rho_0,u_0)$ satisfy conditions $(A_1)$-$(A_2)$ for some $\delta>1$. Suppose that $(\rho^\epsilon, u^\epsilon)$ is the regular solution to the Cauchy
 problem   (\ref{eq:1.1})-(\ref{10000}) with (\ref{initial})-(\ref{far}) obtained in Theorem \ref{tha},
and $(\rho,u)$ is the regular solution to the Cauchy problem  \eqref{eq:1.1E} with (\ref{winitial})-(\ref{wfar})  obtained  in Corollary \ref{serre}. If the initial data 
\begin{equation}\label{simplecase}
(\rho^{\epsilon}, u^{\epsilon})|_{t=0}=(\rho, u)|_{t=0}=(\rho_0,u_0),
\end{equation}
then when $\epsilon\to 0$, for arbitrarily large time $T>0$,
$\big(\rho^{\epsilon}, u^{\epsilon}\big)$
converges to $(\rho, u)$ in the following sense
$$
\lim_{\epsilon \rightarrow 0}\sup\limits_{0\leq t\leq T}\Big(\big\|(\rho^\epsilon)^{\frac{\gamma-1}{2}}(t,\cdot)-\rho^{\frac{\gamma-1}{2}}(t,\cdot)
\big\|_{s'}+\big\| u^\epsilon(t,\cdot)-u(t,\cdot)
\big\|_{s'}\Big)=0,
$$
for any $s'\in [0,3)$. Moreover, we still have
\begin{align}
\sup\limits_{0\leq t\leq T}\Big(\big\|(\rho^\epsilon)^{\frac{\gamma-1}{2}}(t,\cdot)-\rho^{\frac{\gamma-1}{2}}(t,\cdot)
\big\|_{1}+\big\| u^\epsilon(t,\cdot)-u(t,\cdot)
\big\|_{1}\Big)\leq &C^*\epsilon,\label{qiyu1}\\
\sup\limits_{0\leq t\leq T}\Big(\big|(\rho^\epsilon)^{\frac{\gamma-1}{2}}(t,\cdot)-\rho^{\frac{\gamma-1}{2}}(t,\cdot)
\big|_{D^2}+\big| u^\epsilon(t,\cdot)-u(t,\cdot)
\big|_{D^2}\Big)\leq & C^* \epsilon^{\frac{1}{2}},\label{qiyu2}
\end{align}
 for some   constant $C^*=C^*(\gamma, \delta, \alpha,\beta,  A, \kappa, D_0, \|u_0\|_{\Xi},T)>0$.

\end{corollary}


\begin{remark}
According to Theorem \ref{th3A} and the proof for  Corollary 2.1 {\rm (see Section 5)},  we know that    if $(\rho_0,u_0)$ satisfy the  assumptions $(H_1)$-$(H_2)$ in Corollary 2.1,
then  the corresponding Cauchy problem \eqref{eq:1.1E} with (\ref{winitial})-(\ref{wfar})   can be regarded as a  limit problem of the Cauchy problem (\ref{eq:1.1})-(\ref{10000}) with (\ref{initial})-(\ref{far})  as $\epsilon \rightarrow 0$ for some proper $\delta>1$.
\end{remark}

\begin{remark}\label{zhunbei1}
We need to point out that Theorems \ref{tha}-\ref{th3A} still hold when viscosities $\mu$ and $\lambda$ are in the general form:
\begin{equation}\label{nianxing}\mu(\rho^\epsilon)= \epsilon (\rho^\epsilon)^\delta \alpha(\rho^\epsilon),\quad \lambda(\rho^\epsilon)= \epsilon(\rho^\epsilon)^\delta \beta(\rho^\epsilon),
\end{equation}
where
$\alpha(\rho^\epsilon)$ and $\beta(\rho^\epsilon)$ are functions of $\rho^\epsilon$, satisfying
 \begin{equation}\label{zhunbei2}
  \big(\alpha(\rho^\epsilon),\beta(\rho^\epsilon)\big)\in C^4(\mathbb{R}^+),\;\;  \alpha(\rho^\epsilon)\geq C>0, \ \ \text{and} \ \  2\alpha(\rho^\epsilon)+3\beta(\rho^\epsilon)\geq 0,
 \end{equation}
 for some constant $C>0$.
For    function pairs $(g_1(\rho^\epsilon),g_2(\rho^\epsilon))$ with  the form $\sim (\rho^\epsilon)^p$, it is obvious that they do not belong to $C^4(\mathbb{R}^+)$ when $p<4$ in the presence of vacuum. However, Theorems \ref{tha}-\ref{th3A} still hold if the initial data satisfy the  additional initial assumptions:
\begin{equation}\label{yanshen}
\begin{cases}
\alpha(\rho^\epsilon)=P_{k_1}(\rho^\epsilon)+1,\quad \beta(\rho^\epsilon)=P_{k_2}(\rho^\epsilon),\\[10pt]
P_{k_1}(\rho^\epsilon_0)\in H^3,\quad P_{k_2}(\rho^\epsilon_0)\in H^3,\\[10pt]
\sum_{i=1}^2  \Big(\big\|P_{k_i}(\rho^\epsilon_0)\big\|_2+\epsilon^{\frac{1}{2}}|\nabla^3P_{k_i}(\rho^\epsilon_0)|_2\Big) \leq D_0,
\end{cases}
\end{equation}
where $P_{k_i}(\rho^\epsilon)$ $(i=1,2,k_i > 0)$ is a $k_i$-th  degree polynomial of $\rho^\epsilon$ with vanishing constant term, and the minimum power of density among all the terms of  $P_{k_i}(\rho^\epsilon)$ $(i=1,2,k_i \geq 1)$   should  be greater than or  equal to  $\frac{\delta-1}{2}$.

\end{remark}

\begin{remark}If we consider more regular initial data,  the convergence rates  (such as in $L^\infty$)  shown in Theorem \ref{th2} and Corollary \ref{th3A}  can be slightly  improved.
Actually,  if  the   initial data $( \rho^\epsilon_0, u^\epsilon_0)$ satisfy
\begin{itemize}

\item[$(\rm A^*_1)$] $u^\epsilon_0\in \Xi^*$, $\|u^\epsilon_0\|_{\Xi^*}\leq C$ for some constant $C>0$ that is independent of $\epsilon$,
and  there exists a constant  $\kappa>0$ that is also independent of $\epsilon$ such that ,
$$
\text{Dist}\big(\text{Sp}( \nabla u^\epsilon_0(x)), \mathbb{R}_{-} \big)\geq \kappa \quad \text{for\  all}\quad  x\in \mathbb{R}^3;\\
$$

\item[$(\rm A^*_2)$] $\rho^\epsilon_0\geq 0$,   and
$$\big\|(\rho^\epsilon_0)^{\frac{\gamma-1}{2}}\big\|_4+ \big\|(\rho^\epsilon_0)^{\frac{\delta-1}{2}}\big\|_3+\epsilon^{\frac{1}{2}}|\nabla^4 (\rho^\epsilon_0)^{\frac{\delta-1}{2}}|_2 \leq D_0$$
 for some constant  $D_0=D_0(\gamma, \delta, \alpha,\beta,  A, \kappa, \|u^\epsilon_0\|_{\Xi^*})>0$ that   is independent of $\epsilon$,
\end{itemize}
then under proper assumptions,  the convergence rate $\epsilon^{\frac{1}{2}}$  in  $\eqref{shou1A1}_5$-$\eqref{shou1A1}_6$ and \eqref{qiyu2}  can be improved to $\epsilon$. The details can be seen in Section 7. Moreover, it is worth pointing out that under the current frame work, the best convergence rate in $H^s$ {\rm ($0\leq s \leq 2$)} that we can obtain is $\epsilon$, no matter how small  or  smooth the initial data is.
\end{remark}

\subsection{Applications to other degenerate systems}

The above conclusions obtained in  Theorems \ref{tha}-\ref{th3A}  are still available for some other models such as the ones shown in the following two theorems.
\begin{theorem}\label{thglobal2}
Let the viscous stress tensor $\mathbb{T}$ in  (\ref{eq:1.1}) be  given by
$$
\mathbb{T}=\epsilon (\rho^\epsilon)^\delta(2\alpha  \nabla u^\epsilon+\beta  \text{div}u^\epsilon\mathbb{I}_3).
$$
Then the same conclusions obtained    in Theorems \ref{tha}-\ref{th3A} still hold.

\end{theorem}

\begin{theorem}\label{thglobal3}
Let (\ref{canshu}) hold. If the viscous term $\text{div}\mathbb{T}$ is replaced  by
$$
\epsilon \alpha (\rho^\epsilon)^\delta\triangle u^\epsilon,
$$
then under  the  initial assumptions $(A_1)$-$(A_2)$, the same conclusions obtained     in Theorems \ref{tha}-\ref{th3A} still hold.

\end{theorem}

\subsection{Outline}

The rest of this paper can be organized as follows.   In Section 3, first we reformulate the highly degenerate  equations  (\ref{eq:1.1}) into a  trackable system  (see (\ref{li47-1})). Then we introduce the $(t,\epsilon)$-dependent energy space that will be used for establishing the desired  uniform estimates with respect to $\epsilon$.   At last, we   show  the main    strategies  for  the proof   in this paper. Section 4 is devoted to  establishing  the global-in-time   uniform energy  estimates with respect to $\epsilon $ shown in Theorem \ref{tha},  when the initial density is compactly supported. The key point of this step is to derive one proper ODE inequality \eqref{Aeq:2.14}  for our $(t,\epsilon)$-dependent energy space $Z(t)$.  In Section 5,  we first give    the proof for Theorem \ref{tha} for general initial density without the compactly supported assumption. Then we prove  the  global well-posedness of regular solutions   for the three-dimensional compressible Euler equations, i.e., the proof for  Corollary \ref{serre}. In Section 6,   based on the uniform  energy estimates obtained in Sections 4-5,    we establish the vanishing viscosity limit  of the regular solutions of the  degenerate  viscous flow  to that of  the  inviscid flow stated in Theorem \ref{th3A} and Corollary \ref{th2}. Section 7 is devoted to showing how  to improve the convergence rate for the limit process from the viscous flow to the inviscid flow if we consider more regular initial data.
Finally, we also give two  appendixes for  showing  some known results on the   global well-posedness of multi-dimensional systems and some useful lemmas. Furthermore, it is worthing pointing out  that our framework  in this paper is applicable to other physical dimensions, say 1 and 2, with some minor modifications.

\section{Reformulation and main   strategies}

In this section, first we reformulate the highly degenerate  equations  (\ref{eq:1.1}) into a  trackable system  (see (\ref{li47-1})). Then we introduce the $(t,\epsilon)$-dependent energy space that will be used for establishing the desired  uniform estimates with respect to $\epsilon$. Finally, we show  main   strategies  of our following  proof.

\subsection{Reformulation}
Assume that $\epsilon\in(0,1]$. Let $ \widehat{u}$ be the unique  smooth solution to  (\ref{eq:approximation}) obtained in Proposition \ref{p1}.
Under the assumption of Theorem \ref{tha},  according to Proposition \ref{thglobal}, one has that  for arbitrarily large time  $T>0$,  there exists a   unique  regular solution $(\rho^\epsilon, u^\epsilon)$ in $[0,T]\times \mathbb{R}^3$  to the  Cauchy
 problem   (\ref{eq:1.1})-(\ref{10000}) with (\ref{initial})-(\ref{far}).

 In terms of  the new  variables
\begin{equation}\label{bianhuan}
(\varphi^\epsilon, W^\epsilon=(\phi^\epsilon, v^\epsilon=u^\epsilon- \widehat{u}^\epsilon))=\left((\rho^\epsilon)^{\frac{\delta-1}{2}},\sqrt{\frac{4A\gamma}{(\gamma-1)^2}}(\rho^\epsilon)^{\frac{\gamma-1}{2}}, u^\epsilon- \widehat{u}^\epsilon\right)
\end{equation}
with $v^\epsilon=((v^\epsilon)^{(1)},(v^\epsilon)^{(2)},(v^\epsilon)^{(3)})^\top$,
the Cauchy problem  (\ref{eq:1.1})-(\ref{10000}) with (\ref{initial})-(\ref{far}) can be reformulated  into
\begin{equation}\label{li47-1}
\begin{cases}
\displaystyle
\varphi^\epsilon_t+ v^\epsilon \cdot\nabla\varphi^\epsilon+\frac{\delta-1}{2}\varphi^\epsilon\text{div} v^\epsilon=-\widehat{u}^\epsilon \cdot\nabla\varphi^\epsilon-\frac{\delta-1}{2}\varphi^\epsilon\text{div} \widehat{u}^\epsilon,\\[10pt]
\displaystyle
W^\epsilon_t+\sum_{j=1}^3A_j(W^\epsilon) \partial_j W^\epsilon+\epsilon (\varphi^\epsilon)^2\mathbb{{L}}(v^\epsilon)=\epsilon \mathbb{{H}}(\varphi^\epsilon)  \cdot \mathbb{{Q}}(v^\epsilon+ \widehat{u}^\epsilon)+G^*(W^\epsilon, \varphi^\epsilon,  \widehat{u}^\epsilon),\\[10pt]
(\varphi^\epsilon,W^\epsilon)|_{t=0}=(\varphi^\epsilon_0,W^\epsilon_0)=(\varphi^\epsilon_0,\phi^\epsilon_0,0),\quad x\in \mathbb{R}^3,\\[10pt]
(\varphi^\epsilon,W^\epsilon)=(\varphi^\epsilon, \phi^\epsilon,  v^\epsilon)\rightarrow (0,0,0) \quad\quad   \text{as}\quad \quad  |x|\rightarrow \infty \quad \text{for} \quad  t\geq 0,
 \end{cases}
\end{equation}
where
\begin{equation} \label{li47-2}
\begin{split}
\displaystyle
A_j(W^\epsilon)=&\left(\begin{array}{cc}
(v^\epsilon)^{(j)}&\frac{\gamma-1}{2}\phi^\epsilon e_j\\[10pt]
\frac{\gamma-1}{2}\phi^\epsilon e_j^\top &(v^\epsilon)^{(j)}\mathbb{I}_3
\end{array}
\right),\quad j=1,2,3,\quad \mathbb{{L}}(v^\epsilon)=\left(\begin{array}{c}0\\
Lv^\epsilon\\
\end{array}\right),\\[10pt]
\displaystyle
 \mathbb{{H}}(\varphi)=&\left(\begin{array}{c}0\\
\nabla (\varphi^\epsilon)^2\\
\end{array}\right), \quad
 \mathbb{{Q}}(u^\epsilon)=\left(\begin{array}{cc}
0 & 0\\
0 &Q(u^\epsilon)
\end{array}\right), \quad
 Q(u^\epsilon)=\frac{\delta}{\delta-1}\mathbb{S}(u^\epsilon),\\[10pt]
G^*(W^\epsilon, \varphi^\epsilon,  \widehat{u}^\epsilon)=&-B(\nabla  \widehat{u}^\epsilon,W^\epsilon)-\sum_{j=1}^3  (\widehat{u}^\epsilon)^{(j)}\partial_j W^\epsilon- \epsilon D( (\varphi^\epsilon)^2,\nabla^2 \widehat{u}^\epsilon),\\[10pt]
 B(\nabla \widehat{u}^\epsilon,W^\epsilon)=&\left(\begin{array}{c}
\frac{\gamma-1}{2}\phi^\epsilon \text{div} \widehat{u}^\epsilon \\[10pt]
(v^\epsilon\cdot\nabla)  \widehat{u}^\epsilon\end{array}
\right),\quad
 D(\varphi^2,\nabla^2  \widehat{u}^\epsilon)=\left(\begin{array}{c}0\\[10pt]
   (\varphi^\epsilon)^2 L  \widehat{u}^\epsilon\\
\end{array}\right),\\[6pt]
Lu =&-\alpha\triangle u-(\alpha+\beta)\nabla \mathtt{div}u,\quad \mathbb{S}(u)=\alpha(\nabla u+(\nabla u)^\top)+\beta\mathtt{div}u\mathbb{I}_3,\\[10pt]
\end{split}
\end{equation}
and  $e_j=(\delta_{1j},\delta_{2j},\delta_{3j})$ $(j=1,2,3)$ is the Kronecker symbol satisfying $\delta_{ij}=1$, when $ i=j$ and $\delta_{ij}=0$, otherwise.


\subsection{$(t,\epsilon)$-dependent energy space}
In order to establish some  global-in-time uniform  energy  estimates with respect to $\epsilon$ and the lower bound of the initial density  based on the above reformulated  structure,  one of the key steps is  to  introduce one proper energy space.  Motivated by the one used for   the local-in-time  inviscid limit problem  in   \cite{Geng} and the other one used for  the global-in-time well-posedenss  of the regular solution  in  \cite{zz},  now we denote
\begin{equation}\label{E:gz32}\begin{split}
Y_k(t)=&|\nabla^k W^\epsilon(t)|_2,\quad Y^2(t)=\sum^3_{k=0}(1+t)^{2\gamma_k}Y^2_k(t),\\
U_k(t)=&
\Theta(k)|\nabla^k\varphi^\epsilon(t)|_2,\quad U^2(t)=\sum^3_{k=0}(1+t)^{2\delta_k}U^2_k(t),\\
 \gamma_k=& k-n,\ \ \delta_k=k-m,\\
\end{split}\end{equation}
where the numbers  $n$ and $m$ will be determined later and
$$
\Theta(k)=\begin{cases}
1,\quad \text{for}\ k=1,2,\\[10pt]
\epsilon^{\frac{1}{2}},\quad \text{for}\ k=3.
\end{cases}
$$
Then the   new $(t,\epsilon)$-dependent energy space desired in our following proof   can be given by
$$
Z^2(t)=Y^2(t)+U^2(t).
$$
For simplicity, we denote also
$$
Z(0)=Z_0,\quad Y(0)=Y_0,\quad U(0)=U_0.
$$

\subsection{Main   strategies} Now we show the  main   strategies of our proof.  First, we make an analysis on the mathematical structure of the new problem \eqref{li47-1}. On the one hand, for the transport equation  $\eqref{li47-1}_1$ of  $\varphi^\epsilon$, considering its highest order estimate  $|\nabla^3 \varphi^\epsilon|_2$, we need to deal with the following integration
\begin{equation}\label{wuzhiji}
\int \widehat{u}^\epsilon \cdot\nabla^4 \varphi^\epsilon \cdot \nabla^3 \varphi^\epsilon.
\end{equation}
A standard idea is to  transfer one order spatial derivative from  $\frac{1}{2}\nabla | \nabla^3 \varphi^\epsilon|^2$ to $\widehat{u}^\epsilon$  via  the integration by parts. However, for our case, 
$$|\widehat{u}^\epsilon|\rightarrow \infty\quad \text{as} \quad |x|\rightarrow \infty.$$ We are not sure whether  the following fact holds,
$$
\lim_{R\rightarrow \infty} \int_{\partial B_R} | \nabla^3 \varphi^\epsilon|^2 \widehat{u}^\epsilon\cdot \nu_R \text{d}S=0,$$
where  $B_{R}$ is  the ball centered at the origin with radius $R>0$, and  $\nu_R$ is the outward pointing unit normal to $\partial B_R$. In order to deal with this difficulty, one proper method is to first consider the compactly supported initial density:
$$
\supp_x \rho_0 \subset B_{R}.
$$
Then for any finite time $T>0$, from the continuity equation $\eqref{eq:1.1}_1$, we know that the corresponding density  $\rho^{\epsilon,R}(t,x)$ is still compactly supported for all $t\in [0,T]$, which means that the integration in \eqref{wuzhiji} can be estimated as follows:
$$
\int \widehat{u}^\epsilon \cdot\nabla^4 \varphi^\epsilon \cdot \nabla^3 \varphi^\epsilon \leq C|\nabla\widehat{u}^\epsilon |_\infty |\nabla^3 \varphi^\epsilon|^2_2,
$$
where, obviously, the constant $C$ is independent of $(\epsilon, R)$. The corresponding estimates for the general density without compactly supported assumption can be obtained via taking limit $R\rightarrow \infty$ in  some initial approximation process.

On the other hand, it should be pointed out that
not all the first order terms in the system $\eqref{li47-1}_2$ have been written into the symmetric structure
$$
W^\epsilon_t+\sum_{j=1}^3A_j(W^\epsilon) \partial_j W^\epsilon.
$$
This  is why we only call the system $\eqref{li47-1}_2$ "quasi-symmetric"  structure rather than a "symmetric " one. Specially,  treatments are needed for  the degenerate  source terms 
$$\mathbb{{H}}(\varphi^\epsilon)  \cdot \mathbb{{Q}}(v^\epsilon+ \widehat{u}^\epsilon)\quad  \text{and} \quad G^*(W^\epsilon, \varphi^\epsilon,  \widehat{u}^\epsilon).$$

\section{Uniform energy estimates with compactly supported initial density}\label{S:3}

This section will be devoted to  establishing  the global-in-time   uniform energy  estimates with respect to $\epsilon $ shown in Theorem \ref{tha} when the initial density is compactly supported, which can be shown in the following theorem:
\begin{theorem}\label{tha-compact}Let \eqref{canshu}
and any one of the   conditions $(P_1)$-$(P_4)$ (in Theorem \ref{tha})  hold.
If the   initial data $( \rho^\epsilon_0, u^\epsilon_0)$ satisfy the initial assumptions $(A_1)$-$(A_2)$ (in Theorem \ref{tha}), and \begin{itemize}
\item[$(\rm A_3)$] $\rho_0$ is  compactly supported: $\supp_x \rho_0 \subset B_{R} $;
\end{itemize}
then  for any $T>0$,    there exists a   unique  regular solution $(\rho^\epsilon, u^\epsilon)$ in $[0,T]\times \mathbb{R}^3$  to the  Cauchy problem  (\ref{eq:1.1})-(\ref{10000}) with (\ref{initial})-(\ref{far}) satisfying the  uniform estimates \eqref{uniformtime}
for any $t\in [0,T]$.
Particularly,  when condition $(P_3)$ holds, the smallness assumption on $(\rho^\epsilon_0)^{\frac{\delta-1}{2}}$ could be removed.

\end{theorem}

The proof of the above theorem will be given in the following 2 sections.    Let $ \widehat{u}$ be the unique  smooth solution to  (\ref{eq:approximation}) obtained in Proposition \ref{p1}.   First, under the assumption of Theorem \ref{tha-compact},  according to Proposition \ref{thglobal}, one has that  for arbitrarily large time  $T>0$,  there exists a   unique  regular solution $(\rho^\epsilon, u^\epsilon)$ in $[0,T]\times \mathbb{R}^3$  to the  Cauchy
 problem   (\ref{eq:1.1})-(\ref{10000}) with (\ref{initial})-(\ref{far}). Then according to Section 3.1, one gets that
$$
(\varphi^\epsilon, W^\epsilon=(\phi^\epsilon, v^\epsilon=u^\epsilon- \widehat{u}^\epsilon))=\left((\rho^\epsilon)^{\frac{\delta-1}{2}},\sqrt{\frac{4A\gamma}{(\gamma-1)^2}}(\rho^\epsilon)^{\frac{\gamma-1}{2}}, u^\epsilon- \widehat{u}^\epsilon\right)
$$
 is the unique  classical solution  in $[0,T]\times \mathbb{R}^3$  to the Cauchy problem    \eqref{li47-1}.

 In the rest of this section, for simplicity, we denote $(\varphi^\epsilon_0,W^\epsilon_0)=(\varphi^\epsilon_0,\phi^\epsilon_0,0)$ as $(\varphi_0,W_0)=(\varphi_0,\phi_0,0)$  and $(\widehat{u}^\epsilon,\varphi^\epsilon, W^\epsilon=(\phi^\epsilon, v^\epsilon=u^\epsilon- \widehat{u}))$ as $(\widehat{u}, \varphi, W=(\phi, v=u- \widehat{u}))$.

\subsection{Uniform energy estimates under the condition $(P_1)$}

Herinafter, $C\geq 1$ will denote a generic constant depending only on fixed constants
$(\alpha,\beta,\delta,A,\gamma,\kappa)$ but independent of $(\varphi_0,W_0)$, which may be different from line to line, and $C_0>0$ denotes a generic constant depending on $(C,\varphi_0,W_0)$. Specially, $C(l)$ or $C_0(l)$
denotes a generic positive constant depending on $(C,l)$ (or $(C,\varphi_0,W_0,l)$).

We first give the following lemma according to the classical Sobolev imbedding theorem.
\begin{lemma}\begin{equation}
\label{L:2.1}\begin{split}
|W(t)|_\infty\leq& C(1+t)^{\frac{2n-3}{2}}Y(t),\quad |\nabla W(t)|_\infty\leq C(1+t)^{\frac{2n-5}{2}}Y(t),\\
|\varphi(t)|_\infty\leq& C(1+t)^{\frac{2m-3}{2}}U(t),\quad |\nabla \varphi(t)|_\infty\leq C\epsilon^{-\frac{1}{4}}(1+t)^{\frac{2m-5}{2}}U(t).\\
\end{split}\end{equation}
\end{lemma}

\subsubsection{Energy estimates on $W$} First, applying $\nabla^k$  to $\eqref{li47-1}_2$,
multiplying by $\nabla^k W$ and integrating over $\R^3$, one can get
\begin{equation}\label{E:gz33}\begin{split}
&\frac{1}{2}\frac{d}{dt}|\nabla^k W|_2^2+\epsilon\int\big(\alpha\varphi^2|\nabla^{k+1} v|^2
+(\alpha+\beta)\varphi^2|\text{div}\nabla^k v|^2\big)\\
=& R_k(W)+S_k(W,\widehat{u})+L_k(W,\varphi,\widehat{u})+Q_k(W,\varphi,\widehat{u}),
\end{split}\end{equation}
where
\begin{equation}\label{E:gz34}\begin{split}
R_k(W)=&-\int\nabla^kW\cdot\Big(\nabla^k\big(\sum_{j=1}^3A_j(W)\partial_j W\big)-\sum_{j=1}^3A_j(W)
\partial_j\nabla^k W\Big)\\
&+\frac{1}{2}\int \sum_{j=1}^3\nabla^k W\cdot\partial_jA_j(W)\nabla^k W,\\
S_k(W,\widehat{u})=&-\int \nabla^k W\cdot\nabla^k B(\nabla\widehat{u},W)
+\frac{1}{2}\int\sum_{j=1}^3\partial_j\widehat{u}^{(j)}\nabla^k W\cdot\nabla^k W\\
&-\int\nabla^k W\cdot\Big(\nabla^k\big(\sum_{j=1}^3\widehat{u}^{(j)}\partial_j W\big)-
\sum_{j=1}^3\widehat{u}^{(j)}\partial_j\nabla^k W\Big),\\
L_k(W,\varphi,\widehat{u})=&-\epsilon\int\Big(\nabla\varphi^2\cdot\mathbb{S}(\nabla^k v)-
\big(\nabla^k(\varphi^2L v)-\varphi^2L\nabla^k v\big)\Big)\cdot\nabla^k v\\
&+\epsilon\int\nabla^k(\varphi^2L\widehat{u})\cdot\nabla^k v\doteq L_k^1+L_k^2+L^3_k,\\
Q_k(W,\varphi,\widehat{u})=&\epsilon\int\Big(\nabla\varphi^2\cdot Q(\nabla^k v)+\big(\nabla^k(\nabla\varphi^2\cdot Q(v))
-\nabla\varphi^2\cdot Q(\nabla^k v)\big)\Big)\cdot\nabla^k v\\
&+\epsilon\int \nabla^k(\nabla\varphi^2\cdot Q(\widehat{u}))\cdot\nabla^kv\doteq Q_k^1+Q_k^2+Q_k^3.\\
\\
\end{split}\end{equation}

The right hand side of  \eqref{E:gz33} can be estimated in the next lemmas.
\begin{lemma} [\textbf{Estimates on  $R_k$ and $S_k$}]\label{L:3.2}
\begin{equation}\label{E:gz36}\begin{split}
\big|R_k(W)(t,\cdot)\big|\leq& C|\nabla W|_\infty Y_k^2,\\
\frac{k+r}{1+t}Y^2_k+S_k(W,\widehat{u})(t,\cdot)\leq& C_0Y_kZ(1+t)^{-\gamma_k-2},
\end{split}\end{equation}
for $k=0,1,2,3$,  where the constant $r$ is given by:
\begin{equation}\label{E:gz37}\begin{split}
r=-\frac{1}{2}\ \ \text{if} \ \ \gamma\geq\frac{5}{3},\ \ \text{or} \ \ \frac{3}{2}\gamma-3,\ \ \text{if} \ \ 1<\gamma<\frac{5}{3}.
\end{split}\end{equation}
\end{lemma}
\begin{proof}
{\bf Step 1:} Estimates of $R_k$. Noticing that $R_k$ is a sum of terms as
$$
\nabla^k W\cdot\nabla^l W\cdot\nabla^{k+1-l}W\ \ \text{for}\ \ 1\leq l\leq k,
$$
then $\eqref{E:gz36}_1$ is obvious when $k=0,1.$

For $k\neq 0,1,$ one can apply Lemma \ref{gag112} to $\nabla W$ to get
\begin{equation}\label{E:gz38}\begin{split}
|\nabla^j W|_{p_j}\leq C|\nabla W|^{1-2/p_j}_\infty|\nabla^k W|_2^{2/p_j}, \quad p_j=2\frac{k-1}{j-1}.
\end{split}\end{equation}
If $l\neq k,$ and $l\neq 1$, since $1/p_l+1/p_{k-l+1}=\frac{1}{2}$, then H$\ddot{\text{o}}$lder's inequality implies
$$
\int |\nabla^k W\nabla^l W\nabla^{k+1-l} W|\leq |\nabla^k W|_2|\nabla^l W|_{p_l}|\nabla^{k+1-l}W|_{p_{k-l+1}}\leq
C|\nabla W|_\infty|\nabla^k W|_2^2.
$$
The other cases could be handled similarly. Thus $\eqref{E:gz36}_1$ is proved.

{\bf Step 2:} Estimates on $S_k$. The integral of $S_k(W, \widehat{u})$ in $\eqref{E:gz34}_2$
can be rewritten as
\begin{equation}\label{E:gz39}\begin{split}
s_k(W,\widehat{u})=&-\nabla^k W\cdot B(\nabla\widehat{u},\nabla^k W)+\frac{1}{2}\sum_{j=1}^3\partial_j\widehat{u}^{(j)}\nabla^k W\cdot
\nabla^k W\\
&-\nabla^k W\cdot\big(\nabla^k(B(\nabla\widehat{u},W))-B(\nabla\widehat{u},\nabla^k W)\big)\\
&-\nabla^k W\cdot\big(\nabla^k(\sum_{j=1}^3\widehat{u}^{(j)}\partial_jW)-\sum_{j=1}^3\widehat{u}^{(j)}\partial_j\nabla^k W\big)
\doteq s_k^1+s_k^2,\\
\end{split}\end{equation}
where $s_k^1$ is a sum of terms with a derivatives of order one of $\widehat{u},$ and $s_2^k$ is a sum of terms with a derivative of order at least two for $\widehat{u}$.

{\bf Step 2.1:} Estimates on $S_k^1=\int s_k^1$. Let $\nabla^k=\partial_{\beta_1\beta_2\cdots \beta_i\cdots \beta_k}$ with $\beta_i=1,2,3.$ Decompose $S_k^1$ as:
\begin{equation}\label{E:gz40}\begin{split}
S_k^1=&\int\Big(-\nabla^k W\cdot B(\nabla\widehat{u},\nabla^k W)+\frac{1}{2}\sum_{j=1}^3\partial_j\widehat{u}^{(j)}\nabla^k W\cdot\nabla^k W\Big)
\\
&-\int\partial_{\beta_1...\beta_k} W\cdot\sum_{i=1}^k\sum_{j=1}^3\partial_{\beta_i}\widehat{u}^{(j)}\partial_j\partial_{\beta_1\cdots \beta_{i-1}\beta_{i+1}\cdots \beta_k}W
\doteq I_1+I_2+I_3,
\end{split}\end{equation}
where, form Proposition \ref{p1}, $I_1$-$I_3$ are given by
\begin{equation}\label{E:gz41}\begin{split}
I_1=&-\frac{3(\gamma-1)}{2(1+t)}\int\nabla^k\phi\cdot\nabla^k\phi
-\frac{1}{1+t}\int\nabla^k v\cdot\nabla^k v+G_1,\\
I_2=&\frac{3}{2}\frac{1}{1+t}Y_k^2+G_2,\quad I_3=-\frac{k}{1+t}Y_k^2+G_3,\\
 |G_j|\leq&\frac{C_0}{(1+t)^2}Y^2_k,\
\ j=1,2,3.
\end{split}\end{equation}
Therefore,
\begin{equation}\label{E:gz42}\begin{split}
S_k^1(W,\widehat{u})(t,\cdot)\leq& \frac{C_0}{(1+t)^2}Y_k^2-\frac{A_k}{1+t}\int\nabla^kv\cdot\nabla^k v
-\frac{B_k}{1+t}\int\nabla^k\phi\cdot\nabla^k \phi\\
\leq&C_0 Y_kZ(1+t)^{-\gamma_k-2}-\frac{k+r}{1+t}Y^2_k,
\end{split}\end{equation}
where
$$A_k=k-\frac{1}{2},\quad B_k=\frac{3\gamma}{2}-3+k,\quad r=\min(A_k,B_k)-k.
$$
{\bf Step 2.2:} Estimates of $S_k^2=\int s_k^2$. $s_k^2$ is a sum of the terms as
$$
E_1(W)=\nabla^k W\cdot\nabla^l\widehat{u}\cdot\nabla^{k+1-l} W\ \ \text{for}\ \ 2\leq k\leq 3,\ \ \text{and}\ \ 2\leq l\leq k;
$$
$$
E_2(W)=\nabla^k W\cdot\nabla^{l+1}\widehat{u}\cdot\nabla^{k-l} W\ \ \text{for}\ \ 1\leq k\leq 3,\ \ \text{and}\ \ 1\leq l\leq k.
$$
Hence,
\begin{equation}\label{E:gz43}\begin{split}
S_1^2(W)\leq& C|W|_\infty|\nabla^2\widehat{u}|_2|\nabla W|_2\leq C_0Y_1Z(1+t)^{-2-\gamma_1},\\
S_2^2(W)\leq& C\big(|\nabla^2\widehat{u}|_\infty|\nabla W|_2+|\nabla^3\widehat{u}|_2| W|_\infty\big)|\nabla^2W|_2\leq C_0Y_2Z(1+t)^{-2-\gamma_2},\\
S_3^2(W)\leq& C\big(|\nabla^2\widehat{u}|_\infty|\nabla^2 W|_2+|\nabla^3\widehat{u}|_2|\nabla W|_\infty
+|\nabla^4\widehat{u}|_2|W|_\infty\big)|\nabla^3W|_2\\
\leq& C_0Y_3Z(1+t)^{-2-\gamma_3},\\
\end{split}\end{equation}
which imply that
$$
S_k^2(W)\leq C_0Y_kZ(1+t)^{-2-\gamma_k},\ \ \text{for}\ \ k=1,2,3.
$$

This, together with \eqref{E:gz42},  yields $\eqref{E:gz36}_2$.
\end{proof}
\begin{lemma}\label{L:3.3}{\bf(Estimates on $L_k$)}. For any suitable small constant $\nu>0$, there
are two constants $C(\nu)$ and $C_0(\nu)$ such that
\begin{equation}\label{E:gz44}\begin{split}
L_k(W)\leq& \nu\epsilon|\varphi\nabla^{k+1}v|_2^2\delta_{3k}+C(\nu)(1+t)^{2m-5-\gamma_k}Z^3Y_k\\
&+C_0\epsilon^{1/2}(1+t)^{2m-n-4.5-\gamma_k}Z^2Y_k+C_0(\nu)\epsilon (1+t)^{2m-10}Z^2.
\end{split}\end{equation}
\end{lemma}
\begin{proof}
{\bf Step 1:} Estimates on $L^1_k.$ It is easy to check that,
\begin{equation}\label{E:gz45}\begin{split}
L^1_k\leq& C\epsilon|\varphi|_\infty|\nabla\varphi|_\infty|\nabla^{k+1}v|_2|\nabla^k v|_2\\
\leq & C\epsilon^{3/4}(1+t)^{2m-5-\gamma_k}Z^3Y_k \ \ \text{for} \ \  k\leq 2,\\
L^1_3\leq& C\epsilon|\varphi\nabla^4 v|_2|\nabla\varphi|_\infty|\nabla^3v|_2\\
\leq & \nu\epsilon|\varphi\nabla^4 v|_2^2+C(\nu)\epsilon^{1/2}(1+t)^{2m-5-\gamma_3}Z^3Y_3,
\end{split}\end{equation}
where $\nu>0$ is some sufficiently small constant.

{\bf Step 2:} Estimates on $L_k^3.$ If $k=0,1$, one has
\begin{equation}\label{E:gz46}\begin{split}
L^3_0\leq& C\epsilon|\varphi|_\infty|\nabla^2\widehat{u}|_\infty|\varphi|_2|v|_2\leq C_0\epsilon (1+t)^{2m-n-4.5-\gamma_0}Z^2Y_0, \\
L^3_1\leq& C\epsilon\big(|\nabla \varphi|_2|\nabla^2\widehat{u}|_\infty+|\varphi|_\infty|\nabla^3\widehat{u}|_2\big)|\varphi |_\infty|\nabla v|_2\\
\leq& C_0 \epsilon (1+t)^{2m-n-4.5-\gamma_1}Z^2Y_1.
\end{split}\end{equation}
For the case of $k=2,$ decompose $L^3_2\doteq L^3_2(0,2)+L^3_2(1,1)+L^3_2(2,0).$ One gets
\begin{equation}\label{E:gz47}\begin{split}
L_2^3(0,2)\doteq&\epsilon\int\varphi^2\nabla^2 L\widehat{u}\cdot\nabla^2 v\\
\leq&   C\epsilon | \varphi|^2_\infty |\nabla^2 L\widehat{u}|_2|\nabla^2 v|_2
\leq   C_0\epsilon(1+t)^{2m-n-4.5-\gamma_2}Z^{2}  Y_2,\\
L_2^3(1,1)\doteq&\epsilon\int\nabla\varphi^2\cdot \nabla L\widehat{u}\cdot\nabla^2 v\\
\leq& C\epsilon|\varphi|_\infty|\nabla\varphi|_\infty|\nabla L\widehat{u}|_2|\nabla^2 v|_2
\leq  C_0 \epsilon^{3/4}(1+t)^{2m-n-4.5-\gamma_2}Z^2Y_2,\\
L^3_2(2,0)\doteq& \epsilon\int \nabla^2\varphi^2\cdot L\widehat{u}\cdot\nabla^2 v\\
\leq & C\epsilon\big(|\nabla\varphi|_\infty^2|L\widehat{u}|_2+|\varphi|_\infty|\nabla^2\varphi|_2|L\widehat{u}|_\infty\big)|\nabla^2 v|_2\\
\leq& C_0\epsilon^{1/2}(1+t)^{2m-n-4.5-\gamma_2}Z^2Y_2.\\
\end{split}\end{equation}
For last case of $k=3$, decompose $L^3_3\doteq L^3_3(0,3)+L^3_3(1,2)+L^3_3(3,0)$. Then, by integration by parts, one can obtain
\begin{equation}\label{E:gz48}\begin{split}
L^3_3(0,3)\doteq&\epsilon\int\varphi^2\nabla^3L\widehat{u}\cdot\nabla^3 v\\
\leq& C\epsilon|\varphi|_\infty|\nabla^2L\widehat{u}|_2\big(|\varphi\nabla^4 v|_2+|\nabla\varphi|_\infty|\nabla^3 v|_2\big)\\
\leq& \nu\epsilon|\varphi\nabla^{4}v|_2^2+C_0\epsilon^{3/4}(1+t)^{2m-n-4.5-\gamma_3}Z^2Y_3+C_0(\nu)\epsilon(1+t)^{2m-10}Z^2,\\
L^3_3(1,2)\doteq& \epsilon\int \nabla\varphi^2\cdot\nabla^2 L\widehat{u}\cdot\nabla^3 v\\
\leq &
C\epsilon|\varphi|_\infty|\nabla\varphi|_\infty|\nabla^2 L\widehat{u}|_2|\nabla^3 v|_2
\leq C_0\epsilon^{3/4}(1+t)^{2m-n-4.5-\gamma_3}Z^2Y_3,\\
L^3_3(2,1)\doteq& \epsilon\int\nabla^2\varphi^2\cdot \nabla L\widehat{u}\cdot\nabla^3 v\\
\leq& C\epsilon
\big(|\nabla\varphi|_\infty^2|\nabla^3 \widehat{u}|_2+|\varphi|_\infty|\nabla^2\varphi|_6|\nabla^3\widehat{u}|_3\big)|\nabla^3 v|_2\\
\leq& C_0\epsilon^{1/2}(1+t)^{2m-n-4.5-\gamma_3}Z^2Y_3,\\
L^3_3(3,0)\doteq&\epsilon\int
\nabla^3\varphi^2\cdot L\widehat{u}\cdot\nabla^3 v\\
\leq &C\epsilon\big(|\varphi|_\infty|\nabla^3 \varphi|_2+|\nabla\varphi|_\infty|\nabla^2\varphi|_2\big)| L\widehat{u}|_\infty|\nabla^3 v|_2\\
\leq& C_0\epsilon^{1/2}(1+t)^{2m-n-4.5-\gamma_3}Z^2Y_3.\\
\end{split}\end{equation}
It follows from \eqref{E:gz46}-\eqref{E:gz48} that
\begin{equation}\label{E:gz49}\begin{split}
L^3_k
\leq& \nu\epsilon|\varphi\nabla^{k+1}v|_2^2\delta_{3k}+ C_0\epsilon^{1/2}(1+t)^{2m-n-4.5-\gamma_k}Z^2Y_k
+C_0(\nu)\epsilon(1+t)^{2m-10}Z^2.\\
\end{split}\end{equation}
{\bf Step 3:}
Estimates on $L_k^2.$ If $k=1,$ one gets
\begin{equation}\label{E:gz50}\begin{split}
L^2_1=& \epsilon\int \nabla\varphi^2\cdot L v\cdot\nabla v\\
\leq & C\epsilon|\varphi|_\infty|\nabla\varphi|_\infty|\nabla^2 v|_2|\nabla v|_2\leq C \epsilon^{3/4}(1+t)^{2m-5-\gamma_1}Z^3Y_1.
\end{split}\end{equation}
Next for $k=2,$ decompose $L^2_2\doteq L^2_2(1,1)+L^2_2(2,0).$ In a similar way for $L_2^1,$ one has
\begin{equation}\label{E:gz51}\begin{split}
L^2_2(1,1)\doteq& \epsilon\int \nabla\varphi^2\cdot\nabla L v\cdot\nabla^2 v\\
\leq & C\epsilon|\varphi|_\infty|\nabla\varphi|_\infty|\nabla^3 v|_2|\nabla^2 v|_2
\leq C\epsilon^{3/4}(1+t)^{2m-5-\gamma_2}Z^3 Y_2,\\
L^2_2(2,0)\doteq& \epsilon\int \nabla^2\varphi^2\cdot L v\cdot\nabla^2 v\\
\leq &
C\epsilon\big(|\nabla \varphi|_\infty^2
|\nabla^2 v|_2+|\varphi|_6|\nabla^2 v|_6|\nabla^2\varphi|_6\big)|\nabla^2 v|_2\\
\leq& C \epsilon^{1/2}(1+t)^{2m-5-\gamma_2}Z^3Y_2.
\end{split}\end{equation}
At last, for $k=3$, decompose $L^2_3\doteq L^2_3(1,2)+L^2_3(2,1)+L^2_3(3,0).$ In a similar way for $L_3^1$, on can get
\begin{equation}\label{E:gz52}\begin{split}
L^2_3(1,2)\doteq& \epsilon\int \nabla\varphi^2\cdot\nabla^2 L v\cdot\nabla^3 v\\
\leq & \nu\epsilon|\varphi\nabla^4 v|_2^2+C(\nu)\epsilon^{1/2}(1+t)^{2m-5-\gamma_3}Z^3Y_3,\\
L^2_3(2,1)\doteq& \epsilon\int \nabla^2\varphi^2\cdot \nabla L v\cdot\nabla^3 v\\
\leq & C\epsilon\big(|\nabla\varphi|_\infty^2|\nabla^3 v|_2+|\nabla^2\varphi|_3|\varphi\nabla Lv|_6\big)|\nabla^3 v|_2\\
\leq& \nu\epsilon|\varphi\nabla^4 v|_2^2+C(\nu)\epsilon^{1/2}(1+t)^{2m-5-\gamma_3}Z^3Y_3,\\
L^2_3(3,0)\doteq& \epsilon\int \nabla^3\varphi^2\cdot L v\cdot\nabla^3 v\\
\leq &
C\epsilon\big(|\nabla\varphi|_\infty|\nabla^2\varphi|_3|\nabla^3 v|_2|\nabla^2 v|_6+|\varphi\nabla^3 v|_6
|\nabla^2 v|_3|\nabla^3\varphi|_2\big)\\
\leq& \nu\epsilon|\varphi\nabla^4 v|_2^2+C(\nu)(1+t)^{2m-5-\gamma_3}Z^3Y_3.
\end{split}\end{equation}
Then combining the estimates \eqref{E:gz50}-\eqref{E:gz52} yields
\begin{equation}\label{E:gz53}\begin{split}
L^2_k\leq & \nu\epsilon|\varphi\nabla^4 v|_2^2\delta_{3k}+C(\nu)(1+t)^{2m-5-\gamma_k}Z^3Y_k.
\end{split}\end{equation}
\end{proof}

\begin{lemma}\label{L:3.4} {\bf Estimates on $Q_k$}. For any suitable small constant $\nu>0,$
there are two constants $C(\nu)$ and $C_0(\nu)$ such that
\begin{equation}\label{E:gz54}\begin{split}
Q_k(W,\widehat{u})\leq & \nu\epsilon|\varphi\nabla^{k+1} v|_2^2\delta_{3k}+C(\nu)(1+t)^{2m-5-\gamma_k}Z^3Y_k\\
&+C_0\epsilon^{\frac{1}{4}}(1+t)^{2m-n-3.5-\gamma_k}Z^2 Y_k+C_0(\nu)(1+t)^{2m-10}Z^2\\
&+\frac{(4\alpha+6\beta)\delta\epsilon}{(\delta-1)(1+t)}|\nabla^3\varphi|_2|\varphi\nabla^3\text{div}v|_2.\\
\end{split}\end{equation}
\end{lemma}
\begin{proof}
{\bf Step 1:} Estimates on $Q^1_k.$ Similar as $L^1_k$, it is easy to get
$$
Q_k^1\leq \nu\epsilon|\varphi\nabla^{k+1}v|_2^2\delta_{3,k}+C(\nu)\epsilon^{\frac{1}{2}}(1+t)^{2m-5-\gamma_k}Z^3Y_k.
$$
{\bf Step 2:} Estimates on $Q^3_k.$ If $k=0,$ one has
\begin{equation}\label{E:gz55}\begin{split}
Q_0^3\leq C\epsilon|\varphi|_\infty|\nabla\widehat{u}|_\infty|\nabla\varphi|_2|v|_2\leq C_0\epsilon
(1+t)^{2m-n-3.5-\gamma_0}Z^2Y_0.\\
\end{split}\end{equation}
For $k=1,$ in a similar way for $L^3_1,$ one has
\begin{equation}\label{E:gz56}\begin{split}
Q_1^3\leq& C\epsilon\big(|\varphi|_\infty|\nabla^2\varphi|_2|\nabla\widehat{u}|_\infty+|\nabla\varphi|_6|\nabla\varphi|_3|\nabla \widehat{u}|_\infty
+|\varphi|_\infty|\nabla\varphi|_2|\nabla^2\widehat{u}|_\infty\big)|\nabla v|_2\\
\leq& C_0\epsilon^{3/4}
(1+t)^{2m-n-3.5-\gamma_1}Z^2Y_1.\\
\end{split}\end{equation}
For $k=2,$ decompose $Q^3_2\doteq Q^3_2(0,2)+Q^3_2(1,1)+Q^3_2(2,0)$. In a similar way for $L^3_2(1,1)$ and $L^3_2(2,0)$, one obtains
\begin{equation}\label{E:gz57}\begin{split}
Q_2^3(0,2)\doteq& C\epsilon\int\nabla\varphi^2\cdot\nabla^3\widehat{u}\cdot\nabla^2 v\\
\leq & C_0\epsilon^{3/4}
(1+t)^{2m-n-4.5-\gamma_2}Z^2Y_2,\\
Q_2^3(1,1)\doteq& C\epsilon\int\nabla^2\varphi^2\cdot\nabla^2\widehat{u}
\cdot\nabla^2 v\\
\leq &C_0\epsilon^{1/2}
(1+t)^{2m-n-4.5-\gamma_2}Z^2Y_2.\\
Q_2^3(2,0)\doteq& C\epsilon\int\nabla^3\varphi^2\cdot\nabla\widehat{u}\cdot\nabla^2 v\\
\leq  & C\epsilon\big(  |\nabla^3 \varphi|_2  |\varphi|_\infty + |\nabla^2 \varphi|_6|\nabla \varphi|_3\big)|\nabla \widehat{u}|_\infty |\nabla^2 v|_2\\
\leq &
C_0\epsilon^{1/2} (1+t)^{2m-n-3.5-\gamma_2}Z^{2}Y_2.
\end{split}\end{equation}
For $k=3$, let $Q^3_3\doteq Q^3_3(0,3)+Q^3_3(1,2)+Q^3_3(2,1)+Q^3_3(3,0)$, as for $L_3^3(1,2)$, $L^3_3(2,1)$ and  $L^3_3(3,0)$, one can get
\begin{equation}\label{E:gz58a}\begin{split}
\sum_{i=0}^2Q_3^3(i,3-i)\doteq& C\epsilon\int\big(\nabla\varphi^2\cdot\nabla^4\widehat{u}\cdot\nabla^3 v+\nabla^2\varphi^2\cdot\nabla^3\widehat{u}
\cdot\nabla^3 v+\nabla^3\varphi^2\cdot\nabla^2\widehat{u}\cdot\nabla^3 v\big)\\
\leq& C_0\epsilon^{1/2}
(1+t)^{2m-n-4.5-\gamma_3}Z^2Y_3.
\end{split}\end{equation}

Finally, we estimate $Q_3^3(3,0)$ defined below, for which some additional information is needed. Using
integration by parts and Proposition \ref{p1}, one has
\begin{equation}\label{E:gz58}\begin{split}
Q_k^3(3,0)\doteq
&\frac{\delta\epsilon}{\delta-1}\sum_{i,j=1}^3\int\big(\alpha\nabla^3\partial_j \varphi^2(\partial_i\widehat{u}^{(j)}+\partial_j\widehat{u}^{(i)})
+\beta\nabla^3\partial_i\varphi^2\text{div}\widehat{u}\big)\cdot \nabla^3 v^{(i)}\\
=&Q^3_3(A)-\frac{\delta\epsilon}{\delta-1}\sum_{i,j=1}^3\int\alpha\nabla^3\varphi^2(\partial_i\widehat{u}^{(j)}+\partial_j\widehat{u}^{(i)})\cdot \nabla^3 \partial_jv^{(i)}\\
&-\frac{\delta\epsilon}{\delta-1}\int
\beta(\nabla^3 \varphi^2\text{div}\widehat{u})\cdot \nabla^3\text{div} v\\
=&Q^3_3(A)+Q^3_3(B)+Q_3^3(D)\\
&-\frac{\delta\epsilon}{(\delta-1)(1+t)}\int(4\alpha+6\beta)\varphi\nabla^3\varphi\cdot \nabla^3\text{div} v\\
\leq&Q^3_3(A)+Q^3_3(B)+Q_3^3(D)+\frac{(4\alpha+6\beta)\delta\epsilon}{(\delta-1)(1+t)}|\nabla^3\varphi|_2|\varphi\nabla^3\text{div} v|_2.\\
\end{split}\end{equation}
where
\begin{equation}\label{E:gz58a}\begin{split}
Q_3^3(A)\doteq&C\epsilon\int\nabla^3\varphi^2\cdot\nabla^2\widehat{u}\cdot\nabla^3 v \\
\leq& C\epsilon \big(|\varphi|_\infty|\nabla^3\varphi|_2+|\nabla\varphi|_\infty|\nabla^2\varphi|_2\big)|\nabla^2\widehat{u}|_\infty
|\nabla^3 v|_2\\
\leq& C_0\epsilon^{1/2}(1+t)^{2m-n-4.5-\gamma_3}Z^2Y_3,\\
Q_3^3(B)\doteq&C\epsilon\int\nabla\varphi\cdot\nabla^2\varphi\cdot\nabla\widehat{u}\cdot\nabla^4 v\\
\leq& C\epsilon\big(|\nabla\varphi|_\infty|\nabla^3\varphi|_2|\nabla\widehat{u}|_\infty+|\nabla^2\varphi|_6(|\nabla^2 \varphi|_3
|\nabla\widehat{u}|_\infty+|\nabla\varphi|_3|\nabla^2\widehat{u}|_\infty)\big)|\nabla^3 v|_2\\
\leq& C_0\epsilon^{1/4}(1+t)^{2m-n-3.5-\gamma_3}Z^2Y_3,\\
Q_3^3(D)\doteq&-\frac{\delta\epsilon}{(\delta-1)(1+t)^2}\sum_{i,j=1}^3\int\Big(
2\alpha\varphi\nabla^3\varphi(K_{ij}+K_{ji})\cdot\nabla^3\partial_j v^{(i)}\\
&+2\beta\varphi\nabla^3\varphi K_{ii}\cdot\nabla^3\text{div} v\Big)\\
\leq &\epsilon\nu|\varphi\nabla^4 v|_2^2+C_0(\nu)(1+t)^{2m-10}Z^2,
\end{split}\end{equation}
where the matrix $K=\{K_{ij}\}$ is defined in Proposition \ref{p1}.

Then, combining the estimates \eqref{E:gz55}-\eqref{E:gz58a} yields
\begin{equation}\label{E:gz58b}\begin{split}
Q_k^3
\leq &\epsilon\nu|\varphi\nabla^4 v|_2^2\delta_{3k}+C_0(\nu)(1+t)^{2m-10}Z^2+C_0\epsilon^{1/4}(1+t)^{2m-n-3.5-\gamma_k}Z^2Y_k\\
&+\frac{(4\alpha+6\beta)\delta\epsilon}{(\delta-1)(1+t)}|\nabla^3\varphi|_2|\varphi\nabla^3\text{div}v|_2.\\
\end{split}\end{equation}

{\bf Step 3:} Estimates on $Q_k^2.$ For $k=1,$ direct computation gives
\begin{equation}\label{E:gz59}\begin{split}
Q_1^2=& C\epsilon\int\big(\varphi\nabla^2\varphi+\nabla\varphi\cdot\nabla\varphi\big)\cdot\nabla v\cdot\nabla v\\
\leq& C\epsilon\big(|\varphi|_3|\nabla^2\varphi|_6|\nabla v|_\infty+|\nabla v|_2|\nabla\varphi|_\infty^2
\big)|\nabla v|_2\\
\leq&C\epsilon^{1/2}(1+t)^{2m-5-\gamma_1}Z^3Y_1.\\
\end{split}\end{equation}
For $k=2$, in a similar way for $L_2^2(2,0)$, one gets
\begin{equation}\label{E:gz60}\begin{split}
Q_2^2=& C\epsilon\int\big(\nabla^3\varphi^2\cdot\nabla v+\nabla^2\varphi^2\cdot\nabla^2v\big)\cdot\nabla^2 v\\
\leq& C\epsilon\big(|\varphi|_\infty|\nabla^3\varphi|_2|\nabla v|_\infty+|\nabla v|_\infty|\nabla\varphi|_6
|\nabla^2\varphi|_3\big)|\nabla^2v|_2\\
&+C\epsilon^{1/2}(1+t)^{2m-5-\gamma_2}Z^3Y_2\\
\leq& C\epsilon^{1/2}(1+t)^{2m-5-\gamma_2}Z^3Y_2.
\end{split}\end{equation}
For $k=3$, as for $L^2_3(3,0)$ and $L_3^2(2,1)$, it follows from integration by parts that
\begin{equation}\label{E:gz60}\begin{split}
Q_3^2=& C\epsilon\int\big(\nabla^4\varphi^2\cdot\nabla v+\nabla^3\varphi^2\cdot\nabla^2v+\nabla^2\varphi^2\cdot\nabla^3 v\big)\cdot\nabla^3 v\\
\leq&  \nu\epsilon|\varphi\nabla^4 v|_2^2+C(\nu)(1+t)^{2m-5-\gamma_3}Z^3Y_3+Q_3^2(A),\\
\end{split}\end{equation}
with
\begin{equation}\label{E:gz60a}\begin{split}
Q_3^2(A)\doteq& C\epsilon\int\nabla^3\varphi^2\cdot \nabla v\cdot\nabla^4 v\\
=& C\epsilon\int(\varphi\nabla^3\varphi+\nabla\varphi\cdot\nabla^2\varphi)
\cdot\nabla v\cdot\nabla^4 v\\
\leq& C\epsilon|\nabla^3\varphi|_2|\nabla v|_\infty|\varphi\nabla^4 v|_2+Q^2_3(B)\\
\leq &\nu\epsilon|\varphi\nabla^4v|_2^2+C(\nu)(1+t)^{2m-5-\gamma_3}Z^3Y_3+Q_3^2(B),\\
\end{split}\end{equation}
where $Q^2_3(B)$ can be estimated by using integration by parts again,
\begin{equation}\label{E:gz60b}\begin{split}
Q_3^2(B)\doteq& C\epsilon\int\nabla\varphi\cdot\nabla^2\varphi\cdot\nabla v\cdot\nabla^4 v\\
\leq& C\epsilon\big(|\nabla^2\varphi|^2_6|\nabla v|_6+|\nabla^3 \varphi|_2|\nabla\varphi|_\infty|\nabla v|_\infty
+|\nabla\varphi|_6|\nabla^2\varphi|_6|\nabla^2 v|_6\big)|\nabla^3 v|_2\\
\leq &C(1+t)^{2m-5-\gamma_3}Z^3Y_3.\\
\end{split}\end{equation}
Collecting the estimates \eqref{E:gz59}-\eqref{E:gz60b} shows that
\begin{equation}\label{E:gz60c}\begin{split}
Q_k^2(B)\leq \nu\epsilon|\varphi\nabla^4v|_2^2\delta_{3k}+C(\nu)(1+t)^{2m-5-\gamma_k}Z^3Y_k.
\end{split}\end{equation}

\end{proof}
It follows from \eqref{E:gz33} and Lemmas \ref{L:3.2}-\ref{L:3.4} that
\begin{lemma}\label{L:3.5}
\begin{equation}\label{E:gz61}\begin{split}
&\frac{1}{2}\frac{d}{dt}Y_k^2+\frac{k+r}{1+t}Y_k^2+\alpha\epsilon\int|\varphi\nabla^{k+1}v|^2+(\alpha+\beta)\epsilon
\int|\varphi\nabla^k\text{div}v|^2\\
\leq &C(\nu)(1+t)^{2m-5-\gamma_k}Z^3Y_k\\
&+C_0\epsilon^{\frac{1}{4}}(1+t)^{2m-n-3.5-\gamma_k}Z^2Y_k+C(1+t)^{n-2.5-\gamma_k}Z^2Y_k\\
&+C_0(\nu)\big((1+t)^{2m-10}Z^2+(1+t)^{-\gamma_k-2}ZY_k\big)\\
&+\nu\epsilon|\varphi\nabla^{k+1}v|_2^2\delta_{3,k}+\frac{(4\alpha+6\beta)\delta \epsilon }{(\delta-1)(1+t)}|\nabla^3\varphi|^2_2
|\varphi\nabla^{k}\text{div}v|^2_2\delta_{3k}.\\
\end{split}\end{equation}
\end{lemma}

\subsubsection{Energy estimate on $\varphi$}

\begin{lemma}\label{L:3.6}
\begin{equation}\label{E:gz61a}\begin{split}
&\frac{1}{2}\Theta^2(k)\frac{d}{dt}|\nabla^k\varphi|_2^2
+\frac{\frac{3\delta}{2}-3+k}{1+t}\Theta^2(k)|\nabla^k\varphi|_2^2\\
\leq &
C(1+t)^{n-2.5-\delta_k}Z^2U_k+C_0(1+t)^{-2-\delta_k}ZU_k+\frac{\delta-1}{2}\epsilon|\nabla^3\varphi|_2
|\varphi\nabla^{3}\text{div}v|_2\delta_{3k}.
\end{split}\end{equation}

\end{lemma}
\begin{proof}
\textbf{Step 1}.
Applying $\nabla^k$ to $\eqref{li47-1}_1$, multiplying $\nabla^k\varphi$ for $k=1,2$,  and $\epsilon\nabla^k \varphi$ for $k=3$,  and integrating by parts over $\R^3$,  one gets
\begin{equation}\label{E:gz62}\begin{split}
\frac{1}{2}\Theta^2(k)\frac{d}{dt}|\nabla^k\varphi|_2^2=&\int\Theta^2(k)\big(v\cdot\nabla^{k+1}\varphi+\frac{\delta-1}{2}\nabla^k\varphi\text{div} v\big)
\cdot\nabla^k\varphi\\
&+S_k^*(\varphi,\widehat{u})-\Lambda_1^k-\Lambda_2^k\\
\leq& C\Theta^2(k)|\nabla v|_\infty|\nabla^k \varphi(t)|_2^2+S_k^*(\varphi,\widehat{u})-\Lambda_1^k-\Lambda_2^k,
\end{split}\end{equation}
where $S_k^*(\varphi,\widehat{u})$ is given in \textbf{Step 2} below, and
$$
\Lambda_1^k=\Theta^2(k)\big(\nabla^k(v\cdot\nabla\varphi)-v\cdot\nabla^{k+1}\varphi\big)\cdot \nabla^k \varphi,\quad
\Lambda_2^k=\Theta^2(k)\frac{\delta-1}{2}\big(\nabla^k(\varphi\text{div}v)-\nabla^{k}\varphi\text{div}v\big)\cdot\nabla^{k}\varphi.
$$
\textbf{Step 2:} Estimates on  $S^*_k=\int \Theta^2(k)s^*_k$ with the  integrand defined as
 \begin{equation}\label{eq:2.8ss}
 \begin{split}
 s^*_k(\varphi,  \widehat{u})=&-\frac{\delta-1}{2}\nabla^k \varphi \cdot \nabla^k \varphi \text{div}\widehat{u}+\frac{1}{2}\text{div}\widehat{u}  \nabla^k \varphi  \cdot \nabla^k \varphi \\
 &-\frac{\delta-1}{2}\nabla^k \varphi \cdot \Big(\nabla^k(\varphi \text{div}\widehat{u})-\text{div}\widehat{u}\nabla^k \varphi \Big)\\
&-\nabla^k \varphi \cdot \Big( \nabla^k\big( \widehat{u} \cdot \nabla \varphi\big)- \widehat{u} \cdot \nabla^{k+1} \varphi \Big)\equiv: s^{*1}_k+s^{*2}_k,
 \end{split}
 \end{equation}
where $s^{*1}_k$ is a sum of terms with a derivative of order one of $\widehat{u}$, and $s^{*2}_k$ is a sum of terms with a derivative of order  at least two of $\widehat{u}$.

\textbf{Step 2.1}: Estimates on  $S^{*1}_k=\int \Theta^2(k)s^{*1}_k$.
Let $\nabla^k=\partial_{\beta_1...\beta_k}$ with $0\leq \beta_i=1,2,3$.
 \begin{equation}\label{a1ss}
 \begin{split}
S^{*1}_k=&\int \Theta^2(k)\Big(-\frac{\delta-1}{2}\nabla^k \varphi \cdot \nabla^k \varphi \text{div}\widehat{u}+\frac{1}{2}\text{div}\widehat{u}  \nabla^k \varphi  \cdot \nabla^k \varphi\Big)\\
&-\int \Theta^2(k)\partial_{\beta_1...\beta_k} \varphi \cdot \sum_{i=1}^k \sum_{j=1}^k \partial_{ \beta_i}\widehat{u}^{(j)} \partial_j\partial_{\beta_1...\beta_{i-1}\beta_{i+1}...\beta_k}\varphi =J_1+J_2+J_3,
 \end{split}
 \end{equation}
where, by Proposition \ref{p1}, $J_1$-$J_3$ are estimated  by
 \begin{equation}\label{ka1ss}
 \begin{split}
J_1=&-\frac{3(\delta-1)}{2(1+t)}U^2_k+N_1,\quad
J_2=\frac{3}{2}\frac{1}{1+t}U^2_k+N_2,\\
J_3=& -\frac{k}{1+t}U^2_k+N_3, \quad |N_j|\leq \frac{C_0}{(1+t)^2}U^2_k,\quad j=1,2,3.
 \end{split}
 \end{equation}

 Therefore, it holds that
 \begin{equation}\label{eq:2.11ss}
 \begin{split}
 & S^{*1}_k(W,  \widehat{u})(t, \cdot)\text{d} x
 \leq  \frac{C_0}{(1+t)^2}U^2_k-\frac{\frac{3}{2}\delta-3+k}{1+t}U^2_k.
 \end{split}
 \end{equation}

\textbf{Step 2.2:} Estimates on  $S^{*2}_k=\int \Theta^2(k)s^{*2}_k$.
$s^{*2}_k$ is a sum of the terms defined as
$$
E_1(\varphi)=\Theta^2(k)\nabla^k \varphi \cdot \nabla^l \widehat{u} \cdot  \nabla^{k+1-l} \varphi\quad \text{for}\quad  2\leq k\leq 3 \quad  \text{and}\quad  2\leq l\leq k;
$$
$$
E_2(\varphi)=\Theta^2(k)\nabla^k \varphi \cdot \nabla^{l+1} \widehat{u} \cdot \nabla^{k-l} \varphi\quad \text{for}\quad  1\leq k\leq 3 \quad  \text{and}\quad  1\leq l\leq k.
$$
Then it holds that
\begin{equation}\label{zhon5r}
\begin{split}
S^{*2}_1(\varphi)\leq & C| \varphi|_6|\nabla^2\widehat{u}|_2|\nabla \varphi|_3
\leq C_0 (1+t)^{-2-\delta_1}ZU_1,\\
 S^{*2}_2(\varphi)\leq &C(|\nabla^2 \widehat{u}|_\infty |\nabla \varphi|_2+|\nabla^3 \widehat{u}|_2|\varphi|_\infty) |\nabla^2 \varphi|_2
\leq  C_0 (1+t)^{-2-\delta_2}ZU_2,\\
 S^{*2}_3(\varphi)\leq & C\epsilon \big(|\nabla^2 \widehat{u}|_\infty|\nabla^{2}\varphi|_2+|\nabla^3 \widehat{u}|_6|\nabla \varphi|_3+|\nabla^4 \widehat{u}|_2|\varphi|_\infty\big) |\nabla^3 \varphi|_2\\
 \leq&  C_0\epsilon^{\frac{1}{2}} (1+t)^{-2-\delta_3}ZU_3,
\end{split}
\end{equation}
which  implies immediately  that
$$
S^{*2}_k(\varphi)\leq  C_0 (1+t)^{-2-\delta_k}ZU_k, \quad \text{for} \quad  k=1,2,3.
$$

{\bf Step 3:} Estimates on $\Lambda_1^k+\Lambda_2^k.$ It follows from Lemma \ref{L:3.2}
and H$\ddot{\text{o}}$lder's inequality that
\begin{equation}\label{E:gz67}\begin{split}
\Lambda_1^1+\Lambda_2^1\leq& C|\nabla v|_\infty|\nabla\varphi|_2^2\leq C(1+t)^{n-2.5-\delta_1}Z^2U_1,\\
\Lambda_1^2+\Lambda_2^2\leq& C\big(|\nabla \varphi|_3|\nabla^2 v|_6+|\nabla v|_\infty|\nabla^2\varphi|_2\big)
|\nabla^2\varphi|_2\\
\leq& C(1+t)^{n-2.5-\delta_2}Z^2U_2,\\
\Lambda_1^3\leq& C\epsilon\big(|\nabla \varphi|_\infty|\nabla^3 v|_2+|\nabla^2 \varphi|_6|\nabla^2 v|_3
+|\nabla v|_\infty|\nabla^3\varphi|_2\big)|\nabla^3\varphi|_2\\
\leq& C(1+t)^{n-2.5-\delta_3}Z^2 U_3,\\
\Lambda_2^3\leq& C\epsilon\big(|\nabla \varphi|_\infty|\nabla^3 v|_2+|\nabla^2 \varphi|_6|\nabla^2 v|_3
+\frac{\delta-1}{2}|\varphi\nabla^3\text{div} v|_2\big)|\nabla^3\varphi|_2\\
\leq& C(1+t)^{n-2.5-\delta_3}Z^2U_3+\frac{\delta-1}{2}\epsilon|\varphi\nabla^3\text{div} v|_2|\nabla^3\varphi|_2\\
\end{split}\end{equation}
Thus, \eqref{E:gz62}-\eqref{E:gz67} yields the desired \eqref{E:gz61a}.
\end{proof}

\subsubsection{Energy inequality on $Z$}

Finally, set
\begin{equation}\label{E:gz68}\begin{split}
H(\Phi^*,\Phi,\Psi)=&\alpha(\Phi^*)^2+(\alpha+\beta)\Phi^2+\frac{\frac{3}{2}\delta-3+m}{1+t}\B^2\\
&-\Big(\frac{\delta-1}{2}(1+t)^{n-m}+\frac{2\delta}{\delta-1}(2\alpha+3\beta)(1+t)^{m-n-1}\Big)\A\B,\\
\end{split}\end{equation}
where
$$
\A^*=(1+t)^{\gamma_3}\epsilon^{\frac{1}{2}}|\varphi\nabla^4 v|_2,\quad \A=(1+t)^{\gamma_3}\epsilon^{\frac{1}{2}}|\varphi\nabla^3\text{div}v|_2,
\quad \B=(1+t)^{\delta_3}\epsilon^{\frac{1}{2}}|\nabla^3\varphi|_2.
$$

Then the following lemma holds.
\begin{lemma}\label{L:3.5}
There exists some positive constants $b_{m,n}$  such that
\begin{equation}\label{E:gz69}\begin{split}
\frac{1}{2}\frac{d}{dt}Z^2&+\frac{b_{m,n}}{1+t}(Z^2-\B^2)+H(\A^*,\A,\B)+\alpha \epsilon\sum_{k=0}^2(1+t)^{2\gamma_k}
|\varphi\nabla^{k+1} v|_2^2\\
&+(\alpha+\beta)\epsilon\sum_{k=0}^2(1+t)^{2\gamma_k}
|\varphi\nabla^k\text{div} v|_2^2\\
\leq& \nu(1+t)^{2\gamma_3}\epsilon|\varphi\nabla^4 v|^2_2+C(\nu)(1+t)^{2m-5}Z^4+C_0\epsilon^{\frac{1}{4}}(1+t)^{2m-n-3.5}Z^3\\
&+C(1+t)^{n-2.5}Z^3+C_0(\nu)\big((1+t)^{2m-2n-4}+(1+t)^{-2}\big)Z^2,\\
\end{split}\end{equation}
with the constant $b_{m,n}$ given by
\begin{equation}\label{E:gz70}\begin{split}
b_{m,n}=\begin{cases}
\min\{n-0.5,\frac{3\delta}{2}-3+m\},\quad \ \ \ \ \   \quad  \text{if}\  \quad \gamma\geq\frac{5}{3},\\
\min\{\frac{3\gamma}{2}-3+n,\frac{3\delta}{2}-3+m\},\quad \ \ \ \ \text{if}\ \quad 1<\gamma<\frac{5}{3}.\\
\end{cases}
\end{split}\end{equation}
\end{lemma}
\begin{proof}
\eqref{E:gz69} can be obtained by multiplying \eqref{E:gz61} and \eqref{E:gz61a} by $(1+t)^{2\gamma_k}$ and
$(1+t)^{2\delta_k}$ respectively and summing the resulting inequalities together.
\end{proof}

\subsubsection{Proof of Uniform energy estimates under the condition $(P_1)$}
Set
\begin{equation}\label{Agab}\begin{split}
M_1=&\frac{2\alpha+3\beta}{2\alpha+\beta},\quad
M_2=-3\delta+1+\frac{1}{2}M_3,\\
M_3=&\frac{(\delta-1)^2}{4(2\alpha+\beta)}+\frac{4\delta^2(2\alpha+\beta)}{(\delta-1)^2}M^2_1+2M_1\delta,\\
M_4=&\frac{1}{2}\min\Big\{ \frac{3\gamma-3}{2},\frac{-M_2-1}{2},1\Big\}+M_2,\\
\Lambda(\A,\B)=&(2\alpha+\beta)\A^2+\frac{\frac{3}{2}\delta-3+m}{1+t}\B^2\\
&-\Big(\frac{\delta-1}{2}(1+t)^{n-m}+\frac{2\delta}{\delta-1}(2\alpha+3\beta)(1+t)^{m-n-1}\Big)\A\B.
\end{split}
\end{equation}

\begin{lemma}\label{ll2} Let  condition $(P_1)$ and  $M_1>0$ hold.  Then for $m=n+0.5=3$,
there exist some positive constants $\eta^*$, $b_*$ and $\epsilon^*$,  such that
\begin{equation}\label{Aeq:2.14}
\begin{split}
&\frac{1}{2}\frac{d}{dt} Z^2+\frac{(1-\eta^*)b_*}{1+t} Z^2
 \leq C(1+t)^{1+\epsilon^*}Z^{4}+C_0(1+t)^{-1-\epsilon^*}Z^2,
\end{split}
\end{equation}
where
 \begin{equation}\label{Aeq:2.6q}
\begin{split}
\epsilon^*=&\frac{1}{2}\min\Big\{ \frac{3\gamma-3}{2},\frac{-M_2-1}{2},1\Big\}>0,\\
 \eta^*=&\min\Big\{\frac{3\gamma-3}{4(3\gamma-1)}, \frac{-M_4-1}{6\delta-M_3}, \frac{1}{10}\Big\}>0,\\
b_*=&
 \begin{cases}
\min\Big\{2, \frac{3}{2}\delta-\frac{1}{4}M_3\Big\}>1 \;\qquad\qquad \quad \ \ \ \quad  \ \   \text{if} \ \ \gamma\geq \frac{5}{3},\\[6pt]
\min\Big\{\frac{3\gamma}{2}-0.5, \frac{3}{2}\delta-\frac{1}{4}M_3\Big\}>1  \qquad \ \  \  \ \ \  \quad  \text{if} \ \  1< \gamma<\frac{5}{3}.
 \end{cases}
\end{split}
 \end{equation}
Moreover, there exists a  constant $\Lambda(C_0)$ such that  $Z(t)$ is globally well-defined in $[0,+\infty)$ if $Z_0\leq \Lambda(C_0) $.
\end{lemma}

\begin{proof}
\textbf{Step 1}. We state that
\begin{equation}\begin{split}\label{changea22}
|\varphi\nabla^3\text{div}v|^2_2\leq |\varphi\nabla^4 v|^2_2+J^*,
\end{split}
\end{equation}
where the term $J^*$ can be controlled by
$$
J^*\leq \nu |\varphi \nabla^4 v|^2_2+C(\nu)\epsilon^{-\frac{1}{2}}(1+t)^{2m-5-2\gamma_3}Z^4,$$
for any constant $\nu>0$ small enough, and the constant $C(\nu)>0$.

Actually, according to the definition of $\text{div}$, one directly has
\begin{equation}\label{changeA2}
\begin{split}
|\varphi\nabla^3\text{div}v|^2_2=\sum_{i=1}^3 |\varphi\nabla^3\partial_iv^{(i)}|^2_2+2\sum_{i,j=1,\ i<j}^3 \int \varphi^2 \nabla^3\partial_iv^{(i)}\cdot  \nabla^3\partial_jv^{(j)}.
\end{split}
\end{equation}
Via the integration by parts, one can obtain that
\begin{equation}\label{changeA3}
\begin{split}
J_{ij} =&\int \varphi^2 \nabla^3\partial_iv^{(i)}\cdot  \nabla^3\partial_jv^{(j)}\\
=& \int \varphi^2 \nabla^3\partial_iv^{(j)}\cdot  \nabla^3\partial_jv^{(i)}+\int \big(\partial_i \varphi^2 \nabla^3 \partial_j v^{(i)} \nabla^3  v^{(j)}-\partial_j \varphi^2 \nabla^3 \partial_i v^{(i)} \nabla^3  v^{(j)} \big).
\end{split}
\end{equation}
It is easy to see that
\begin{equation}\label{changeA4}
\begin{split}
J^*_{ij} =& \int \big(\partial_i \varphi^2 \nabla^3 \partial_j v^{(i)} \nabla^3  v^{(j)}-\partial_j \varphi^2 \nabla^3 \partial_i v^{(i)} \nabla^3  v^{(j)} \big)\\
\leq & C|\nabla \varphi|_\infty|\varphi \nabla^4 v|_2|\nabla^3 v|_2\\
\leq & \nu |\varphi \nabla^4 v|^2_2+C(\nu)\epsilon^{-\frac{1}{2}}(1+t)^{2m-5-2\gamma_3}Z^4,
\end{split}
\end{equation}
for any constant $\nu>0$ small enough, and the constant $C(\nu)>0$,  which, along with (\ref{changeA2})-(\ref{changeA3}), quickly implies (\ref{changea22}).

\textbf{Step 2}.
First, note that
\begin{equation}\label{GAB}
\begin{split}
\Lambda(\A,\B)=& a\A^2+b\B^2-c\A\B
=a\big(\A-\frac{c}{2a}\B\big)^2+d \B^2,
\end{split}
\end{equation}
with
\begin{equation*}\begin{split}
a=&(2\alpha+\beta),\quad  b=\frac{\frac{3}{2}\delta-3+m}{1+t},\\
c=&\Big(\frac{\delta-1}{2}(1+t)^{n-m}+\frac{2\delta}{\delta-1}(2\alpha+3\beta)(1+t)^{m-n-1}\Big),\\
d=&(1+t)^{-1}\Big(\frac{3}{2}\delta-3+m\Big)-\frac{1}{4}\Big(\frac{(\delta-1)^2}{4(2\alpha+\beta)}(1+t)^{2n-2m}\\
&+\frac{4\delta^2(2\alpha+\beta)}{(\delta-1)^2}M^2_1 (1+t)^{2m-2n-2}+2M_1\delta (1+t)^{-1}\Big).
\end{split}
\end{equation*}

According to  \eqref{E:gz69} and  \eqref{changea22}, for any $\eta\in(0,1),$ one has
\begin{equation}\label{E:gz73}\begin{split}
\frac{1}{2}\frac{d}{dt}Z^2&+\frac{b_{m,n}}{1+t}(Z^2-\B^2)+\eta\alpha(\A^*)^2+(1-\eta)d\B^2+\eta\Big(b-\frac{c^2}{4(\alpha+\beta)}\Big)\B^2
\\
&+\alpha\epsilon\sum_{k=0}^2(1+t)^{2\gamma_k}
|\varphi\nabla^{k+1} v|_2^2+(\alpha+\beta)\epsilon\sum_{k=0}^2(1+t)^{2\gamma_k}
|\varphi\nabla\text{div} v|_2^2\\
\leq& \nu(\A^*)^{2}+C(\nu)((1+t)^{2m-5}+(1+t)^{2m-5-2\gamma_3})Z^4\\
&+C_0\epsilon^{\frac{1}{4}}(1+t)^{2m-n-3.5}Z^3\\
&+C(1+t)^{n-2.5}Z^3+C_0(\nu)\big((1+t)^{2m-2n-4}+(1+t)^{-2}\big)Z^2.\\
\end{split}\end{equation}
By Proposition \ref{p1},  in order to obtain the uniform estimates on $Z$, one needs
$$
2m-2n-2=2n-2m=-1,
$$
which  is equivalent to
$$
d=(1+t)^{-1}d^*(m)=(1+t)^{-1}\big(\frac{3\delta}{2}-3+m-\frac{1}{4}M_3\big).
$$
Choose $\epsilon^*=\frac{1}{2}\min\Big\{\frac{3\gamma-3}{2},\frac{-M_2-1}{2},\frac{1}{10}\Big\}>0.$
Then it follows from \eqref{E:gz73} that
\begin{equation}\label{E:gz74}\begin{split}
\frac{1}{2}\frac{d}{dt}Z^2&+\frac{b_{m,n}}{1+t}(Z^2-\B^2)+\eta\alpha(\A^*)^2+(1-\eta)\frac{d^*(m)}{1+t}\B^2\\
&+\frac{\eta}{1+t}\Big(d^*-\frac{1}{4}M_3\frac{\alpha}{\alpha+\beta}\Big)\B^2\\
\leq& \nu(\A^*)^{2}+C(\nu)((1+t)^{2m-5+\epsilon^*}+(1+t)^{2m-5-2\gamma_3})Z^4\\
&+C_0(\nu)(1+t)^{-1-\epsilon^*}Z^2.
\end{split}\end{equation}

\textbf{Step 3}.
On one hand, for fixed $\delta>1,$ a necessary condition to guarantee that the  set
$$
\Pi=\{(\alpha,\beta)|\alpha>0,\ \  2\alpha+3\beta>0, \ \  M_2<-1\}
$$
is not empty is that $M_1<\frac{3}{2}-\frac{1}{\delta}$ (see Remark \ref{necessarym1}).

On other hand, for fixed $\alpha,\beta,\gamma,\delta,$ there always exists a sufficiently large
number $m$ such that $d^*(m)>0.$ Indeed, one needs that
$$
d^*(m)=\frac{3\delta}{2}-3+m-\frac{1}{4}M_3=-\frac{1}{2}M_2+m-\frac{5}{2}>0.
$$
Here for the choice $m=n+0.5=3,$ one has therefore $d^*=\frac{3\delta}{2}-\frac{1}{4}M_3$.

\textbf{Step 4}.
Due to
$$d^*=\frac{3\delta}{2}-\frac{1}{4}M_3>1 \quad \text{and} \quad     M_4=\epsilon^*+M_2<-1,$$
so for $\eta^*=\min\Big\{\frac{3\gamma-3}{4(3\gamma-1)}, \frac{-M_4-1}{6\delta-M_3}, \frac{1}{20}\Big\}>0$,  and $\eta=\min\Big\{\frac{1}{200}, \frac{4d^*\eta^*(\alpha+\beta)}{M_3\alpha}\Big\}$, it hold that
\begin{equation}\label{keyinformation}\begin{split}
1+\epsilon^*-2(1-\eta^*)d^*=M_4+2\eta^* d^*<&-1,\\
 (1-\eta)d^*+\eta\Big(d^*-\frac{1}{4}M_3\frac{\alpha}{\alpha+\beta}\Big)\geq & (1-\eta^*)d^*.
\end{split}
\end{equation}
Based on this observation, we choose $\nu=\frac{\eta \alpha}{100}$, which, together
with \eqref{E:gz69}, $m=n+0.5=3$ and \eqref{GAB}, implies that
\begin{equation}\label{E:gz75}\begin{split}
\frac{1}{2}\frac{d}{dt}Z^2&+\frac{(1-\eta^*)b_*}{1+t}Z^2+\frac{\eta\alpha}{2}(1+t)^{2\gamma_3}\epsilon|\varphi\nabla^4 v|^2_2\\
&+\alpha \epsilon\sum_{k=0}^2(1+t)^{2\gamma_k}
|\varphi\nabla^{k+1} v|_2^2+(\alpha+\beta)\epsilon\sum_{k=0}^2(1+t)^{2\gamma_k}
|\varphi\nabla\text{div} v|_2^2\\
\leq& C(1+t)^{1+\epsilon^*}Z^4+C_0(1+t)^{-1-\epsilon^*}Z^2,
\end{split}\end{equation}
where $b_*$ is defined in \eqref{Aeq:2.6q}.    Therefore,
\begin{equation}\label{E:gz76}\begin{split}
\frac{d}{dt}Z&+\frac{(1-\eta^*)b_*}{1+t}Z
\leq C(1+t)^{1+\epsilon^*}Z^3+C_0(1+t)^{-1-\epsilon^*}Z.\\
\end{split}\end{equation}
According to Proposition \ref{P:xinzhu1} and
$$
\epsilon^*>0,\quad 1+\epsilon^*-2(1-\eta^*)b_*<-1,
$$
then $Z(t)$ satisfies
$$
Z(t)\leq \frac{(1+t)^{-(1-\eta^*)b_*\exp\Big(-\frac{C_0}{\epsilon^*}\Big((1+t)^{-\epsilon^*}-1\Big)\Big)}}
{\big(Z_0^{-2}-2C\int_0^t(1+s)^\xi\exp\Big(-\frac{2C_0}{\epsilon^*}\Big((1+t)^{-\epsilon^*}-1)\Big)\text{d}s\Big)^{1/2}},
$$
where $\xi=1+\epsilon^*-2(1-\eta^*)b_*<-1.$

Moreover, $Z(t)$ is globally well-defined for $t\geq 0$ if and if
$$
0<Z_0<\frac{1}{\big(2C\int_0^t(1+s)^\xi\exp\big(-2C_0/\epsilon^*((1+t)^{-\epsilon^*}-1)\big)\text{d}s\big)^{1/2}}.
$$
Moreover,
\begin{equation}\label{E:gz77}
Z(t)\leq C_0(1+t)^{-(1-\eta^*)b_*}\qquad \text{for}\ \text{all} \ \ t\geq 0,
\end{equation}
which implies that, for $t\geq 0$,
\begin{equation}\label{yu}
\begin{split}
Y_k(t)\leq & C_0(1+t)^{-\gamma_k-(1-\eta^*)b_*},\\[2pt]
U_k(t)\leq& C_0(1+t)^{-\delta_k-(1-\eta^*)b_*}.
\end{split}
\end{equation}

Finally, it follows from \eqref{E:gz75} and \eqref{E:gz77} that for $k=0,1,2,3$
\begin{equation}\label{yuend}
\begin{split}
\sum_{k=1}^3\epsilon\int_0^t(1+t)^{2\gamma_k} |\varphi\nabla^{k+1}v|^2_2\ \text{\rm d}s\leq C_0.
\end{split}
\end{equation}

\end{proof}

\begin{remark}\label{necessarym1}
Denote 
$
x=(\delta-1)^2/(4(2\alpha+\beta))
$.
Then $M_2<-1$ can be rewritten as 
$$
F(x)=x+\frac{1}{x}\delta^2M^2_1+2M_1\delta-6\delta+4<0.
$$
Then, in order to make sure that the set $\Pi$ is not empty, at least one needs that the minimum value of the function $F(x)$:
$$F(\delta M_1)=4M_1\delta-6\delta+4<0,$$  which requires $M_1<\frac{3}{2}-\frac{1}{\delta}$. 

\end{remark}

\subsection{Uniform energy estimates under the condition $(P_2)$}

For this case, we first choose
$$2n-2m=-(1+\epsilon)<-1$$
 for sufficiently small constant
 $$0<\epsilon<\frac{1}{2}\min\big\{1,3\gamma-3\big\}.$$ Then it follows from Lemma \ref{L:3.5} that
\begin{lemma}\label{Ul2}
There exist some positive constants $b_{m,n}$, $C$ and $C_0$  such that
\begin{equation}\label{ULeq:2.14}
\begin{split}
&\frac{1}{2}\frac{d}{dt} Z^2+\frac{b_{m,n}}{1+t} Z^2+\frac{\alpha}{2}\epsilon \sum_{k=0}^3 (1+t)^{2\gamma_k} |\varphi \nabla^{k+1}v(t)|^2_2\\
 \leq& C(1+t)^{2n-4+\epsilon}Z^{4}+C_0(1+t)^{-1-\epsilon}Z^2,
\end{split}
\end{equation}
with  the constant
 $b_{m,n}$   given by (\ref{E:gz70}).
\end{lemma}
The proof of this lemma is routine based on the conclusion obtained in Lemmas 4.1-4.7,  and  thus omitted. The rest of the proof for  Theorem \ref{tha-compact} under the condition $(P_2)$ is similar to the one in Subsection 4.1.4 for the condition $(P_1)$.
The proof for Theorem \ref{tha-compact} under condition $ (P_3)$- $(P_4)$ follows quickly   from  the similar  arguments shown  in  Sections 4.1-4.2 under  the conditions $(P_1)$-$(P_2)$. Here
we omit the details.

\section{Uniform energy estimates for general  initial density}
\subsection{The proof for  Theorem \ref{tha}}
Now we give the proof for  Theorem \ref{tha} for general initial density. Let  $F(x)\in C^\infty_c(\mathbb{R}^3)$ be a  truncation function  satisfying
 \begin{equation}\label{eq:2.6-77A}
0\leq F(x) \leq 1, \quad \text{and} \quad F(x)=
 \begin{cases}
1 \;\qquad  \text{if} \ \ |x|\leq 1,\\[8pt]
0   \ \ \ \ \ \ \    \text{if} \ \   |x|\geq 2.
 \end{cases}
 \end{equation}
According to the proof for  Theorem \ref{tha-compact} (see \eqref{yu}-\eqref{yuend}),  for the initial data
$$
(\varphi^R_0,\phi^R_0, v^\epsilon_0)=(\varphi^\epsilon_0F(|x|/R), \phi^\epsilon_0F(|x|/R), 0),
$$
there exists the unique  global regular solution  $(\varphi^R, \phi^R,v^R)$ satisfying:
\begin{equation}\label{decay2}
\begin{cases}
\begin{split}
&|\nabla^k \phi^R(t)|_2+|\nabla^kv^R(t)|_2\leq  C_0(1+t)^{-(1-\eta^*)b_*+2.5-k}\quad \text{for} \quad k=0, 1,2,3;\\[10pt]
 & |\nabla^k \varphi^R(t)|_2\leq  C_0(1+t)^{-(1-\eta^*)b_*+3-k}\quad \text{for} \quad k=0, 1, 2;\\[10pt]
  &\epsilon^{\frac{1}{2}} |\nabla^3 \varphi^R(t)|_2\leq  C_0(1+t)^{-(1-\eta^*)b_*};\\[10pt]
&\epsilon^{\frac{1}{2}}\sum_{k=0}^3 \|(1+t)^{k-2.5}\varphi^R\nabla^{k+1} v^R\|_{L^2L^2_t}\leq  C_0,\\
\end{split}
\end{cases}
\end{equation}
for any  time $t\geq 0$ and one constant  $C_0>0$ that is   independent of $R$ and $\epsilon$.

Due to  (\ref{decay2}) and the  following relations:
\begin{equation}\label{shijianfangxiang}
\begin{cases}
\displaystyle
\varphi^R_t=-(v^R+ \widehat{u}^\epsilon )\cdot \nabla\varphi^R-\frac{\delta-1}{2}\varphi^R\text{div} (v^R+ \widehat{u}^\epsilon ),\\[8pt]
\displaystyle
\phi^R_t=-(v^R+ \widehat{u}^\epsilon )\cdot \nabla \phi^R-\frac{\gamma-1}{2}\phi^R\text{div} (v^R+ \widehat{u}^\epsilon ),\\[8pt]
\displaystyle
v^R_t=-v^R\cdot \nabla v^R-\frac{\gamma-1}{2}\phi^R\nabla \phi^R-\epsilon (\varphi^R)^2 Lv^R\\[8pt]
\quad +\epsilon \nabla (\varphi^R)^2\cdot Q(v^R+ \widehat{u}^\epsilon )- \widehat{u}^\epsilon \cdot \nabla v^R-v^R\cdot \nabla  \widehat{u}^\epsilon -\epsilon (\varphi^R)^2 L\widehat{u}^\epsilon,
 \end{cases}
\end{equation}
 it holds that  for any fixed finite  constant  $R_0>0$ and finite time $T>0$,
\begin{equation}\label{uniformshijianN}
\begin{split}
\|\varphi^R_t\|_{H^2(B_{R_0})}+\|\phi^R_t\|_{H^2(B_{R_0})}+\|v^R_t\|_{H^1(B_{R_0})}+\int_0^t \|\nabla^2 v^R_t\|^2_{L^2(B_{R_0})}\text{d}s\leq C_0(R_0,T,\epsilon),
\end{split}
\end{equation}
for $0\leq t \leq T$, where the constant $C_0(R_0,T,\epsilon)>0$ depends only  on $C_0$, $\epsilon$,  $R_0$ and $T$, but is independent of $R$.

Since  (\ref{decay2}) and (\ref{uniformshijianN}) are  independent of  $R$,   there exists a subsequence of solutions (still denoted by) $(\varphi^R, \phi^R, v^R)$   converging to  a limit  $(\varphi^\epsilon,\phi^\epsilon,v^\epsilon)$  in the  sense:
\begin{equation}\label{strongjixian}
\begin{split}
&(\varphi^R, \phi^R,v^R)\rightarrow   (\varphi^\epsilon, \phi^\epsilon,v^\epsilon) \quad \text{strongly \ in } \ C([0,T];H^2(B_{R_0}))\quad \text{as} \quad R\rightarrow \infty.
\end{split}
\end{equation}

For $k=0,1,2,3$, denote
$$
a_k=-(1-\eta^*)b_*+2.5-k,\quad b_k=-(1-\eta^*)b_*+3-k,\quad c_k=k-2.5.
$$
Again due to  (\ref{decay2}), there exists a subsequence of solutions (still denoted by) $(\varphi^R, \phi^R, v^R)$ converging  to the same limit $(\varphi^\epsilon,\phi^\epsilon,v^\epsilon)$ as above in the following   weak* sense (for $k=0,1,2,3$):
\begin{equation}\label{ruojixianvv}
\begin{split}
(1+t)^{b_k}\varphi^R \rightharpoonup  (1+t)^{b_k}\varphi^\epsilon \quad \text{weakly* \ in } \ & L^\infty([0,T];H^3(\mathbb{R}^3)),\\
(1+t)^{a_k}\big( \phi^R, v^R\big) \rightharpoonup  (1+t)^{a_k}\big(\phi^\epsilon, v^\epsilon) \quad \text{weakly* \ in } \ & L^\infty([0,T];H^3(\mathbb{R}^3)).
\end{split}
\end{equation}

Combining  the strong convergence in (\ref{strongjixian}) and the weak convergence in (\ref{ruojixianvv}) shows that  $(\varphi^\epsilon, \phi^\epsilon,v^\epsilon) $ also satisfies the corresponding  estimates (\ref{decay2}) and (for $k=0,1,2,3$):
\begin{equation}\label{ruojixian1}
\begin{split}
(1+t)^{c_k}\varphi^R \nabla^{k+1} v^R \rightharpoonup (1+t)^{c_k} \varphi^\epsilon\nabla^{k+1} v^\epsilon \quad &\text{weakly \ in } \ L^2([0,T]\times \mathbb{R}^3).
\end{split}
\end{equation}


It is then obvious that $(\varphi^\epsilon, W^\epsilon) $ is a solution to problem (\ref{li47-1}) in the sense of distributions, and also satisfies the uniform estimate \eqref{uniformtime}.

\subsection{The proof for Corollary \ref{serre}. }

\subsubsection{Initial approximation}
In this subsection, we  give the following lemma for initial approximation:
\begin{lemma}\label{example}
If functions $(\rho_0,u_0)$ satisfy conditions $(\rm H_1)$-$(\rm H_2)$ shown in Corollary \ref{serre},
then for every $\epsilon \in (0,1]$,  we state that there exist functions $(\rho^\epsilon_0\geq 0,u^\epsilon_0)$ satisfying  the initial conditions $(A_1)$-$(A_2)$ shown  in Theorem \ref{tha},  and
\begin{equation}\label{5.1}
\begin{split}
 1+\|(\rho^\epsilon_0)^{\frac{\gamma-1}{2}}\|^2_{3}+\|u^\epsilon_0\|^2_{3}\leq &C_0,\\
 \|\nabla (\rho^\epsilon_0)^{\frac{\delta-1}{2}}\|^2_{2} +\epsilon|\nabla^3 (\rho^\epsilon_0)^{\frac{\delta-1}{2}}|^2_2 \leq &\Upsilon_1(\|(\rho^\epsilon_0)^{\frac{\gamma-1}{2}}\|_{3})+\Upsilon_2(\epsilon),
 \end{split}
\end{equation}
for some proper $\delta>1$ and some  constant $C_0>0$ that are both   independent of $\epsilon$, and $\Upsilon_i(y)$ {\rm ($i=1,2$)} are both  monotonically increasing continuous functions defined in $\mathbb{R}$ that are  independent of $\epsilon$ and satisfy $\Upsilon_i(0)=0$. Moreover, we still have
\begin{equation}\label{5.2}
\lim_{\epsilon \to 0}\Big\|\Big((\rho^{\epsilon}_0)^{\frac{\gamma-1}{2}}-\rho^{\frac{\gamma-1}{2}}_0, u^\epsilon_0
-u_0\Big)\Big\|_{3}=0.
\end{equation}
\end{lemma}

\begin{proof}
First, let $
f(x)=\frac{1}{1+|x|^{2a}}$. It is easy to see that if
$$
a>\max\Big\{\frac{3}{2(\gamma-1)},\frac{3}{2(\delta-1)}\Big\},
$$
then
$f^{\frac{\gamma-1}{2}}\in H^3(\R^3)$ and $f^{\frac{\delta-1}{2}}\in H^3(\R^3)$.

Second, we use the above function $f$  to construct the desired $\rho_0^\epsilon$.
\begin{lemma}\label{L:6.2} Let  $q$ and $p$ be any two positive constants, and

\begin{equation}\begin{split}\label{E:1}
(\rho_0^\epsilon)^{\frac{\gamma-1}{2}}= \rho_0^{\frac{\gamma-1}{2}}\chi_{\epsilon^q}+\eta \epsilon^pf^{\frac{\gamma-1}{2}},
\end{split}
\end{equation}
where $\eta>0$ is some sufficiently small constant, $\chi(x)\in C^\infty_c(\mathbb{R}^3)$ satisfying
$$
\chi(x)=\begin{cases}
1\quad   \ x\leq1,\\[5pt]
0,\quad  x\geq 2,
\end{cases}
$$
and $\chi_{\epsilon^q}(x)=\chi(\epsilon^q x)$.  Then it holds that
\begin{equation}\label{E:2}
\begin{split}
\big\|(\rho_0^\epsilon)^{\frac{\gamma-1}{2}}-\rho_0^{\frac{\gamma-1}{2}}\big\|_{3}
\rightarrow0 \quad \text{as}\quad \epsilon\rightarrow 0 \quad \text{for \ any } \quad \gamma>1.
\end{split}
\end{equation}
\end{lemma}
\begin{proof}
First, it is obvious that
$$
\big\|\rho_0^{\frac{\gamma-1}{2}}\big\|_{3}+\eta \epsilon^p \big\| f^{\frac{\gamma-1}{2}}\big\|_{3}\leq C_0,
$$
for some constant $C_0$ independent of $\epsilon$. Then, due to
\begin{equation}\label{boundness1}
|\nabla^k\chi_{\epsilon^q}|_{\infty}\leq \epsilon^{kq}|\nabla^k\chi|_{\infty}\leq C_0 \quad \text{for} \quad
k=0,1,2,3,
\end{equation}
one can obtain that
\begin{equation}\label{boundness}\big\|(\rho_0^\epsilon)^{\frac{\gamma-1}{2}}\big\|_{3}\leq C_0.
\end{equation}

Second, one has
 \begin{equation}\label{expansion2}
\begin{split}
\int_{ |x|\geq \frac{1}{\epsilon^q}} \big|\nabla^k\rho^{\frac{\gamma-1}{2}}_0 \big|^2\text{d}x\rightarrow 0\quad \text{as} \quad \epsilon \rightarrow 0 \quad \text{for} \quad k=&0,1,2,3,\end{split}
\end{equation}
which, together with \eqref{boundness1}-\eqref{boundness}, implies that
\begin{equation}\label{E:3}
\begin{split}\big\|(\rho_0^\epsilon)^{\frac{\gamma-1}{2}}-\rho_0^{\frac{\gamma-1}{2}}\big\|_{3}
&\leq \big\|\rho_0^{\frac{\gamma-1}{2}}\big\|_{H^3(|x|\geq \frac{1}{\epsilon^q} )}+\eta \epsilon^p\big\| f^{\frac{\gamma-1}{2}}\big\|_{3}\rightarrow 0 \quad \text{as} \quad \epsilon \rightarrow 0.
\end{split}
\end{equation}

\end{proof}

For simplicity,  in the rest of this section we denote
$$e=\frac{\delta-1}{\gamma-1},\quad B_{\epsilon,q}=B_{\frac{2}{\epsilon^q}}\quad \text{and} \quad B^C_{\epsilon,q}=\mathbb{R}^3/B_{\frac{2}{\epsilon^q}}.$$
Based on the conclusion of Lemma \ref{L:6.2}, now we have
\begin{lemma}\label{L:8.3q}      If one of  the following assumptions holds:
\begin{itemize}

\item  $e=1$, \text{or} $e=2$, \text{or} $e\geq 3$;\\
\item $2<e<3$ \text{and}  $
\frac{1}{2}-(3-e)(p+aq(\gamma-1))\geq 0
$,
\end{itemize}
then one has
\begin{equation}\label{finalbound}
\|\nabla (\rho^\epsilon_0)^{\frac{\delta-1}{2}}\|^2_{2} +\epsilon|\nabla^3 (\rho^\epsilon_0)^{\frac{\delta-1}{2}}|^2_2 \leq \Upsilon_1(\|(\rho^\epsilon_0)^{\frac{\gamma-1}{2}}\|_{3})+\Upsilon_2(\epsilon),
\end{equation}
for some two monotonically increasing continuous functions $\Upsilon_i(y)$ {\rm ($i=1,2$)} defined in $\mathbb{R}$ that are both   independent of $\epsilon$ and satisfy $\Upsilon_i(0)=0$.
\end{lemma}
\begin{proof}
On the one hand,  let
$$
\varphi^\epsilon_0=(\rho_0^\epsilon)^{\frac{\delta-1}{2}}=\big((\rho_0^\epsilon)^{\frac{\gamma-1}{2}}\big)^e
=\big(\rho_0^{\frac{\gamma-1}{2}}\chi_{\epsilon^q}+\eta \epsilon^pf^{\frac{\gamma-1}{2}}\big)^{e-1}(\rho_0^\epsilon)^{\frac{\gamma-1}{2}}.
$$
Under the assumption that $e=1$, \text{or} $e=2$, \text{or} $e\geq 3$, it is easy to show that \eqref{finalbound} holds.

On the other hand, under the assumption that $2<e<3$, it is  easy to see that
we only need to consider the bound of $\epsilon |\nabla^3 \varphi^\epsilon_0|^2_{2}$. First,  one has
\begin{equation}\label{insidebound}
\epsilon\|\nabla^3 \varphi^\epsilon_0\|_{L^2(B^C_{\epsilon,q})}=\eta^e \epsilon\big\|\epsilon^{ep}\nabla^3 f^{\frac{\delta-1}{2}}
\big\|_{L^2(B^C_{\epsilon,q})}\leq C\eta^e \epsilon^{ep+1}.
\end{equation}
Second, it follows from the direct calculation that
\begin{equation*}\begin{split}
\nabla^3 \varphi^\epsilon_0=&e(e-1)(e-2)\big((\rho_0^\epsilon)^{\frac{\gamma-1}{2}}\big)^{e-3}\Big(\nabla(\rho_0^\epsilon)^{\frac{\gamma-1}{2}}
\Big)^3\\
&+3e(e-1)\big((\rho_0^\epsilon)^{\frac{\gamma-1}{2}}\big)^{e-2}\nabla(\rho_0^\epsilon)^{\frac{\gamma-1}{2}}
\nabla^2(\rho_0^\epsilon)^{\frac{\gamma-1}{2}}\\
&+e\big((\rho_0^\epsilon)^{\frac{\gamma-1}{2}}\big)^{e-1}\nabla^3
(\rho_0^\epsilon)^{\frac{\gamma-1}{2}}.\\
\end{split}\end{equation*}
Then for $2< e<3$,  under the assumption that  $
\frac{1}{2}-(3-e)(p+aq(\gamma-1))\geq 0
$,
one has
\begin{equation}\label{E:9.1}\begin{split}
\epsilon^{\frac{1}{2}}\|\nabla^3 \varphi^\epsilon_0\|_{L^2(B_{\epsilon,q})}
\leq& C\epsilon^{\frac{1}{2}}\Big(\eta \epsilon^p\Big(\frac{\epsilon^{2aq}}{2^{2a}+\epsilon^{2aq}}
\Big)^{\frac{\gamma-1}{2}}\Big)^{e-3}\big|
\nabla \big(\rho_0^\epsilon\big)^{\frac{\gamma-1}{2}}\big|^2_{\infty}  \Big\|
\nabla\big(\rho_0^\epsilon\big)^{\frac{\gamma-1}{2}}\Big\|_{L^2(B_{\epsilon,q})}\\
&+C\epsilon^{\frac{1}{2}}(1+\eta\epsilon^p)^{e-2}   \big|
\nabla \big(\rho_0^\epsilon\big)^{\frac{\gamma-1}{2}}\big|_{\infty}      \big|
\nabla^2\big(\rho_0^\epsilon\big)^{\frac{\gamma-1}{2}}\big|_{2}\\
&+C\epsilon^{\frac{1}{2}}(1+\eta\epsilon^p)^{e-1}\big|
\nabla^3\big(\rho_0^\epsilon\big)^{\frac{\gamma-1}{2}}\big|_{2}\\
\leq & C(\big\|(\rho_0^\epsilon)^{\frac{\gamma-1}{2}}\big\|_{3}+\big\|(\rho_0^\epsilon)^{\frac{\gamma-1}{2}}\big\|^3_{3}),
\end{split}\end{equation}
for some constant $C$ independent of $\epsilon$, which, together with \eqref{insidebound}, implies  \eqref{finalbound}.

\end{proof}

At last, we just take $u^\epsilon_0(x)=u_0(x)$ for $x\in \mathbb{R}^3$. Then according to Lemmas \ref{L:6.2}-\ref{L:8.3q}, the proof for  Lemma \ref{example} is completed.
\end{proof}

\subsubsection{The proof for Corollary \ref{serre}} Now we give the proof for  Corollary \ref{serre}.
\begin{proof} We divide our proof into three steps: existence, uniqueness and time continuity.

 \textbf{Step 1:} Existence of the regular solution.
First, according to Lemma \ref{example}, we know that  for every $\epsilon \in (0,1]$,  there exists functions $(\rho^\epsilon_0,u^\epsilon_0)$ satisfying  the initial conditions $(A_1)$-$(A_2)$ for some fixed constant $\delta_0=\delta_0(\rho_0,u_0)>1$ (independent of $\epsilon$) in Theorem \ref{tha},  and \eqref{5.1}-\eqref{5.2}. Here, we take $u^\epsilon_0(x)=u_0(x)$ for $x\in \mathbb{R}^3$.
Then,  according to Theorem \ref{tha}, for such fixed $\delta_0$ and any $(\alpha,\beta)$ satisfying \eqref{canshu} and any one of  conditions $(P_1)$-$(P_4)$, we know that for any $T>0$ and $\epsilon\in (0,1]$,  there exists a   unique  regular solution $(\rho^\epsilon, u^\epsilon)$ in $[0,T]\times \mathbb{R}^3$  to the  Cauchy problem  (\ref{eq:1.1})-(\ref{10000}) with (\ref{initial})-(\ref{far}) satisfying
$$
(\rho^\epsilon,u^\epsilon)|_{t=0}=(\rho^\epsilon_0(x)\geq 0,  u_0(x)) \quad  \text{for} \quad   x\in \mathbb{R}^3,
$$
  and  the  uniform estimates  \eqref{uniformtime}. We denote
  \begin{equation*}
(\varphi^\epsilon, W^\epsilon=(\phi^\epsilon, v^\epsilon=u^\epsilon- \widehat{u}))=\left((\rho^\epsilon)^{\frac{\delta-1}{2}},\sqrt{\frac{4A\gamma}{(\gamma-1)^2}}(\rho^\epsilon)^{\frac{\gamma-1}{2}}, u^\epsilon- \widehat{u}\right).
\end{equation*}

Second,  for any bounded smooth domain $\Omega\subset \mathbb{R}^3$, due to the uniform estimates  \eqref{uniformtime} and  the Aubin-Lions Lemma  (see \cite{jm}) (i.e., Lemma \ref{aubin}), there exists a subsequence (still denoted by $( \phi^\epsilon,  u^{ \epsilon})$ and limit functions $(\phi^*,u^*)$ satisfying
\begin{equation}\label{ertfinalq}\begin{split}
(\phi^{ \epsilon},  u^{ \epsilon})\rightarrow& (\phi^*,u^*) \quad \text{in } \ C([0,T];H^2(\Omega)).
\end{split}
\end{equation}
Moreover,  there also exists a subsequence (still denoted by $(\phi^{ \epsilon},  u^{\epsilon})$) converging  to the  limit $( \phi^*,u^*)$ in the following  weak* sense:
\begin{equation}\label{ruojixianass}
\begin{split}
(\phi^{\epsilon}, v^\epsilon= u^{ \epsilon}- \widehat{u})\rightharpoonup  (\phi^*, v^*=u^*- \widehat{u}) \quad \text{weakly* \ in } \ L^\infty([0,T];H^3).
\end{split}
\end{equation}
From the lower semi-continuity  of norms in  the weak convergence, it is obvious that $( \phi^*,v^*)$ also satisfies the following   estimates
\begin{equation}\label{Euleruniformtime}\begin{split}
\sum_{k=0}^3 (1+t)^{2(k-2.5)+(1-\eta^*)b_*}\Big(\big|\nabla^k \phi^* (t,\cdot)\big|^2_2+\big|\nabla^k v^*(t,\cdot)\big|^2_2\Big)\leq C,
\end{split}
\end{equation}
for some constant $C=C(\gamma, A,  \kappa, \rho_0, u_0)>0$ that is independent of $T$.

Now we  are going to show that $(\phi^*,v^*)$ is the unique  regular  solution  to the  Cauchy problem  \eqref{eq:1.1E} with (\ref{winitial})-(\ref{wfar}).
First,  multiplying the equations in  $(\ref{li47-1})$ for the time evolution of $v^\epsilon$ by a test function  $h(t,x)=(h^1,h^2,h^3)\in C^\infty_c ([0,T)\times \mathbb{R}^3)$ on both sides, and then  integrating over $[0,T]\times \mathbb{R}^3$, one has
\begin{equation}\label{zhenzheng1q}
\begin{split}
&\int_0^T \int_{\mathbb{R}^3} \Big(v^{ \epsilon} \cdot h_t - (v^{\epsilon} \cdot \nabla) v^{\epsilon} \cdot h +\frac{\gamma-1}{4}\big(\phi^{\epsilon}\big)^2\text{div}h \Big)\text{d}x\text{d}t\\
=&\int_0^T \int_{\mathbb{R}^3} \epsilon \Big( (\varphi^{\epsilon})^2 Lv^{\epsilon}-\nabla (\varphi^{\epsilon})^2 \cdot Q(u^{ \epsilon}) \Big)\cdot h \text{d}x\text{d}t\\
&+\int_0^T \int_{\mathbb{R}^3}  \Big( v^{\epsilon}\cdot \nabla \widehat{u}^\epsilon+\sum_{j=1}^3 (\widehat{u}^\epsilon)^{(j)}\partial_j v^\epsilon +\epsilon      (\varphi^\epsilon)^2 L  \widehat{u}^\epsilon \Big)\cdot h \text{d}x\text{d}t.
\end{split}
\end{equation}
Combining  the  uniform estimates obtained above, the strong convergence in (\ref{ertfinalq}),  and  the weak convergences in (\ref{ruojixianass}),  and  letting $\epsilon \rightarrow 0$ in (\ref{zhenzheng1q}),  one has
\begin{equation*}
\begin{split}
&\int_0^T \int_{\mathbb{R}^3}  \Big(v^* \cdot h_t - (v^* \cdot \nabla) v^* \cdot h +\frac{\gamma-1}{4}\big(\phi^* \big)^2\text{div}h \Big)\text{d}x\text{d}t\\
=
&\int_0^T \int_{\mathbb{R}^3}  \Big( v^*\cdot \nabla \widehat{u}+\sum_{j=1}^3 \widehat{u}^{(j)}\partial_j v^* \Big)\cdot h \text{d}x\text{d}t.
\end{split}
\end{equation*}
Similarly,  one can use the same argument to show that $( \phi^*,v^*)$ also satisfies the equation in  $(\ref{li47-1})$ for the time evolution of  $\rho^{\frac{\gamma-1}{2}}$ and the initial data in the sense of distributions.
So it is obvious that $W=(\phi^*, v^*=u^*- \widehat{u})$ satisfies  the following  Cauchy problem  in the sense of distributions,
\begin{equation}\label{cvli47-1}
\begin{cases}
\displaystyle
W_t+\sum_{j=1}^3A_j(W) \partial_j W=G(W,  \widehat{u}),\\[12pt]
W |_{t=0}=W_0=(\phi_0,0),\quad x\in \mathbb{R}^3,\\[12pt]
W=( \phi,  v)\rightarrow (0,0) \quad\quad   \text{as}\quad \quad  |x|\rightarrow \infty \quad \text{for} \quad  t\geq 0,
 \end{cases}
\end{equation}
where
\begin{equation} \label{cvli47-2}
\begin{split}
\displaystyle
A_j(W)=&\left(\begin{array}{cc}
(v^*)^{(j)}&\frac{\gamma-1}{2}\phi^* e_j\\[8pt]
\frac{\gamma-1}{2}\phi^* e_j^\top & (v^*)^{(j)}\mathbb{I}_3
\end{array}
\right),\quad j=1,2,3,\\[6pt]
G(W,  \widehat{u})=&-B(\nabla  \widehat{u},W)-\sum_{j=1}^3  \widehat{u}^{(j)}\partial_j W,\quad B(\nabla \widehat{u},W)=\left(\begin{array}{c}
\frac{\gamma-1}{2}\phi^* \text{div} \widehat{u} \\[10pt]
(v^* \cdot\nabla)  \widehat{u}
\end{array}
\right).
\end{split}
\end{equation}
Moreover, it is easy to check that $W$  satisfies the estimates \eqref{Euleruniformtime}
and  $$
 u^*_t+u^*\cdot\nabla u^* =0\quad  \text{as } \quad  \phi^*(t,x)=0.
 $$

 \textbf{Step 2:} Uniquness.  Let $W_1=(\phi_1,v_1)$ and $W_2=(\phi_2,v_2)$ be two  solutions  to  (\ref{cvli47-1}) satisfying the uniform a priori estimates (\ref{Euleruniformtime}). Set
$F_N=F(|x|/N)$, and
\begin{equation*}\begin{split}
& \overline{W}=(\overline{\phi},\overline{v})=(\phi_1-\phi_2,v_1-v_2),\quad   \overline{W}^N=\overline{W}F_N=(\overline{\phi}^N,\overline{v}^N),
\end{split}
\end{equation*}
then  $\overline{W}^N$ solves the  following problem
 \begin{equation}
\label{weiyixing}
\begin{cases}
\displaystyle
\ \  \ \overline{W}^N_t+\sum\limits_{j=1}^3A_j(W_1)\partial_{j} \overline{W}^N+\sum\limits_{j=1}^3A_j(\overline{W}^N)\partial_{j} W_{2} \\[10pt]
= G(\overline{W}^N, \widehat{u}) +\sum\limits_{j=1}^3\big(A_j(W_1) +\widehat{u}^{(j)}\mathbb{I}_4\big)\overline{W} \partial_{j} F_N ,\\[10pt]
 \ \ \ \  \overline{W}^N|_{t=0}=0 \quad \text{for} \quad x\in \mathbb{R}^3,\\[10pt]
\ \ \ \ \overline{W}^N \rightarrow 0 \quad \text{as } \quad |x|\rightarrow +\infty,\quad t\geq   0.
\end{cases}
\end{equation}

It is not hard to obtain that
\begin{equation}\label{weiyi1}
\begin{split}
\frac{d}{dt}|\overline{W}^N|^2_2
\leq  &C\Big(1+\|W_1\|^2_{3}+\| W_2\|^2_{3}+\| \widehat{u}\|^2_{ \Xi}\Big)|\overline{W}^N|^2_2+M_N,
\end{split}
\end{equation}
where the error term $M_N$ is  given and estimated  by
\begin{equation*}
\begin{split}
M_N=&\int_{N\leq |x|\leq 2N}\frac{1}{N} (|W_1|+|\widehat{u}|)|\overline{W}|^2 \text{d}x\\
\leq & C(|W_1|_\infty+|\nabla \widehat{u}|_\infty)\|\overline{W}\|^2_{L^2(\mathbb{R}^3/B_N)}
\end{split}
\end{equation*}
for $0\leq t \leq T$. Then
one can  have
\begin{equation}\label{jiewei}
\begin{cases}
\displaystyle
\frac{d}{dt}|\overline{W}^N(t)|^2_{2}\leq J(t)|\overline{W}^N(t)|^2_{2}+M_N(t),\\[10pt]
\displaystyle
\int_{0}^{t}(J(s)+M_N(s))\ \text{\rm d}s\leq C_0  \quad  \text{for} \quad 0\leq t\leq T,
\end{cases}
\end{equation}
where the constant   $C_0>0 $ is  independent of $N$.

It follows from
$$
\int_0^t  M_N \text{d}t\rightarrow 0 \quad \text{as} \quad N\rightarrow \infty,
$$
and  Gronwall's inequality that $\overline{\phi}=\overline{v}=0$.
Then the uniqueness is obtained.

\textbf{Step 3}. The time-continuity  can be obtained via the  same arguments as in  Majda \cite{amj} for hyperbolic conservation laws. Here we omit its details.

Thus  the proof of Corollary \ref{serre} is complete.

\end{proof}

\section{Global-in-time inviscid limit}
Based on the uniform  energy estimates obtained in Sections 4-5,  now in this section   we prove the global-in-time inviscid  limit  from  the regular solutions of the  degenerate  viscous flow  to that of  the  inviscid flow which  are stated  in Theorem \ref{th2} and Corollary \ref{th3A}.

\subsection{Reformulations of the corresponding Cauchy problems}
From this subsection, we start to give the proof for Theorem \ref{th3A}. Corollary \ref{th2} can be obtained easily from the conclusion of Theorem \ref{th3A}. 

First, for  the  initial data $( \rho^\epsilon_0, u_0)$ satisfying  the assumptions $(A_1)$-$(A_2)$,   it follows from  Theorem \ref{tha} that for any $T>0$,  there exists a   unique  regular solution $(\rho^\epsilon, u^\epsilon)$ in $[0,T]\times \mathbb{R}^3$  to the  Cauchy problem  (\ref{eq:1.1})-(\ref{10000}) with (\ref{initial})-(\ref{far}) satisfying
$$
(\rho^\epsilon,u^\epsilon)|_{t=0}=\big(\rho^\epsilon_0(x)\geq 0,  u_0(x)\big) \quad  \text{for} \quad   x\in \mathbb{R}^3,
$$
  and  the  uniform estimates  \eqref{uniformtime}. We denote
  \begin{equation*}
(\varphi^\epsilon, W^\epsilon=(\phi^\epsilon, v^\epsilon=u^\epsilon- \widehat{u}))=\left((\rho^\epsilon)^{\frac{\delta-1}{2}},\sqrt{\frac{4A\gamma}{(\gamma-1)^2}}(\rho^\epsilon)^{\frac{\gamma-1}{2}}, u^\epsilon- \widehat{u}\right),
\end{equation*}
and the Cauchy problem  (\ref{eq:1.1})-(\ref{10000}) with (\ref{initial})-(\ref{far}) can be reformulated  into
\begin{equation}\label{li47-1xx}
\begin{cases}
\displaystyle
\quad \varphi^\epsilon_t+ v^\epsilon \cdot\nabla\varphi^\epsilon+\frac{\delta-1}{2}\varphi^\epsilon\text{div} v^\epsilon=-\widehat{u} \cdot\nabla\varphi^\epsilon-\frac{\delta-1}{2}\varphi^\epsilon\text{div} \widehat{u},\\[10pt]
\displaystyle
\quad W^\epsilon_t+\sum_{j=1}^3A_j(W^\epsilon) \partial_j W^\epsilon+\epsilon (\varphi^\epsilon)^2\mathbb{{L}}(v^\epsilon)\\
=\epsilon \mathbb{{H}}(\varphi^\epsilon)  \cdot \mathbb{{Q}}(v^\epsilon+ \widehat{u})+G^*(W^\epsilon, \varphi^\epsilon,  \widehat{u}),\\[10pt]
\quad (\varphi^\epsilon,W^\epsilon)|_{t=0}=(\varphi^\epsilon_0,W^\epsilon_0)=(\varphi^\epsilon_0,\phi^\epsilon_0,0),\quad x\in \mathbb{R}^3,\\[10pt]
\quad (\varphi^\epsilon,W^\epsilon)=(\varphi^\epsilon, \phi^\epsilon,  v^\epsilon)\rightarrow (0,0,0) \quad\quad   \text{as}\quad \quad  |x|\rightarrow \infty \quad \text{for} \quad  t\geq 0.
 \end{cases}
\end{equation}

Second,  on the one hand,  from the initial assumption of Theorem \ref{tha}, one has 
$$ \rho^\epsilon_0\geq 0 \quad \text{and} \quad  \big\|(\rho^\epsilon_0)^{\frac{\gamma-1}{2}}\big\|_3 \leq D_0$$
 for some constant  $D_0=D_0(\gamma, \delta, \alpha,\beta,  A, \kappa, \|u^\epsilon_0\|_{\Xi})>0$ that   is independent of $\epsilon$, where $u^\epsilon_0=u_0$ satisfying the initial assumption $(\rm A_1)$.   Then one gets  there exists a subsequence of  (still denoted by) $\rho^\epsilon_0$, which    converges to  a limit function $0\leq \rho^* \in  H^3(\mathbb{R}^3)$ in the following  weak sense:
\begin{equation}\label{ruojixian}
\begin{split}
(\rho^\epsilon_0)^{\frac{\gamma-1}{2}}\rightharpoonup  \rho^* \quad &\text{weakly \ in } \ H^3(\mathbb{R}^3) \quad \text{as} \quad \epsilon \rightarrow 0.
\end{split}
\end{equation} 
 
 On the other hand,  from \eqref{initialrelation1},  there exists one  function  $\rho_0(x)\in L^2(\mathbb{R}^3)$   such that
$$
\lim_{\epsilon\rightarrow 0}\Big|(\rho^\epsilon_0)^{\frac{\gamma-1}{2}}-\rho^{\frac{\gamma-1}{2}}_0\Big|_{2}=0,
$$
which, along with \eqref{ruojixian}, implies that 
$$
0\leq \rho^{\frac{\gamma-1}{2}}_0=\rho^* \in  H^3(\mathbb{R}^3).$$
 Then one has that   the  initial data $( \rho_0, u_0)$ satisfying  the assumptions $(\rm H_1)$-$(\rm H_2)$ in Corollary \ref{tha}.     It follows from  Corollary  \ref{tha}  that for any $T>0$,  there exists a   unique  regular solution $(\rho, u)$ in $[0,T]\times \mathbb{R}^3$  to the  Cauchy problem  \eqref{eq:1.1E} with \eqref{winitial}-\eqref{wfar} satisfying
$$
(\rho,u)|_{t=0}=\big(\rho_0(x)\geq 0,  u_0(x)\big) \quad  \text{for} \quad   x\in \mathbb{R}^3,
$$
  and  the  uniform estimates  \eqref{Euniformtime}. We denote
  \begin{equation*}
W=(\phi, v=u- \widehat{u})=\left(\sqrt{\frac{4A\gamma}{(\gamma-1)^2}}\rho^{\frac{\gamma-1}{2}}, u- \widehat{u}\right),
\end{equation*}
and the Cauchy problem  (\ref{eq:1.1E}) with \eqref{winitial}-\eqref{wfar} can be reformulated  into \eqref{cvli47-1}.

At last, we denote
  \begin{equation*}
\overline{W}^\epsilon=W^\epsilon-W=(\overline{\phi}^\epsilon,  \overline{v}^\epsilon)=(\phi^\epsilon-\phi, v^\epsilon-v).
\end{equation*}
Then it follows from systems \eqref{li47-1xx} and \eqref{cvli47-1} that $\overline{W}^\epsilon$ satisfies the following problem
\begin{equation}\label{zheng6}
\begin{cases}
\begin{split}\displaystyle
&\overline{W}^\epsilon_t+\sum_{j=1}^3 A_j(W^\epsilon) \partial_j \overline{W}^\epsilon
+\epsilon(\varphi^\epsilon)^2\mathbb{L}(v^\epsilon)\\
\displaystyle
=&-\sum_{j=1}^3A_j(\overline{W}^\epsilon)\partial_j W+G(\overline{W}^\epsilon, \widehat{u})-
\epsilon D\big((\varphi^\epsilon)^2, \nabla^2  \widehat{u}\big)+\epsilon \mathbb{{H}}(\varphi^\epsilon)  \cdot \mathbb{{Q}}(v^\epsilon+ \widehat{u}),\\[10pt]
\displaystyle
&  \overline{W}^\epsilon|_{t=0}=\overline{W}^\epsilon_0=\big(\rho^\epsilon_0-\rho_0,0\big) \quad \text{for} \quad x\in \mathbb{R}^3,\\[10pt]
\displaystyle
&\overline{W}^\epsilon\rightarrow 0 \quad \text{as } \quad  |x|\rightarrow \infty \quad \text{for} \quad  t\geq 0.
\end{split}
\end{cases}
\end{equation}

Next we need to do a series of energy estimates for $\overline{W}^\epsilon$.
Before this, for regular solutions $(W^\epsilon,W)$ obtained above, we list some useful time weighted  estimates.
Denote 
$$\iota=(1-\eta^*)b_*>1,$$  and  also $\iota<2$ from \eqref{Aeq:2.6q}, where the definitions of $\eta^*$ and $b_*$ can be found in Lemma \ref{ll2}.
According to Lemma 4.1, Proposition \ref{p1}, \eqref{E:gz77} and the Gagliardo-Nirenberg inequality,  for some constant $C_0=C_0(A, \alpha, \beta, \gamma, \delta, \rho^\epsilon_0, \rho_0, u_0)>0$, it holds that for all $t\geq 0$,
\begin{equation}\label{E:gz77xx}
\begin{split}
Z(t)\leq C_0(1+t)^{-\iota},\quad |\nabla^k W^\epsilon|_2\leq &C_0(1+t)^{-(k-2.5+\iota)},\\[2pt] 
|W^\epsilon|_\infty\leq  C_0(1+t)^{1-\iota},\quad  |\nabla W^\epsilon|_\infty\leq & C_0(1+t)^{-\iota},\\[2pt]
 |\nabla W^\epsilon|_3\leq   C_0(1+t)^{1-\iota},\quad  |\nabla^2 W^\epsilon|_3\leq  & C_0(1+t)^{-\iota},\\[2pt]
     \Theta(k)|\nabla^k \varphi^\epsilon|_2\leq & C_0(1+t)^{-(k-3+\iota)},  \\[2pt] 
     |\varphi^\epsilon|_\infty\leq C_0(1+t)^{1.5-\iota},\quad |\nabla \varphi^\epsilon|_\infty\leq & C_0
\epsilon^{-\frac{1}{4}}(1+t)^{0.5-\iota},\\[2pt] 
 |\nabla\varphi^\epsilon|_3\leq C_0(1+t)^{1.5-\iota},\quad
 |\nabla^2\varphi^\epsilon|_3\leq & C_0\epsilon^{-\frac{1}{4}}(1+t)^{0.5-\iota}.
\end{split}\end{equation}
From Section 5, it is obvious  that the solution $W$ to the Cauchy problem \eqref{cvli47-1} also satisfies the same  estimates as the ones on $W^\epsilon$ shown  in the first three lines of  \eqref{E:gz77xx}.
For the Cauchy problem \eqref{zheng6},  actually one can obtain the following energy estimates:
\begin{theorem}\label{T:7.1}
If $(\varphi^\epsilon, W^\epsilon)$ and $W$ are the regular solutions to  the Cauchy problems \eqref{li47-1xx}  and
\eqref{cvli47-1} respectively, then for arbitrarily large time $T>0$,  one has
\begin{equation*}\begin{split}
|\overline{W}^\epsilon(t)|^2_2
\leq& (1+t)^C\Big(|\overline{W}^\epsilon_0|^2_2+C\epsilon^2\ln(1+t)\Big) \quad \text{when}\quad \iota=\frac{7}{4};\\
|\overline{W}^\epsilon(t)|^2_2
\leq& (1+t)^C\Big(|\overline{W}^\epsilon_0|^2_2+C\epsilon^2\big|(1+t)^{7-4\iota}-1\big|\Big)\quad \text{when} \quad \iota\neq \frac{7}{4},\\
|\overline{W}^\epsilon(t)|^2_{D^1}\leq& (1+t)^C\Big(|\overline{W}^\epsilon_0|^2_{D^1}+C\epsilon^2\ln(1+t)\Big) \quad \text{when}\quad \iota=\frac{5}{4};\\
|\overline{W}^\epsilon(t)|^2_{D^1}\leq& (1+t)^C\Big(|\overline{W}^\epsilon_0|^2_{D^1}+C\epsilon^2\big|(1+t)^{5-4\iota}-1\big|\Big)\quad \text{when} \quad \iota\neq \frac{5}{4},\\
|\overline{W}^\epsilon(t)|^2_{D^2}\leq& \exp\big( C((1+t)^{\frac{1}{2}}-1) \big) \Big(|\overline{W}^\epsilon_0|^2_{D^2}+C|\overline{W}^\epsilon_0|^2_{D^1}\int_0^t (1+s)^{-4+C}\ \text{\rm d}s\\
&+C\epsilon\big(1+\int_0^t (1+s)^{-4+C}\ln (1+s) \ \text{\rm d}s\big)\Big)\quad \text{when}\quad \iota=\frac{5}{4};\\
|\overline{W}^\epsilon(t)|^2_{D^2}\leq&  \exp\big( C((1+t)^{3-2\iota}-1) \big)\Big(|\overline{W}^\epsilon_0|^2_{D^2}+C|\overline{W}^\epsilon_0|^2_{D^1}\int_0^t (1+s)^{1-4\iota+C}\ \text{\rm d}s\\
&+C\epsilon\big(1+\int_0^t (1+s)^{1-4\iota+C}\big|(1+s)^{5-4\iota}-1\big| \ \text{\rm d}s\big)\Big) \\
&\text{when}\quad \iota \in \big(1,\frac{5}{4}\big)\cup \big(\frac{5}{4},\frac{3}{2} \big);\\
|\overline{W}^\epsilon(t)|^2_{D^2}\leq& (1+t)^C\Big(|\overline{W}^\epsilon_0|^2_{D^2}+C|\overline{W}^\epsilon_0|^2_{D^1}\int_0^t (1+s)^{1-4\iota+C}\ \text{\rm d}s+C\epsilon\\
&+C\epsilon\int_0^t (1+s)^{1-4\iota+C}\big|(1+s)^{5-4\iota}-1\big| \ \text{\rm d}s\Big)\quad  \text{when} \quad \iota\in \big[\frac{3}{2}, 2\big),
\end{split}\end{equation*}
for any  $t\in [0,T]$ and the  constant  $C=C(A, \alpha, \beta, \gamma, \delta, \rho^\epsilon_0, \rho_0, u_0)>0$.
\end{theorem}
In order to  prove this theorem, we first need to consider the case when the mass density is compactly supported.
\subsection{Compactly supported approximation}
Denote
$F_N=F(|x|/N)$ (see \eqref{eq:2.6-77A}), and
\begin{equation*}\begin{split}
  \overline{W}=\overline{W}^\epsilon F_N=(\overline{\phi}^N,\overline{v}^N).
\end{split}
\end{equation*}
It follows from \eqref{zheng6} that   $\overline{W}$ solves the  following problem
\begin{equation}\label{zheng7}
\begin{cases}
\begin{split}\displaystyle
&\overline{W}_t+\sum_{j=1}^3 A_j(W^\epsilon) \partial_j \overline{W}-\sum_{j=1}^3 A_j(W^\epsilon)\overline{W}^\epsilon \partial_j F_N \\
\displaystyle
=&-\sum_{j=1}^3A_j(\overline{W})\partial_j W-B( \nabla\widehat{u},\overline{W})-\sum_{j=1}^3\widehat{u}^{(j)}\partial_j\overline{W}+\sum_{j=1}^3 \widehat{u}^{(j)}\partial_j F_N\overline{W}^\epsilon\\
&-\epsilon F_N \Big((\varphi^\epsilon)^2\mathbb{L}(v^\epsilon)+D\big((\varphi^\epsilon)^2, \nabla^2  \widehat{u}\big)
-\mathbb{{H}}(\varphi^\epsilon)  \cdot \mathbb{{Q}}(v^\epsilon+ \widehat{u})\Big),\\[6pt]
\displaystyle
&  \overline{W}|_{t=0}=\overline{W}_0=\big((\rho^\epsilon_0-\rho_0)F_N,0\big) \quad \text{for} \quad x\in \mathbb{R}^3,\\[6pt]
\displaystyle
&\overline{W}\rightarrow 0 \quad \text{as }  \quad  |x|\rightarrow \infty \quad \text{for} \quad  t\geq 0.
\end{split}
\end{cases}
\end{equation}

\begin{lemma}\label{L:7.1} 
\begin{equation}\label{E:7.1}\begin{split}
|\overline{W}(t)|^2_2
\leq& (1+t)^C\Big(|\overline{W}_0|^2_2+C\epsilon^2\ln(1+t)+J^*_N\Big) \quad \text{when}\quad \iota=\frac{7}{4};\\
|\overline{W}(t)|^2_2
\leq& (1+t)^C\Big(|\overline{W}_0|^2_2+C\epsilon^2\big|(1+t)^{7-4\iota}-1\big|+J^*_N\Big)\quad \text{when} \quad \iota\neq \frac{7}{4},
\end{split}\end{equation}
for any $0\leq t \leq T$ and the  constant  $C=C(A, \alpha, \beta, \gamma, \delta, \rho^\epsilon_0, \rho_0, u_0)>0$, and 
\begin{equation*}\begin{split}
J^*_N=&\int_0^T C\|\overline{W}\|_{L^2(\R^3/B_N)}\|\overline{W}^\epsilon\|_{L^2(\R^3/B_N)}\big|\nabla\widehat{u}|_\infty \ \text{\rm d}t+C(T)N^{-1},
\end{split}
\end{equation*}
for the  constant  $C(T)=C(A, \alpha,  \beta, \gamma, \delta, \rho^\epsilon_0, \rho_0, u_0,T)>0$ that is independent of $N$.
\end{lemma}
\begin{proof}
Multiplying $\eqref{zheng7}$ by $2\overline{W}$ on both sides and integrating over $\R^3$, then one has
\begin{equation*}
\begin{split}
\displaystyle
\frac{d}{dt} &\int |\overline{W}|^2+2\sum_{j=1}^3\int\overline{W}^\top A_j({W}^\epsilon) \partial_j \overline{W}
+2\sum_{j=1}^3\int \overline{W}^\top A_j(\overline{W}) \partial_j {W}\\
=&-2\int B(\nabla\widehat{u},\overline{W}) \cdot  \overline{W}+\sum_{j=1}^3\int(\overline{W})^\top\partial_j\widehat{u}^{(j)}\overline{W}\\
&-2\epsilon F_N\int\Big((\varphi^\epsilon)^2 \mathbb{L}(v^\epsilon)+D\big((\varphi^\epsilon)^2, \nabla^2  \widehat{u}\big)- \mathbb{{H}}(\varphi^\epsilon)  \cdot \mathbb{{Q}}(v^\epsilon+ \widehat{u})\Big) \cdot  \overline{W}\\
&+2\sum_{j=1}^3\int\overline{W}^\top A_j(W^\epsilon)\overline{W}^\epsilon \partial_j F_N+
2\sum_{j=1}^3\int\overline{W}^\top \widehat{u}^{(j)}\partial_j F_N\overline{W}^\epsilon.
\end{split}
\end{equation*}
It follows from the  integration by parts and H\"older's inequality that
\begin{equation}\label{E:7.2abca}
\begin{split}
\displaystyle
\frac{d}{dt}|\overline{W}|^2_2
\leq&C\big(|\nabla W^\epsilon|_\infty+|\nabla W|_\infty+|\nabla \widehat{u}|_\infty\big)|\overline{W}|^2_2
+C\epsilon \big(|\varphi^\epsilon|^2_\infty(
|\nabla^2 v^\epsilon|_2\\
&+|\nabla^2 \widehat{u}|_2)
+|\varphi^\epsilon|_\infty|\nabla \varphi^\epsilon|_2(|\nabla v^\epsilon|_\infty+|\nabla \widehat{u}|_\infty)\big)|\overline{W}|_2+J_N^1+J_N^2\\[2pt]
\leq & C(1+t)^{-1}|\overline{W}|^2_2+C\epsilon  (1+t)^{2.5-2\iota}|\overline{W}|_2+J_N^1+J_N^2,\\
\end{split}
\end{equation}
where one has used the fact $1<\iota <2$. Noticing that 
\begin{equation}\label{E:7.2bcd}|\widehat{u}|\leq 2N|\nabla\widehat{u}|,\quad |\nabla^k F_N(|x|/N)|\leq CN^{-k},\quad N<|x|<2N, \quad
k=0,1,2,3,
\end{equation}
then  $J_N^1$ and $J_N^2$ can be estimated as follows:
\begin{equation}
\label{E:7.2bc}\begin{split}
J_N^1\doteq&2\sum_{j=1}^3\int_{N<|x|<2N}\overline{W}^\top A_j(W^\epsilon)\overline{W}^\epsilon\partial_jF_N  \ \text{\rm d}x\\
\leq &C|\nabla F_N|_\infty|\overline{W}|_2|\overline{W}^\epsilon|_2\big|W^\epsilon|_\infty
\leq CN^{-1}(1+t)^{3.5-2\iota}|\overline{W}|_2,\\
J_N^2\doteq&2\sum_{j=1}^3\int_{N<|x|<2N}\overline{W}^\top\widehat{u}^{(j)}
\partial_jF_N\overline{W}^\epsilon  \ \text{\rm d}x \\
\leq& C\|\widehat{u} \nabla F_N\|_{L^\infty(N<|x|<2N)}\|\overline{W}\|_{L^2(\R^3/B_N)}\|\overline{W}^\epsilon\|_{L^2(\R^3/B_N)}\\[2pt]
\leq& C\|\overline{W}\|_{L^2(\R^3/B_N)}\|\overline{W}^\epsilon\|_{L^2(\R^3/B_N)}\big|\nabla\widehat{u}|_\infty=J_N,\\[2pt]
\end{split}
\end{equation}
where one has  used the fact 
$$|\overline{W}^\epsilon|_2=|W^\epsilon-W|_2\leq \big(|W^\epsilon|_2+|W|_2\big).$$

Then according to \eqref{E:7.2abca}-\eqref{E:7.2bc}, one has
\begin{equation*}\label{E:7.2abc}
\begin{split}
\displaystyle
\frac{d}{dt}|\overline{W}|^2_2
\leq & C(1+t)^{-1}|\overline{W}|^2_2+C\epsilon
(1+t)^{2.5-2\iota}|\overline{W}|_2\\
&+N^{-1}(1+t)^{3.5-2\iota}|\overline{W}|_2+J_N\\
\leq  &C(1+t)^{-1}|\overline{W}|^2_2+C\epsilon^2
(1+t)^{6-4\iota}+N^{-2}(1+t)^{8-4\iota}+J_N.\end{split}
\end{equation*}

If $\iota=\frac{7}{4}$ that means $6-4\iota=-1$,
according to   Gronwall's inequality, one has
\begin{equation*}\begin{split}|\overline{W}(t)|^2_2\leq& (1+t)^C\Big(|\overline{W}_0|^2_2+C\epsilon^2\ln(1+t)+\frac{1}{2}N^{-2}\big((1+t)^{2}-1\big)+ \tilde{J}_N\Big);
\end{split}\end{equation*}
if $\iota\neq\frac{7}{4}$,
according to Gronwall's inequality, one has
\begin{equation*}\begin{split}|\overline{W}(t)|^2_2\leq& (1+t)^C\Big(|\overline{W}_0|^2_2+\frac{C\epsilon^2}{7-4\iota}\big((1+t)^{7-4\iota}-1\big)+\frac{N^{-2}}{9-4\iota}\big((1+t)^{9-4\iota}-1\big)+\tilde{J}_N\Big),
\end{split}\end{equation*}
where  $ \tilde{J}_N=\int_0^t J_N \text{d}s$. Thus, (\ref{E:7.1}) follows immediately.
 \end{proof}
Now we consider  the estimate of $|\partial_x^\zeta \overline{W}|_2$ when $ |\zeta|=1$, $2$.  For simplicity, we denote $Q=Q(v^\epsilon+ \widehat{u})$ in the rest of this section.

\begin{lemma}\label{L:7.2} 
\begin{equation}\label{E:7.3}\begin{split}
|\overline{W}(t)|^2_{D^1}\leq& (1+t)^C\Big(|\overline{W}_0|^2_{D^1}+C\epsilon^2\ln(1+t)+ I^*_N\Big) \quad \text{when}\quad \iota=\frac{5}{4};\\
|\overline{W}(t)|^2_{D^1}\leq& (1+t)^C\Big(|\overline{W}_0|^2_{D^1}+C\epsilon^2\big|(1+t)^{5-4\iota}-1\big|+I^*_N\Big)\quad \text{when} \quad \iota\neq \frac{5}{4},
\end{split}\end{equation}
for any $0\leq t \leq T$, where the constant  $C=C(A, \alpha, \beta, \gamma, \delta, \rho^\epsilon_0, \rho_0, u_0)>0$, and 
\begin{equation*}\begin{split}
I^*_N=&\int_0^T C\|\nabla \overline{W}^\epsilon\|_{L^2(\R^3/B_N)}
|\nabla  \widehat{u}|_\infty\|\nabla\overline{W}\|_{L^2(\R^3/B_N)} \ \text{\rm d}t+C(T)N^{-1},
\end{split}
\end{equation*}
for the  constant  $C(T)=C(A, \alpha, \beta, \gamma, \delta, \rho^\epsilon_0, \rho_0, u_0,T)>0$ that is independent of $N$.

\end{lemma}
\begin{proof}
Applying the operator $\partial_x^\zeta$ on $\eqref{zheng7}_1$ leads to
\begin{equation}\label{E:7.4}
\begin{split}
&(\partial_x^\zeta \overline{W})_t+\sum_{j=1}^3A_j(W^\epsilon)\partial_j(\partial_x^\zeta \overline{W})-\sum_{j=1}^3\partial_x^\zeta\big(A_j(W^\epsilon)\overline{W}^\epsilon\partial_j F_N\big)\\
=&-\sum_{j=1}^3\left(\partial_x^\zeta (A_j(W^\epsilon)\partial_j\overline{W})
-A_j(W^\epsilon)\partial_j(\partial_x^\zeta \overline{W})\right)-\sum_{j=1}^3\partial_x^\zeta (A_j(\overline{W}))\partial_j W\\
&-\sum_{j=1}^3 \left(\partial_x^\zeta \big(A_j(\overline{W})\partial_j W\big)
-\partial_x^\zeta \big(A_j(\overline{W})\big)\partial_j  W\right)+\sum_{j=1}^3\partial_x^\zeta\big(\widehat{u}^{(j)}\partial_j F_N\overline{W}^\epsilon\big)\\
&-B(\nabla\widehat{u},\partial_x^\zeta \overline{W})-\big(\partial_x^\zeta B( \nabla\widehat{u},\overline{W})
-B(\nabla\widehat{u},\partial_x^\zeta\overline{W})\big)\\
&-\sum_{j=1}^3\widehat{u}^{(j)}\partial_j\partial_x^\zeta\overline{W}-\sum_{j=1}^3\Big(\partial_x^\zeta\big(\widehat{u}^{(j)}\partial_j\overline{W}\big)-
\widehat{u}^{(j)}\partial_j\partial_x^\zeta\overline{W}\Big)\\
&-\epsilon F_N(\varphi^\epsilon)^2\mathbb{L}(\partial_x^\zeta v^\epsilon)-\epsilon\partial_x^\zeta\big(F_ND((\varphi^\epsilon)^2,\nabla^2\widehat{u})\big)+\epsilon F_N\mathbb{H}(\varphi^\epsilon)\cdot\mathbb{Q}(\partial_x^\zeta (v^\epsilon+\widehat{u}))\\
&-\epsilon\left(\partial_x^\zeta\big(F_N(\varphi^\epsilon)^2
\mathbb{L}(v^\epsilon)\big)-F_N(\varphi^\epsilon)^2\mathbb{L}(\partial_x^\zeta v^\epsilon)\right)\\
&+\epsilon\left(\partial_x^\zeta\big(F_N\mathbb{H}(\varphi^\epsilon)\cdot\mathbb{Q}(v^\epsilon+\widehat{u})\big)-F_N\mathbb{H} (\varphi^\epsilon)
\cdot \partial_x^\zeta\mathbb{Q}(v^\epsilon+\widehat{u})\right).
\end{split}
\end{equation}
Then multiplying (\ref{E:7.4}) by $2\partial^\zeta_x \overline{W}$  and integrating  over $\mathbb{R}^3$ by parts,  one can obtain that
\begin{equation*}
\begin{split}
 \frac{d}{dt}|\partial^\zeta_x \overline{W}|_2^2
=&\sum_{j=1}^3\int (\partial^\zeta_x \overline{W})^\top \Big(\partial_jA_j (W^\epsilon)\partial^\zeta_x \overline{W}+2\partial_x^\zeta\big(A_j(W^\epsilon)\overline{W}^\epsilon\partial_j F_N\big)\Big)\\
&-2\sum_{j=1}^3\int\left(\partial_x^\zeta \left(A_j(W^\epsilon)\partial_j\overline{W}\right)-A_j(W^\epsilon)\partial_j(\partial_x^\zeta \overline{W})\right)\cdot\partial_x^\zeta \overline{W}\\
&-2\sum_{j=1}^3\int(\partial_x^\zeta \overline{W})^\top \Big(\partial_x^\zeta (A_j(\overline{W}))\partial_j W-\partial_x^\zeta\big(\widehat{u}^{(j)}\partial_j F_N\overline{W}^\epsilon\big)\Big)\\
&-2\sum_{j=1}^3 \int  \left(\partial_x^\zeta\big(A_j(\overline{W})\partial_j W \big) -\partial_x^\zeta A_j(\overline{W})\partial_j  W \right)\cdot\partial_x^\zeta \overline{W} \\
&-2\int \Big(B( \nabla\widehat{u},\partial_x^\zeta \overline{W})+\Big(\partial_x^\zeta B(\nabla\widehat{u},\overline{W})-B( \nabla\widehat{u},\partial_x^\zeta \overline{W})\Big)
\Big)\cdot \partial_x^\zeta \overline{W}\\
&+\sum_{j=1}^3\int \Big( \partial_j\widehat{u}^{(j)}|\partial_x^\zeta\overline{W}|^2-2(\partial_x^\zeta\overline{W})^\top\Big(\partial_x^\zeta\big(\widehat{u}^{(j)}\partial_j\overline{W}\big)-
\widehat{u}^{(j)}\partial_j\partial_x^\zeta\overline{W}\Big)\Big)\\
&-2 \epsilon\int \Big( F_N(\varphi^\epsilon)^2  L(\partial^\zeta_x v^\epsilon)+\partial_x^\zeta \big(F_ND((\varphi^\epsilon)^2,\nabla^2\widehat{u})\Big)\cdot\partial_x^\zeta\overline{v}\\
&+2\epsilon\int \Big(F_N\nabla (\varphi^\epsilon)^2 \cdot \partial^\zeta_x Q\Big) \cdot \partial^\zeta_x \overline{v}\\
&-2 \epsilon\int \Big( \partial^\zeta_x \big(F_N(\varphi^\epsilon)^2 Lv^\epsilon\big)-F_N(\varphi^\epsilon)^2 L\partial^\zeta_x v^\epsilon \Big)\cdot \partial^\zeta_x \overline{v}\\
\end{split}
\end{equation*}
\begin{equation}\label{gongshiding}
\begin{split}
&+2 \epsilon\int  \Big(\partial^\zeta_x\big(F_N\nabla (\varphi^\epsilon)^2  \cdot Q\big)-F_N\nabla (\varphi^\epsilon)^2  \cdot  \partial^\zeta_xQ\Big)\cdot \partial^\zeta_x \overline{v}
:=\sum_{i=1}^{15}I_i.\\
\end{split}
\end{equation}
When $|\zeta|=1$, according to the estimate \eqref{E:gz77xx}, one has
\begin{equation}\label{E:7.6}\begin{split}
I_1=&\sum_{j=1}^3\int (\partial^\zeta_x \overline{W})^\top\partial_jA_j (W^\epsilon)\partial^\zeta_x \overline{W}
\leq C|\nabla W^\epsilon|_\infty| \overline{W}|^2_{D^1}
\leq  C(1+t)^{-\iota}| \overline{W}|^2_{D^1},\\
I_3=&-2\sum_{j=1}^3\int\left(\partial_x^\zeta (A_j(W^\epsilon)\partial_j\overline{W})-A_j(W^\epsilon)\partial_j(\partial_x^\zeta \overline{W})\right)\cdot\partial_x^\zeta \overline{W}\\
\leq& C|\nabla W^\epsilon|_\infty|\overline{W}|^2_{D^1}\leq C(1+t)^{-\iota}|\overline{W}|^2_{D^1},\\
I_4=&-2\sum_{j=1}^3\int(\partial_x^\zeta \overline{W})^\top\partial_x^\zeta (A_j(\overline{W}))\partial_j W
\leq C|\nabla W|_\infty|\overline{W}|^2_{D^1}
\leq C(1+t)^{-\iota}|\overline{W}|^2_{D^1},\\
I_6=&- 2\sum_{j=1}^3 \int  \left(\partial_x^\zeta\big(A_j(\overline{W})\partial_j W \big) -\partial_x^\zeta (A_j(\overline{W}))\partial_j  W \right)\cdot\partial_x^\zeta \overline{W} \\
\leq &C|\nabla^2 W|_3|\nabla \overline{W}|_2|\overline{W}|_6
\leq C(1+t)^{-\iota}| \overline{W}|^2_{D^1},\\
I_{7}=&-2\int B(\nabla\widehat{u},\partial_x^\zeta \overline{W})\cdot\partial_x^\zeta \overline{W}
\leq C|\nabla\widehat{u}|_\infty|\nabla\overline{W}|_2^2
\leq C(1+t)^{-1}|\overline{W}|_{D^1}^2,\\
I_{8}=&-2\int\Big(\partial_x^\zeta B( \nabla\widehat{u},\overline{W})-B(\nabla\widehat{u},\partial_x^\zeta \overline{W})\Big)
\cdot\partial_x^\zeta \overline{W}\\
\leq& C|\nabla^2\widehat{u}|_3|\overline{W}|_6|\nabla\overline{W}|_2
\leq C(1+t)^{-2}|\overline{W}|_{D^1}^2,\\
I_{9}=&\sum_{j=1}^3\int \partial_j\widehat{u}^{(j)}|\partial_x^\zeta\overline{W}|^2
\leq C|\nabla\widehat{u}|_\infty|\nabla\overline{W}|_2^2
\leq C(1+t)^{-1}|\overline{W}|_{D^1}^2.
\end{split}\end{equation}
It follows from \eqref{E:7.2bcd} that
\begin{equation}\label{E:7.6aa}\begin{split}
I_2=&2\sum_{j=1}^3\int_{N<|x|<2N}(\partial_x^\zeta \overline{W})^\top\partial_x^\zeta
\big(A_j(W^\epsilon)\overline{W}^\epsilon\partial_j F_N\big)  \ \text{\rm d}x \\
\leq&C\Big(\big(|  W^\epsilon|_\infty|\nabla \overline{W}^\epsilon|_2
+|\nabla W^\epsilon|_2|\overline{W}^\epsilon|_\infty\big)|\nabla F_N|_\infty\\
&+
| W^\epsilon|_\infty|\overline{W}^\epsilon|_2|\nabla^2 F_N|_\infty\Big)|\nabla\overline{W}|_{2}
\leq  CN^{-1}(1+t)^{3.5-2\iota}|\overline{W}|_{D^1},\\
I_5=&2\sum_{j=1}^3\int_{N<|x|<2N}(\partial_x^\zeta \overline{W})^\top\partial_x^\zeta
\big(\widehat{u}^{(j)}\partial_j F_N\overline{W}^\epsilon\big)  \ \text{\rm d}x\\
\leq&C\Big(|  \nabla \widehat{u}|_\infty|\nabla F_N|_\infty|\overline{W}^\epsilon|_2+| \overline{W}^\epsilon|_2\|\widehat{u}^{(j)}\partial_j \partial_x^\zeta F_N\|_{L^\infty(N<|x|<2N)}\Big)|\nabla\overline{W}|_{2}\\
&+C\|\nabla \overline{W}^\epsilon\|_{L^2(\R^3/B_N)}\|\widehat{u}^{(j)}\partial_j F_N\|_{L^\infty(N<|x|<2N)}\|\nabla\overline{W}\|_{L^2(\R^3/B_N)}\\
\leq & CN^{-1}(1+t)^{1.5-\iota}|\overline{W}|_{D^1}+I_N,\\
\end{split}\end{equation}
where the term $I_N$ can be defined and estimated as follows
\begin{equation*}\label{E:7.6a}\begin{split}
I_N=&C\|\nabla \overline{W}^\epsilon\|_{L^2(\R^3/B_N)}
|\nabla  \widehat{u}|_\infty\|\nabla\overline{W}\|_{L^2(\R^3/B_N)}.
\end{split}\end{equation*}
Similarly, one gets
\begin{equation}\label{E:4.9b}\begin{split}
I_{10}=&-2\sum_{j=1}^3\int(\partial_x^\zeta\overline{W})^\top\Big(\partial_x^\zeta\big(\widehat{u}^{(j)}\partial_j\overline{W}
\big)-\widehat{u}^{(j)}\partial_j\partial_x^\zeta\overline{W}\Big)\\
\leq& C|\nabla\widehat{u}|_\infty|\nabla\overline{W}|_2^2
\leq C(1+t)^{-1}|\overline{W}|_{D^1}^2,\\
I_{11}=&-2\epsilon\int \Big( F_N(\varphi^\epsilon)^2 \cdot L(\partial^\zeta_x v^\epsilon)\Big) \cdot \partial^\zeta_x \overline{v}\\
\leq& C\epsilon |\varphi^\epsilon|^2_\infty| \nabla^3 v^\epsilon|_2|\nabla \overline{v}|_2
\leq C\epsilon(1+t)^{2.5-3\iota}|\nabla \overline{v}|_2,\\
I_{12}=&-2\epsilon\int\partial_x^\zeta \big(F_ND((\varphi^\epsilon)^2,\nabla^2\widehat{u})\big)\cdot\partial_x^\zeta\overline{v}\\
\leq& C\epsilon\big(|\varphi^\epsilon|_\infty|\nabla\varphi^\epsilon|_2|\nabla^2\widehat{u}|_\infty+
|\varphi^\epsilon|_\infty^2|\nabla^3\widehat{u}|_2+|\nabla F_N|_\infty|\varphi^\epsilon|^2_\infty|\nabla^2\widehat{u}|_2\big)|\nabla\overline{v}|_2\\
\leq& C\epsilon\big((1+t)^{0.5-2\iota}+N^{-1}(1+t)^{1.5-2\iota}\big)|\nabla\overline{v}|_2,\\
I_{13}=&2\epsilon\int \Big(F_N\nabla (\varphi^\epsilon)^2 \cdot Q(\partial^\zeta_x (v^\epsilon+\widehat{u}))\Big) \cdot \partial^\zeta_x \overline{v}\\
\leq& C\epsilon |\varphi^\epsilon|_\infty\big(|\nabla \varphi^\epsilon|_3|\nabla^2 v^\epsilon|_6+|\nabla \varphi^\epsilon|_2|\nabla^2\widehat{u}|_\infty)|\nabla \overline{v}|_2
\leq C\epsilon (1+t)^{2.5-3\iota}|\nabla \overline{v}|_2,\\
I_{14}= &-2\epsilon\int \Big( \partial^\zeta_x (F_N(\varphi^\epsilon)^2 Lv^\epsilon)-F_N(\varphi^\epsilon)^2 L\partial^\zeta_x v^\epsilon \Big)\cdot \partial^\zeta_x \overline{v}\\
\leq&C\epsilon ( |\varphi^\epsilon|_\infty|\nabla \varphi^\epsilon|_3|\nabla^2 v^\epsilon|_6+|\nabla F_N|_\infty|\varphi^\epsilon|^2_\infty|\nabla^2 v^\epsilon|_2)|\nabla \overline{v}|_2\\
\leq& C\epsilon\big((1+t)^{2.5-3\iota}+N^{-1}(1+t)^{3.5-3\iota}\big)|\nabla \overline{v}|_2,\\
I_{15}=&2\epsilon\int  \Big(\partial^\zeta_x(F_N\nabla (\varphi^\epsilon)^2  \cdot Q(v^\epsilon+\widehat{u}))-F_N\nabla (\varphi^\epsilon)^2  \cdot  Q(\partial^\zeta_x (v^\epsilon+\widehat{u}))\Big)\cdot \partial^\zeta_x \overline{v}\\
\leq& C\epsilon  \big(|\nabla\varphi^\epsilon|_6|\nabla\varphi^\epsilon|_3+|\varphi|_\infty|\nabla^2\varphi^\epsilon|_2\big)(|\nabla v^\epsilon|_\infty+|\nabla\widehat{u}|_\infty)|\nabla \overline{v}|_2\\
&+C|\nabla F_N|_\infty|\varphi^\epsilon|_\infty|\nabla\varphi^\epsilon|_2\big(|\nabla v^\epsilon|_\infty+|\nabla\widehat{u}|_\infty\big)|\nabla \overline{v}|_2\\
\leq&C\epsilon\Big((1+t)^{1.5-2\iota}+N^{-1}(1+t)^{2.5-2\iota}\Big)|\nabla \overline{v}|_2.
\end{split}\end{equation}

It follows from  \eqref{gongshiding}-\eqref{E:4.9b}  that 
\begin{equation*}\begin{split}
 \frac{d}{dt}| \overline{W}|_{D^1}^2\leq& C(1+t)^{-1}|\overline{W}|^2_{D^1}+C\big(\epsilon(1+t)^{1.5-2\iota}+N^{-1}(1+t)^{3.5-2\iota}\big)|\overline{W}|_{D^1}+I_N\\
 \leq & C(1+t)^{-1}|\overline{W}|^2_{D^1}+C\epsilon^2(1+t)^{4-4\iota}+C(T)N^{-1}+I_N,
  \end{split}\end{equation*}
which, along with   Gronwall's inequality,  implies that: 
when $\iota=\frac{5}{4}$ that means $4-4\iota=-1$,
\begin{equation*}\begin{split}|\overline{W}(t)|^2_{D^1}\leq& (1+t)^C\Big(|\overline{W}_0|^2_{D^1}+C\epsilon^2\ln(1+t)+C(T)N^{-1}+ \tilde{I}_N\Big);
\end{split}\end{equation*}
when $\iota\neq\frac{5}{4}$,
\begin{equation*}\begin{split}|\overline{W}(t)|^2_{D^1}\leq& (1+t)^C\Big(|\overline{W}_0|^2_{D^1}+\frac{C\epsilon^2}{5-4\iota}\big((1+t)^{5-4\iota}-1\big)+C(T)N^{-1}+\tilde{I}_N\Big),
\end{split}\end{equation*}
where  $ \tilde{I}_N=\int_0^t I_N \text{d}s$. Thus, (\ref{E:7.3}) follows immediately.

\end{proof}
When $|\zeta|=2$, one has
\begin{lemma}\label{L:7.3}
\begin{equation}\label{E:7.10}\begin{split}
|\overline{W}(t)|^2_{D^2}\leq& \exp\big( C((1+t)^{\frac{1}{2}}-1) \big) \Big(|\overline{W}_0|^2_{D^2}+C|\overline{W}_0|^2_{D^1}\int_0^t (1+s)^{-4+C}\ \text{\rm d}s\\
&+C\epsilon\big(1+\int_0^t (1+s)^{-4+C}\ln (1+s) \ \text{\rm d}s\big)+ L^*_N\Big)\quad \text{when}\quad \iota=\frac{5}{4};\\
|\overline{W}(t)|^2_{D^2}\leq&  \exp\big( C((1+t)^{3-2\iota}-1) \big)\Big(|\overline{W}_0|^2_{D^2}+C|\overline{W}_0|^2_{D^1}\int_0^t (1+s)^{1-4\iota+C}\ \text{\rm d}s\\
&+C\epsilon\big(1+\int_0^t (1+s)^{1-4\iota+C}\big|(1+s)^{5-4\iota}-1\big| \ \text{\rm d}s\big)+ L^*_N\Big) \\
&\text{when}\quad \iota \in \big(1,\frac{5}{4}\big)\cup \big(\frac{5}{4},\frac{3}{2} \big);\\
|\overline{W}(t)|^2_{D^2}\leq& (1+t)^C\Big(|\overline{W}_0|^2_{D^2}+C|\overline{W}_0|^2_{D^1}\int_0^t (1+s)^{1-4\iota+C}\ \text{\rm d}s+C\epsilon\\
&+C\epsilon\int_0^t (1+s)^{1-4\iota+C}\big|(1+s)^{5-4\iota}-1\big| \ \text{\rm d}s+L^*_N\Big)\quad  \text{when} \quad \iota\in \big[\frac{3}{2}, 2\big),
\end{split}\end{equation}
for any $0\leq t \leq T$, where the constant  $C=C(A, \alpha, \beta, \gamma, \delta, \rho^\epsilon_0, \rho_0, u_0)>0$, and 
\begin{equation*}\begin{split}
L^*_N=&\int_0^T \big(L_N +C(T)I^*_N\big)\ \text{\rm d}t+C(T)N^{-1}
\end{split}
\end{equation*}
for the  constant  $C(T)=C(A, \alpha, \beta, \gamma, \delta, \rho^\epsilon_0, \rho_0, u_0,T)>0$ that is independent of $N$.
\end{lemma}
\begin{proof}
When $|\zeta|=2,$  similarly to \eqref{E:7.6}-\eqref{E:7.6aa}, one has
\begin{equation}\label{E:7.11}\begin{split}
I_1+&I_4\leq C(1+t)^{-\iota}|\overline{W}|^2_{D^2},\\
I_2
\leq& C|\nabla^2\overline{W}|_2\Big(|\nabla F_N|_\infty\big(|\nabla^2W^\epsilon|_2|\overline{W}^\epsilon|_\infty
+|\nabla W^\epsilon|_3|\nabla\overline{ W}^\epsilon|_6+|W^\epsilon|_\infty|\nabla^2\overline{W}^\epsilon|_2\big)\\
&+|\nabla^2 F_N|_\infty\big(|\nabla W^\epsilon|_2|\overline{W}^\epsilon|_\infty
+|W^\epsilon|_\infty|\nabla\overline{W}^\epsilon|_2\big)+|\nabla^3 F_N|_\infty|W^\epsilon|_2|\overline{W}^\epsilon|_\infty\Big)\\
\leq&CN^{-1}(1+t)^{3.5-2\iota}|\nabla^2\overline{W}|_2,\\
 I_3
\leq&  C(|\nabla W^\epsilon|_\infty|\nabla^2\overline{W}|_2+|\nabla^2 W^\epsilon|_3|\nabla\overline{W}|_6)|\nabla^2\overline{W}|_2
\leq C(1+t)^{-\iota}|\overline{W}|^2_{D^2},\\
I_5
\leq& C|\nabla^2\overline{W}|_2\Big(|\nabla F_N|_\infty\big(|\nabla^2\widehat{u}|_2|\overline{W}^\epsilon|_\infty
+|\nabla \widehat{u}|_\infty|\nabla\overline{ W}^\epsilon|_2\big)\\
&+|\nabla^2 F_N|_\infty(  |\nabla \widehat{u}|_\infty|\overline{W}^\epsilon|_2+\|\widehat{u}\|_{L^\infty(N<|x|<2N)} |\nabla\overline{W}^\epsilon|_2 )\\
&+|\nabla^3 F_N|_\infty\|\widehat{u}\|_{L^\infty(N<|x|<2N)} |\overline{W}^\epsilon|_2\Big)+I_5^N\\
\leq &CN^{-1}(1+t)^{1.5-\iota}|\overline{W}|_{D^2}+I_5^N,\\
\end{split}\end{equation}
where $I_5^N$ can be defined and estimated as follows  
\begin{equation*}
\begin{split}
I_5^N=&C\|\widehat{u}\|_{L^\infty(N<|x|<2N)}\|\nabla^2\overline{W}\|_{L^2(\R^3/B_N)}\|\nabla^2\overline{ W}^\epsilon\|_{L^2(\R^3/B_N)}|\nabla F_N|_\infty\\
\leq& C\|\nabla^2 \overline{ W}^\epsilon\|_{L^2(\R^3/B_N)}|\nabla\widehat{u}|_\infty\|\nabla^2\overline{W}\|_{L^2(\R^3/B_N)}=L_N.
\end{split}\end{equation*}

Similarly, one has
\begin{equation}\label{E:4.9c1}\begin{split}
I_6
\leq& C\big( |\nabla \overline{W}|_6|\nabla^2 W|_3+|\nabla^3 W|_2|\overline{W}|_\infty\big)|\nabla^2 \overline{W}|_2\\[2pt]
\leq& C\big( |\nabla^2 W|_3|\overline{W}|^2_{D^2}+|\nabla^3 W|_2|\overline{W}|^{\frac{1}{2}}_{D^1}|\overline{W}|^{\frac{3}{2}}_{D^2}\big)\\[2pt]
\leq& C\big((1+t)^{-\iota}|\overline{W}|^2_{D^2}+(1+t)^{-0.5-\iota}|\overline{W}|^{\frac{1}{2}}_{D^1}|\overline{W}|^{\frac{3}{2}}_{D^2}\big),\\
I_{7}+&I_9
\leq 
C|\nabla\widehat{u}|_\infty |\overline{W}|_{D^2}^2
\leq C(1+t)^{-1}|\overline{W}|_{D^2}^2,\\
I_{8}
\leq& C\big(|\nabla^3\widehat{u}|_2|\overline{W}|_\infty+|\nabla^2\widehat{u}|_3|\nabla\overline{W}|_6
\big)|\nabla^2\overline{W}|_{2}\\[2pt]
\leq& C\Big((1+t)^{-1}|\overline{W}|_{D^2}^2+(1+t)^{-2.5}|\overline{W}|^{\frac{1}{2}}_{D^1}|\overline{W}|^{\frac{3}{2}}_{D^2}\Big),\\
I_{10}
\leq& C\big(|\nabla\widehat{u}|_\infty|\nabla^2\overline{W}|_2+|\nabla^2\widehat{u}|_3|\nabla\overline{W}|_6\big)|\nabla^2\overline{W}|_2\\
\leq & C(1+t)^{-1}|\nabla^2\overline{W}|_2^2,\\
I_{11}
\leq& C\epsilon |\varphi^\epsilon|_\infty|\varphi^\epsilon\nabla^4 v^\epsilon|_2|\nabla^2 \overline{v}|_2
\leq C\epsilon(1+t)^{1.5-\iota}|\varphi^\epsilon\nabla^4 v^\epsilon|_2|\nabla^2 \overline{v}|_2,\\[2pt]
\leq&  C\epsilon^2(1+t)|\varphi^\epsilon\nabla^4 v^\epsilon|_2^2+C(1+t)^{2-2\iota}|\overline{W}|^2_{D^2},\\
I_{12}
\leq& C\epsilon\Big((|\varphi^\epsilon|_\infty|\nabla^2\varphi^\epsilon|_2+
|\nabla\varphi^\epsilon|_3|\nabla\varphi^\epsilon|_6)|\nabla^2\widehat{u}|_\infty\\
&+|\varphi^\epsilon|_\infty|\nabla\varphi^\epsilon|_\infty
|\nabla^3\widehat{u}|_2+|\varphi^\epsilon|_\infty^2|\nabla^4\widehat{u}|_2\Big)|\nabla^2\overline{v}|_2\\
&+C\epsilon|\nabla F_N|_\infty\big(|\varphi^\epsilon|_\infty|\nabla \varphi^\epsilon|_2|\nabla^2\widehat{u}|_\infty+|\varphi^\epsilon|_\infty^2|\nabla^3\widehat{u}|_2\big)|\nabla^2\overline{v}|_2\\
&+C\epsilon|\nabla^2 F_N|_\infty|\varphi^\epsilon|_\infty^2|\nabla^2\widehat{u}|_2|\nabla^2\overline{v}|_2\\
\leq& C\big(\epsilon^{\frac{3}{4}}(1+t)^{-0.5-2\iota}+N^{-1}(1+t)^{1.5-2\iota}\big)|\nabla^2\overline{v}|_2,\\
I_{13}
\leq& C\epsilon |\varphi^\epsilon|_\infty|\nabla \varphi^\epsilon|_\infty(|\nabla^3 v^\epsilon|_2+|\nabla^3\widehat{u}|_2)|\nabla^2 \overline{v}|_2\\
\leq & C\epsilon^{\frac{3}{4}}(1+t)^{1.5-3\iota}|\nabla^2 \overline{v}|_2,\\
I_{14}
\leq &C\epsilon\Big(|\varphi^\epsilon|_\infty
|\nabla^2\varphi^\epsilon|_3|\nabla^2v^\epsilon|_6+|\nabla \varphi^\epsilon|^2_6|\nabla^2 v^\epsilon|_6+|\varphi^\epsilon|_\infty|\nabla \varphi^\epsilon|_\infty|\nabla^3 v^\epsilon|_2\\
&+N^{-1}\big(
|\varphi^\epsilon|_\infty|\nabla\varphi^\epsilon|_3
|\nabla^2 v^\epsilon|_6+|\varphi^\epsilon|_\infty^2|\nabla^3v^\epsilon|_2+|\varphi^\epsilon|_\infty^2|\nabla^2 v^\epsilon|_2\big)\Big)
|\nabla^2 \overline{v}|_2\\
\leq& C\big(\epsilon^{\frac{3}{4}}(1+t)^{1.5-3\iota}+N^{-1}(1+t)^{3.5-3\iota}\big)|\nabla^2 \overline{v}|_2,\\
I_{15}
\leq&C\epsilon \Big(\big(|\varphi^\epsilon|_\infty |\nabla^3 \varphi^\epsilon|_2+
|\nabla\varphi^\epsilon|_3|\nabla^2\varphi^\epsilon|_6\big)(|\nabla v^\epsilon|_\infty+|\nabla\widehat{u}|_\infty)\\
 &+\big(|\varphi^\epsilon|_\infty|\nabla^2 \varphi^\epsilon|_3+
|\nabla\varphi^\epsilon|_\infty|\nabla\varphi^\epsilon|_3\big)|\nabla^2 v^\epsilon|_6\\
&+\big(|\varphi^\epsilon|_\infty|\nabla^2 \varphi^\epsilon|_2+
|\nabla\varphi^\epsilon|_3|\nabla\varphi^\epsilon|_6\big)|\nabla^2\widehat{u}|_\infty\Big)|\nabla^2\overline{v}|_2\\
&+CN^{-1}\epsilon \Big((|\varphi^\epsilon|_\infty|\nabla^2\varphi^\epsilon|_2+|\nabla\varphi^\epsilon|_3
|\nabla\varphi^\epsilon|_6)(|\nabla v^\epsilon|_\infty+|\nabla\widehat{u}|_\infty)\\[2pt]
&+|\varphi^\epsilon|_\infty(|\nabla\varphi^\epsilon|_3|\nabla^2 v^\epsilon|_6+|\nabla\varphi^\epsilon|_2|\nabla^2\widehat{u}|_\infty)\\
&+ \big(|\varphi^\epsilon|_\infty|\nabla\varphi^\epsilon|_2(|\nabla v^\epsilon|_\infty+|\nabla\widehat{u}|_\infty)\Big)|\nabla^2\overline{v}|_2\\
\leq& C\big(\epsilon^{\frac{1}{2}}(1+t)^{0.5-2\iota}+N^{-1}
(1+t)^{2.5-2\iota}\big)|\nabla^2\overline{v}|_2.
\end{split}\end{equation}

It follows from  \eqref{E:7.11}-\eqref{E:4.9c1} that 
\begin{equation}\label{E:7.15}\begin{split}
 \frac{d}{dt}| \overline{W}|_{D^2}^2\leq&C(1+t)^{\max\{-1,2-2\iota\}}|\overline{W}|_{D^2}^2+C\epsilon^2(1+t)|\varphi^\epsilon\nabla^4 v^\epsilon|^2_2\\
 &+C(1+t)^{1-4\iota}|\overline{W}|^2_{D^1}+C\epsilon(1+t)^{2-4\iota}+C(T)N^{-1}+L_N,
\end{split}\end{equation}
which, along with \eqref{E:7.3} and  Gronwall's inequality, implies  \eqref{E:7.10}.
\end{proof}

\subsection{The proof for Theorem \ref{T:7.1}}
\begin{proof}
First,  it is easy to see that 
\begin{equation}\label{appro}
\begin{split}
|\nabla^k\big((\rho^\epsilon_0-\rho_0)F_N\big)|_2 \rightarrow & |\nabla^k (\rho_0^\epsilon-\rho_0)|_2 \quad \text{for} \quad k=0,1,2, 3 \quad \text{as} \quad N\rightarrow \infty,\\
|\nabla^k \overline{W}|_2\rightarrow & |\nabla^k \overline{W}^\epsilon|_{2} \quad \text{for} \quad k=0,1,2, 3 \quad \text{as} \quad N\rightarrow \infty.
\end{split}
\end{equation}

Second, due to the uniform estimates \eqref{E:gz77xx}, one also has
\begin{equation}\label{appro1}
\begin{split}
J^*_N+I^*_N+L^*_N \rightarrow 0 \quad \text{as} \quad N\rightarrow \infty.
\end{split}
\end{equation}

Then, according to Lemmas \ref{L:7.1}-\ref{L:7.3}, \eqref{appro} and \eqref{appro1},  Theorem \ref{T:7.1} can be proved.
\end{proof}

\section{More regular initial data}
This section will be  devoted to showing how  to improve the convergence rate for the inviscid  limit process if we consider more regular initial data.

First, one has 
\begin{theorem}\label{tha-7}{\bf(Uniform energy estimates)}  Let \eqref{canshu} and 
any one of  $(P_1)$-$(P_4)$ hold.
If the   initial data $( \rho^\epsilon_0, u^\epsilon_0)$ satisfy the assumptions $(A^*_1)$-$(A^*_2)$ (in Remark 2.4),
then for any $T>0$,  there exists a   unique  regular solution $(\rho^\epsilon, u^\epsilon)$ in $[0,T]\times \mathbb{R}^3$  to the  Cauchy problem  (\ref{eq:1.1})-(\ref{10000}) with (\ref{initial})-(\ref{far}) satisfying
\begin{equation*}\begin{split}
&(\textrm{D1})\quad   \Big((\rho^\epsilon)^{\frac{\delta-1}{2}},(\rho^\epsilon)^{\frac{\gamma-1}{2}}\Big)\in C([0,T]; H^{s'}_{loc})\cap L^\infty([0,T]; H^4);\\[2pt]
& (\textrm{E1})\quad u^\epsilon-  \widehat{u}^\epsilon\in C([0,T]; H^{s'}_{loc})\cap L^\infty([0,T]; H^{4}),\quad  (\rho^\epsilon)^{\frac{\delta-1}{2}}\nabla^5 u^\epsilon \in L^2([0,T]; L^2),
\end{split}
\end{equation*}
for any $s'\in[3,4)$, and  the following uniform estimates:
\begin{equation}\label{uniformtime-7}\begin{split}
\sum_{k=0}^4 (1+t)^{2(k-2.5)}\Big(\big|\nabla^k(\rho^\epsilon)^{\frac{\gamma-1}{2}}(t,\cdot)\big|^2_2+\big|\nabla^k(u^\epsilon- \widehat{u}^\epsilon)(t,\cdot)\big|^2_2\Big)&\\
+\sum_{k=0}^3(1+t)^{2(k-3)}\big|\nabla^k(\rho^\epsilon)^{\frac{\delta-1}{2}}(t,\cdot)\big|^2_2+\epsilon\big|(\rho^\epsilon)^{\frac{\delta-1}{2}}(t,\cdot)\big|^2_{D^4}\leq &C_{05}(1+t)^{-\iota},\\
\epsilon \sum_{k=0}^4\int_0^t (1+s)^{2(k-2.5)}\big|(\rho^\epsilon)^{\frac{\delta-1}{2}}\nabla^{k+1} (u^\epsilon- \widehat{u}^\epsilon)(s,\cdot)\big|_2^2
\ \text{\rm d}s\leq &C_{05}
\end{split}
\end{equation}
for any $t\in [0,T]$ and some  positive constant $C_{05}=C_{05}(\gamma, \delta, \alpha,\beta,  A, \kappa, \rho^\epsilon_0, u^\epsilon_0)$ that is  independent of $\epsilon$.
Particularly,  when condition $(P_3)$ holds, the smallness assumption on $(\rho^\epsilon_0)^{\frac{\delta-1}{2}}$ could be removed.

\end{theorem}
The proof of the above theorem is completely  similar to that of  Theorem \ref{tha}. Here we omit it.

Second, for the inviscid flow, similarly, one also has
\begin{corollary} \label{serre-7} Let $\gamma>1$.
If the following initial hypothesis hold:
\begin{itemize}
\item[$(\rm H^*_1)$]  $u_0\in \Xi^*$, and  there exists a constant  $\kappa>0$ such that ,
$$
\text{Dist}\big(\text{Sp}( \nabla u_0(x)), \mathbb{R}_{-} \big)\geq \kappa \quad \text{for\  all}\quad  x\in \mathbb{R}^3;\\
$$
\item[$(\rm H^*_2)$] $\rho_0\geq 0$\  and  \ $\big\|\rho^{\frac{\gamma-1}{2}}_0\big\|_4 \leq D^*_0(\gamma,  A, \kappa, \|u_0\|_{\Xi^*})$
\end{itemize}
for some constant $D^*_0=D^*_0( \gamma, A,  \kappa, \|u_0\|_{\Xi^*})>0$,
then for any $T>0$,  there exists a   unique  regular solution $(\rho, u)$ in $[0,T]\times \mathbb{R}^3$  to the  Cauchy problem  \eqref{eq:1.1E} with (\ref{winitial})-(\ref{wfar}) satisfying
\begin{equation*}\begin{split}
&(\textrm{D2})\quad   \rho^{\frac{\gamma-1}{2}} \in C([0,T]; H^4),\quad  u-  \widehat{u}\in C([0,T]; H^{4}),\end{split}
\end{equation*}
and 
\begin{equation}\label{Euniformtime-7}\begin{split}
\sum_{k=0}^4 (1+t)^{2(k-2.5)}\Big(\big|\nabla^k \rho^{\frac{\gamma-1}{2}}(t,\cdot)\big|^2_2+\big|\nabla^k(u-\widehat{u})(t,\cdot)\big|^2_2\Big)\leq C_{06},
\end{split}
\end{equation}
for any $t\in [0,T]$ and  some constant $C_{06}=C_{06}(\gamma, A,  \kappa, \rho_0, u_0)>0$ that is independent of  the time $T$.

\end{corollary}
The proof of the above corollary is also completely similar to that of  Corollary \ref{serre}. Here we omit it.

At last, we can improve the convergence rate in the inviscid limit as follows:
\begin{theorem}\label{th3A-7}Let $T>0$ be an arbitrarily large time. Under the assumptions of Theorem {\rm\ref{tha-7}}, if we assume in addition that  $u^\epsilon_0=u_0$ is independent of $\epsilon$, and there exists one  function  $\rho_0(x)$ defined in $\mathbb{R}^3$  such that
\begin{equation}\label{initialrelation1-7}
\lim_{\epsilon\rightarrow 0}\Big|(\rho^\epsilon_0)^{\frac{\gamma-1}{2}}-\rho^{\frac{\gamma-1}{2}}_0\Big|_{2}=0,
\end{equation}
then there exist functions $(\rho,u)$ defined in $[0,T]\times \mathbb{R}^3$  satisfying:
\begin{equation}\label{jkkab1-7}
\sum_{k=0}^4 (1+t)^{2(k-2.5)}\Big(\big|\nabla^k\rho^{\frac{\gamma-1}{2}}(t,\cdot)\big|^2_2+\big|\nabla^k(u-\widehat{u})(t,\cdot)\big|^2_2\Big)  \leq C_{07}(1+t)^{-\iota},
\end{equation}
 for any  $t\in [0,T]$ and some   constant $C_{07}=C_{07}(\gamma, \delta, \alpha,\beta,  A, \kappa, \rho_0,u_0)>0$,  and
\begin{equation}\label{shou1A1-7}
\begin{split}
&\sup\limits_{0\leq t\leq T}\Big(\big\|(\rho^\epsilon)^{\frac{\gamma-1}{2}}(t,\cdot)-\rho^{\frac{\gamma-1}{2}}(t,\cdot)
\big\|_{2}+\big\| u^\epsilon(t,\cdot)-u(t,\cdot)\|_2 \Big)\\
\leq &C_{08}\Big(  \epsilon+\big\|(\rho^\epsilon_0)^{\frac{\gamma-1}{2}}-\rho^{\frac{\gamma-1}{2}}_0
\big\|_{2}\Big),\\
\end{split}
\end{equation}
\begin{equation}\label{shou1A2-7}
\begin{split}
&\sup\limits_{0\leq t\leq T}\Big(\big|(\rho^\epsilon)^{\frac{\gamma-1}{2}}(t,\cdot)-\rho^{\frac{\gamma-1}{2}}(t,\cdot)
\big|_{D^3}+\big| u^\epsilon(t,\cdot)-u(t,\cdot)|_{D^3} \Big)\\
\leq &C_{08}\Big(  \epsilon^{\frac{1}{2}}+\big\|(\rho^\epsilon_0)^{\frac{\gamma-1}{2}}-\rho^{\frac{\gamma-1}{2}}_0
\big\|_{3}\Big),
\end{split}
\end{equation}
for some   constant $C_{08}=C_{08}(\gamma, \delta, \alpha,\beta,  A, \kappa, D_0, \|u_0\|_{\Xi^*},T)>0$.
Furthermore,  $(\rho, u)$ is  the unique regular
solution       to the Cauchy problem  \eqref{eq:1.1E} with (\ref{winitial})-(\ref{wfar}) satisfying (D2) and 
$$(\rho,u)|_{t=0}=(\rho_0(x)\geq 0,  u_0(x)) \quad \text{for} \quad   x\in \mathbb{R}^3.$$

\end{theorem}

\begin{remark} 
Compared with Theorem \ref{th3A}, it is obvious  that   our current convergence rate is  more accurate than
that of \eqref{shou1A1} and \eqref{qiyu2} in the sense of $L^\infty$ or $H^2$ norm.
\end{remark}
Based on the proof for Theorem \ref{th2}, now we give a simple proof for Theorem \ref{th3A-7}.
\begin{proof}First, according to  \eqref{uniformtime-7}, one knows that the $L^{\infty}L^2$ estimates on 
$$
\nabla^{k_1}(\rho^\epsilon)^{\frac{\gamma-1}{2}}, \quad \nabla^{k_1}(u^\epsilon- \widehat{u}^\epsilon) \quad \text{and} \quad    \nabla^{k_2}(\rho^\epsilon)^{\frac{\delta-1}{2}}
$$
are independent of $\epsilon$, when $k_1\in [0,4]$ and $k_2\in [0,3]$.  

Second, we can come back to the proof for Lemma 6.3. Noticing that 
$$
\sum_{k=0}^3(1+t)^{2(k-3)}\big|\nabla^k(\rho^\epsilon)^{\frac{\delta-1}{2}}(t,\cdot)\big|^2_2\leq C_{05}(1+t)^{-\iota},
$$
then in \eqref{E:4.9c1}, actually one can obtain
\begin{equation*}\begin{split}
\sum_{i=11}^{15} I_i\leq  &C\epsilon^2(1+t)|\varphi^\epsilon\nabla^4 v^\epsilon|_2^2+C(1+t)^{2-2\iota}|\overline{W}|^2_{D^2}\\
&+C\big(\epsilon (1+t)^{1.5-2\iota}+N^{-1}
(1+t)^{2.5-2\iota}\big)|\nabla^2\overline{v}|_2.
\end{split}
\end{equation*}

At last,  we can prove \eqref{shou1A2-7}  by the completely same argument used in the proof for Lemma 6.3.
Thus, we have proved Theorem \ref{th3A-7}.
\end{proof}

\section*{Appendix A. Related global well-posedenss of multi-dimensional  systems}

This appendix will  show some known results on the   global well-posedenss of multi-dimensional systems for supporting  our previous  proof.

First,
let $\widehat{u}$ be the solution to the   problem (\ref{eq:approximation}) in    d-dimensional space. Then,  along the particle path $X(t;x_0)$ defined as
\begin{equation}\label{eq:1.5}
\frac{d}{d t}X(t;x_0)= \widehat{u}(t, X(t;x_0)), \quad X(0;  x_0)=x_0,
\end{equation}
$\widehat{u}$ is a constant in t: $\widehat{u}(t,X(t;x_0))=u_0(x_0)$ and
$$
\nabla  \widehat{u}(t,X(t;x_0))=\big(\mathbb{I}_d+t\nabla u_0(x_0)\big)^{-1}\nabla u_0(x_0).
$$

Based on this observation, one can have the following  decay estimates of $ \widehat{u}$, which play important roles in  establishing the global existence of the smooth solution to the  problem considered in this paper, and their proof  could be found in \cite{grassin} or  the  appendix of \cite{zz}.
\begin{proposition}\cite{grassin,zz}\label{p1} Let $m>1+\frac{d}{2}$. Assume that
$$\nabla u_0\in L^\infty( \mathbb{R}^d),\quad \nabla^2 u_0\in H^{m-1}( \mathbb{R}^d),$$
 and  there exists a constant  $\kappa>0$ such that for all $x\in \mathbb{R}^d$,
$$
\text{Dist}\big(\text{Sp}( \nabla u_0(x)), \mathbb{R}_{-} \big)\geq \kappa,
$$
then there exists a unique global  classical solution $ \widehat{u}$ to the  problem \eqref{eq:approximation}, which satisfies
\vspace{0.1cm}
\begin{itemize}
\item[(1)]$\nabla \widehat{u}(t, x)= \frac{1}{1+t}\mathbb{I}_d+\frac{1}{(1+t)^2}K(t,x),\quad \text{for \ all} \quad x\in \mathbb{R}^d,\quad t\geq 0$;\\[1pt]
\item[(2)]$
\|\nabla^l \widehat{u}(\cdot, t)\|_{L^2(\mathbb{R}^d)}\leq C_{0,l} (1+ t)^{\frac{d}{2}-(l+1)},\quad \text{for} \quad  2\leq l\leq m+1$;\\[1pt]
\item[(3)]$
\|\nabla^2 \widehat{u}(\cdot, t)\|_{L^\infty(\mathbb{R}^d)}\leq C_0(1+ t)^{-3}\|\nabla^2 u_0\|_{L^\infty(\mathbb{R}^d)}$,
\end{itemize}
where the matrix $K(t,x)=K_{ij}:  \mathbb{R}^+\times \mathbb{R}^d  \rightarrow M_d(\mathbb{R}^d)$  satisfies
$$\|K\|_{L^{\infty}(\mathbb{R}^+\times\mathbb{R}^d)}\leq C_0\big(1+\kappa^{-d}\|\nabla u_0\|^{d-1}_{L^\infty(\mathbb{R}^d)}\big).$$ Here, $C_{0}$ is a constant depending only on $m$, $d$, $\kappa$ and $u_0$, and $C_{0,l}$ are all constants depending on $C_0$ and $l$.
Moreover, if  $u_0(0)=0$,  then it holds  that
\begin{equation}\label{linear relation}
| \widehat{u}(t,x)|\leq |\nabla \widehat{u}|_\infty |x|,\quad \text{for\ \ any } \quad t\geq 0.
\end{equation}
\end{proposition}

Recently, Xin-Zhu \cite{zz} established the global existence of  the unique regular solution to the 3-D compressible isentropic  Navier-Stokes equation with degenerate viscosities:
\begin{proposition}\cite{zz}\label{thglobal} Let \eqref{canshu}
and any one of the   conditions $(P_1)$-$(P_4)$ (in Theorem \ref{tha})  hold. If the   initial data $( \rho^\epsilon_0, u^\epsilon_0)$ satisfy
\begin{itemize}

\item[$(\rm E_1)$] $u^\epsilon_0\in \Xi$, 
and  there exists a constant  $\kappa>0$  such that ,
$$
\text{Dist}\big(\text{Sp}( \nabla u^\epsilon_0(x)), \mathbb{R}_{-} \big)\geq \kappa \quad \text{for\  all}\quad  x\in \mathbb{R}^3;\\
$$

\item[$(\rm E_2)$] $\rho^\epsilon_0\geq 0$,  $(\rho^\epsilon_0)^{\frac{\gamma-1}{2}}\in H^3$,    $(\rho^\epsilon_0)^{\frac{\delta-1}{2}}\in H^3$, and
$$\big\|(\rho^\epsilon_0)^{\frac{\gamma-1}{2}}\big\|_3+ \big\|(\rho^\epsilon_0)^{\frac{\delta-1}{2}}\big\|_3 \leq E_0$$
 for some constant  $E_0=E_0(\gamma, \delta, \alpha,\beta,  A, \kappa, \|u^\epsilon_0\|_{\Xi})>0$,
\end{itemize}
then  for arbitrarily large time  $T>0$,  there exists a   unique  regular solution $(\rho^\epsilon, u^\epsilon)$ in $[0,T]\times \mathbb{R}^3$  to the  Cauchy
 problem   (\ref{eq:1.1})-(\ref{10000}) with (\ref{initial})-(\ref{far}).
\end{proposition}
The above conclusion  can be found in Theorem 1.4 and Remark 1.2 in \cite{zz}.

Finally, we show the global well-posedness theory of  the following nonlinear ODE.
\begin{proposition}\label{P:xinzhu1}
For the constants $b,C_i$ and $D_i(i=1,2)$ satisfying
$$
a>1,\quad D_1-(a-1)b<-1,\quad D_2<-1,\quad C_i\geq 0,\quad \text{for}\ \ i=1,2,
$$
there exists a constant $\Lambda$ such that there exists a global smooth solution to the following problem
\begin{equation}\label{E:xinzhu4}
\begin{cases}
\frac{dZ(t)}{dt}+\frac{b}{1+t}Z(t)=C_1(1+t)^{D_1}Z^\alpha(t)+C_2(1+t)^{D_2}Z(t),\\
\displaystyle
Z(x,0)=Z_0<\Lambda.
\end{cases}
\end{equation}
\end{proposition}
The proof of Proposition \ref{P:xinzhu1} can be found in \cite{zz}.

\section*{Appendix B. Useful lemmas}
In this appendix, we list some useful lemmas which were used frequently in the previous sections.
The first one is the  well-known Gagliardo-Nirenberg inequality.
\begin{lemma}\cite{oar}\label{lem2as}\
For $p\in [2,6]$, $q\in (1,\infty)$, and $r\in (3,\infty)$, there exists some generic constant $C> 0$ that may depend on $q$ and $r$ such that for
$$f\in H^1(\mathbb{R}^3),\quad \text{and} \quad  g\in L^q(\mathbb{R}^3)\cap D^{1,r}(\mathbb{R}^3),$$
one has
\begin{equation}\label{33}
\begin{split}
&|f|^p_p \leq C |f|^{(6-p)/2}_2 |\nabla f|^{(3p-6)/2}_2,\\[8pt]
&|g|_\infty\leq C |g|^{q(r-3)/(3r+q(r-3))}_q |\nabla g|^{3r/(3r+q(r-3))}_r.
\end{split}
\end{equation}
\end{lemma}
Some special cases of this inequality can be written as
\begin{equation}\label{ine}\begin{split}
|u|_6\leq C|u|_{D^1},\quad |u|_{\infty}\leq C\|\nabla u\|_{1}, \quad |u|_{\infty}\leq C\|u\|_{W^{1,r}}.
\end{split}
\end{equation}

The second lemma is some compactness results obtained via the Aubin-Lions Lemma.
\begin{lemma}\cite{jm}\label{aubin} Let $X_0\subset X\subset X_1$ be three Banach spaces.  Suppose that $X_0$ is compactly embedded in $X$ and $X$ is continuously embedded in $X_1$. Then the following statements hold.

\begin{enumerate}
\item[i)] If $J$ is bounded in $L^p([0,T];X_0)$ for $1\leq p < +\infty$, and $\frac{\partial J}{\partial t}$ is bounded in $L^1([0,T];X_1)$, then $J$ is relatively compact in $L^p([0,T];X)$;\\

\item[ii)] If $J$ is bounded in $L^\infty([0,T];X_0)$  and $\frac{\partial J}{\partial t}$ is bounded in $L^p([0,T];X_1)$ for $p>1$, then $J$ is relatively compact in $C([0,T];X)$.
\end{enumerate}
\end{lemma}

Next we give some Sobolev inequalities on the interpolation estimate, product estimate,  composite function estimate and so on in the following three lemmas.
\begin{lemma}\cite{amj}\label{gag111}
Let  $u\in H^s$, then for any $s'\in[0,s]$,  there exists  a constant $C_s$ only depending on $s$ such that
$$
\|u\|_{s'} \leq C_s \|u\|^{1-\frac{s'}{s}}_0 \|u\|^{\frac{s'}{s}}_s.
$$
\end{lemma}

\begin{lemma}\cite{oar} \label{gag112}
Let  $r\geq 0$,  $i\in[0,r]$, and $f\in L^\infty\cap H^r$.  Then $\nabla^i f\in L^{2r/i}$, and there exist some generic constant
$C_{i,r}>0$ such that
$$
|\nabla^i f|_{2r/i} \leq C_{i,r} |f|_\infty^{1-i/r}|\nabla^r f|^{i/r}_2.
$$
\end{lemma}

\begin{lemma}\cite{amj}\label{gag113}
Let  functions $u,\ v \in H^s$ and $s>\frac{3}{2}$, then  $u\cdot v \in H^s$,  and  there exists  a constant $C_s$ only depending on $s$ such that
$$
\|uv\|_{s} \leq C_s \|u\|_s \|v\|_s.
$$
\end{lemma}

\begin{lemma}\cite{amj}
\begin{enumerate}
\item For functions $f,\ g \in H^s \cap L^\infty$ and $|\nu|\leq s$,   there exists  a constant $C_s$ only depending on $s$ such that
\begin{equation}\label{liu01}
\begin{split}
\|\nabla^\nu (fg)\|_s\leq C_s(|f|_{\infty}|\nabla^s g|_2+|g|_{\infty}|\nabla^s f|_{2}).
\end{split}
\end{equation}

\item Assume that  $g(u)$ is a smooth vector-valued function on $G$, $u(x)$ is a continuous function with $u\in H^s \cap L^\infty$. Then for $s\geq 1$,   there exists  a constant $C_s$ only depending on $s$ such that
\begin{equation}\label{liu02}
\begin{split}
|\nabla^s g(u)|_2\leq C_s\Big \|\frac{\partial g}{\partial u }\Big\|_{s-1}|u|^{s-1}_{\infty}|\nabla^s u|_{2}.
\end{split}
\end{equation}
\end{enumerate}
\end{lemma}

The last one  is a useful tool to improve weak convergence to strong convergence.
\begin{lemma}\cite{amj}\label{zheng5}
If function sequence $\{w_n\}^\infty_{n=1}$ converges weakly in a Hilbert space $X$ to $w$, then $w_n$  converges strongly to $w$ in $X$ if and only if
$$
\|w\|_X \geq \lim \text{sup}_{n \rightarrow \infty} \|w_n\|_X.
$$
\end{lemma}

\bigskip

{\bf Acknowledgement:}
This research  was funded in part
by National Natural Science Foundation of China under grants   11571232  and 11831011. 
 The research of Shengguo Zhu was also  supported in part by
 the Royal Society-- Newton International Fellowships NF170015, and  Monash University--Robert Bartnik Visiting Fellowship.

\bigskip

{\bf Conflict of Interest:} The authors declare that they have no conflict of interest.

\end{document}